\documentclass[12pt,reqno]{amsart}
\usepackage{etex}
\usepackage{amsmath,amsthm,amsfonts,amssymb,amscd,flafter,pinlabel, booktabs}
\usepackage[mathscr]{eucal}
\usepackage{graphicx}
 \usepackage[all,knot,arc]{xy}
 \usepackage{array, makecell}

\usepackage[usenames,dvipsnames]{color}
\usepackage{graphicx}
 \usepackage{tabu}
 \usepackage{epstopdf}
 \usepackage[T1]{fontenc}
 \usepackage{scalefnt}

\usepackage{bbm}
\usepackage{mathtools}

\usepackage{pgf}

\usepackage{tikz}
\usetikzlibrary{arrows,automata}
\usetikzlibrary{matrix,arrows}
\tikzset{cdlabel/.style={above,sloped,%
    execute at begin node=$\scriptstyle,execute at end node=$}}
\tikzset{algarrow/.style={->, thick}}   
\tikzset{alb/.style={->, bend right=55, thick}}
\tikzset{arb/.style={->, bend left=25, thick}} 
\tikzset{al/.style={->, bend right=20, thick}}
\tikzset{ar/.style={->, bend left=20, thick}}
\tikzset{unst1/.style={->, dashed, bend right=30, thick}}
\tikzset{unst2/.style={->, dashed, bend left=30, thick}}
\tikzset{als/.style={->, bend right=15, thick}}
\tikzset{ars/.style={->, bend left=15, thick}}
\tikzset{blgarrow/.style={->, thick}}
\tikzset{clgarrow/.style={->, thick}}
\tikzset{tensoralgarrow/.style={double, double equal sign distance, -implies}}
\tikzset{tensorblgarrow/.style={double, double equal sign distance, -implies}}
\tikzset{tensorclgarrow/.style={double, double equal sign distance, -implies}}
\tikzset{tensorelgarrow/.style={double, double equal sign distance, -implies}}
\tikzset{modarrow/.style={->, dashed}}
\tikzset{Amodar/.style={->, dashed}}
\tikzset{Dmodar/.style={->, dashed}}
\tikzset{DAmodar/.style={->, dashed}}

\usepackage[latin1]{inputenc}

\usepackage[colorlinks=true,linkcolor=blue,citecolor=blue,urlcolor=blue]{hyperref}
\usepackage{xcolor}
\usepackage{footnote}
\usepackage{marginnote}

\tikzset{cdlabel/.style={above,sloped,
    execute at begin node=$\scriptstyle,execute at end node=$}}    
\tikzset{al/.style={->, bend right=45, thick}}
\tikzset{ar/.style={->, bend left=45, thick}}

\normalfont\upshape

\usepackage[lmargin=1in,rmargin=1in,tmargin=1in,bmargin=1in]{geometry}

\newtheorem{theorem}{Theorem}[subsection]
\newtheorem{corollary}[theorem]{Corollary}

\newtheorem{proposition}[theorem]{Proposition}
\newtheorem{lemma}[theorem]{Lemma}

\newtheorem*{remark*}{Remark}
\newtheorem*{acknowledgements*}{Acknowledgements}

\usepackage{enumitem}
\newlist{casesp}{enumerate}{3} 
\setlist[casesp]{align=left, 
                 listparindent=\parindent, 
                 parsep=\parskip, 
                 font=\normalfont\bfseries, 
                 leftmargin=0pt, 
                 labelwidth=0pt, 
                 itemindent=.4em,labelsep=.4em, 
                 partopsep=0pt, 
                 }
\setlist[casesp,1]{label=Case~\arabic*:,ref=\arabic*}
\setlist[casesp,2]{label=Case~\thecasespi.\arabic*:,ref=\thecasespi.\arabic*}
\setlist[casesp,3]{label=Case~\thecasespii.\alph*:,ref=\thecasespii.\alph*}

\newcommand{\Z}{\ensuremath{\mathbb{Z}}}

\newcommand\F{\mathbb F}
\newcommand\M{\mathcal M}

\newcommand\thi{\mathrm{th}}
\newcommand\thiq{\mathrm{th}(Q_n(K))}
\newcommand\gq{g(Q_n(K))}

\newcommand\dbox{\bdy^{\boxtimes}}

\newcommand{\bdy}{\partial}

\newcommand{\ha}{\langle h_A \rangle }
\newcommand{\hd}{\langle h_D \rangle }

\newcommand{\gr}{\mathrm{gr}}

\newcommand{\cf}{\mathit{CF}}
\newcommand{\hf}{\mathit{HF}}
\newcommand{\cfd}{\mathit{CFD}}
\newcommand{\cfa}{\mathit{CFA}}

\newcommand{\cfk}{\mathit{CFK}}
\newcommand{\hfk}{\mathit{HFK}}

\newcommand{\cfhat}{\widehat{\cf}}
\newcommand{\hfhat}{\widehat{\hf}}
\newcommand{\cfdhat}{\widehat{\cfd}}
\newcommand{\cfahat}{\widehat{\cfa}}

\newcommand{\cfkhat}{\widehat{\cfk}}
\newcommand{\hfkhat}{\widehat{\hfk}}

\newcommand{\cfkm}{\cfk^-}
\newcommand{\hfkm}{\hfk^-}
\newcommand{\cfam}{\cfa^-}

\begin{document}

\title[{Twisted Mazur Pattern Satellite Knots \& Bordered Floer Theory}]{Twisted Mazur Pattern Satellite Knots \& Bordered Floer Theory}

\author{Ina Petkova}
\address {Department of Mathematics, Dartmouth College\\ Hanover, NH 03755}
\email {ina.petkova@dartmouth.edu}
\urladdr{\href{http://math.dartmouth.edu/~ina}{http://math.dartmouth.edu/~ina}}
\author[Biji Wong]{Biji Wong}
\address {CIRGET, Universit\'{e} du Qu\'{e}bec \`{a} Montr\'{e}al\\ Montr\'{e}al, Qu\'{e}bec H3C 3P8}
\email {biji.wong@cirget.ca}
\urladdr{\href{https://sites.google.com/view/cirget-bijiwong/home}{https://sites.google.com/view/cirget-bijiwong/home}}



\begin{abstract}
We use bordered Floer theory to study properties of twisted Mazur pattern satellite knots $Q_{n}(K)$. We prove that $Q_n(K)$ is not Floer homologically thin, with two exceptions. We calculate the 3-genus of $Q_{n}(K)$ in terms of the twisting parameter $n$ and the 3-genus of the companion $K$, and we determine when $Q_n(K)$ is fibered. As an application to our results on Floer thickness and 3-genus, we verify the Cosmetic Surgery Conjecture for many of these satellite knots.
\end{abstract}

\maketitle





\section{Introduction}\label{sec:intro}

In its simplest form, knot Floer homology, introduced by Ozsv\'{a}th--Szab\'{o} in \cite{hfk} and Rasmussen in \cite{jrth}, assigns to a knot $K \subset S^3$ an abelian group $\hfkhat(K)$ that is endowed with two $\mathbb{Z}$-gradings $M$ and $A$. We call $M$ the Maslov grading and $A$ the Alexander grading, and we denote their difference $M-A$ by $\delta$. Knot Floer homology has proven quite useful for studying knots in $S^3$. For example, it detects the 3-genus \cite{hfkg} and fiberedness \cite{pgf, nfibered}, and has a lot to say about knot concordance \cite{tau, hepsilon, ossupsilon}.

A knot $K \subset S^3$ is said to be \textit{knot Floer homologically thin} (\textit{$\delta$-thin} for short), if its knot Floer homology $\hfkhat(K)$ takes a particularly simple form: all of its generators have the same $\delta$-grading. The class of $\delta$-thin knots includes alternating knots \cite{ozalternating}, quasi-alternating knots \cite{quasi}, and some non-quasi-alternating knots \cite{greene}. Recently, cable knots with nontrivial companions were shown to not be $\delta$-thin \cite{dey}. It is natural to conjecture whether this is true for all satellite knots. Recall the satellite construction: Every framing $n \in \mathbb{Z}$ of a knot $K \subset S^3$ gives rise to an embedding of $S^1 \times D^2$ in $S^3$ as a tubular neighborhood of $K$, which is unique up to isotopy. We define the $n$-twisted \textit{satellite knot} $P_n(K)$ of an $n$-framed \textit{companion} knot $K$ with oriented \textit{pattern} knot $P \subset S^1 \times D^2$ to be the image of $P$ under this embedding. Once we fix a generator of $H_1(S^1 \times D^2; \mathbb{Z})$, the \textit{winding number} of $P$ is the integer $w$ for which $P$ represents $w$ times the generator.

In recent years, satellite knots with winding number $1$ have been instrumental in producing exotic structures on smooth 4-manifolds, see for example \cite{yasui, hmp}. One of the more well-known winding number 1 patterns is the Mazur knot $Q$ in Figure \ref{fig:mazurknot}. Mazur \cite{mazur} used it to construct the first example of a contractible 4-manifold whose boundary is an integral homology sphere not equal to $S^3$. Levine in \cite{levine-satellite} and Feller--Park--Ray in \cite{fpr} used $0$-twisted Mazur pattern satellite knots to understand the structure of the smooth knot concordance group.  Recently, Chen \cite{wchen} using bordered Floer homology studied a class of satellite knots that encompasses $0$-twisted Mazur pattern satellite knots, and recovered some of the $\tau$ computations in \cite{levine-satellite} for $0$-twisted Mazur pattern satellite knots.

\begin{figure}[h]
  \centering
  \includegraphics[scale=.35]{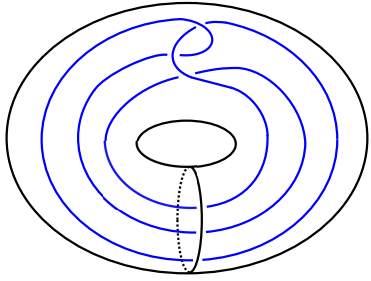}
  \caption{The Mazur pattern knot $Q$ in the solid torus}
 \label{fig:mazurknot}
\end{figure}

In this paper, we use bordered Floer homology to study some $3$-dimensional properties of arbitrarily twisted Mazur pattern satellite knots $Q_n(K)$. We show that for all but two satellites, $Q_n(K)$ is not $\delta$-thin:

\begin{theorem}\label{theorem_delta_thick}
$Q_n(K)$ is $\delta$-thick for all knots $K \subset S^3$ and integers $n$, except when $Q_n(K)$ is the trivial satellite $Q_{0}(U)$ or the $5_2$ satellite $Q_{-1}(U)$, which are $\delta$-thin.
\end{theorem}

Since quasi-alternating knots are $\delta$-thin, Theorem~\ref{theorem_delta_thick} implies the following.

\begin{corollary}
$Q_n(K)$ is quasi-alternating if and only if $Q_n(K)$ is the trivial satellite $Q_{0}(U)$ or the $5_2$ satellite $Q_{-1}(U)$.
\end{corollary}

Given any knot $K \subset S^3$, the \textit{$\delta$-thickness} of $K$, denoted  $\mathrm{th}(K)$, is defined as the difference between the maximum and minimum $\delta$-gradings in $\hfkhat(K)$ \cite{quasi}. We show that the $\delta$-thickness of $Q_n(K)$ increases without bound as we increase the number of twists:

\begin{theorem}\label{thm:th-unbounded}
For any knot $K \subset S^3$, $\lim_{n \rightarrow \pm \infty} \mathrm{th}(Q_n(K)) = \infty$.
\end{theorem}

We also show that the analogue of Theorem~\ref{thm:th-unbounded} does not hold for all patterns.

\begin{theorem}\label{thm:th-bounded}
There exist nontrivial patterns $P$ and nontrivial companions $K$ such that as ${n \rightarrow \pm \infty}$,  $\thi(P_n(K))$ is bounded by a constant that only depends on $P$ and $K$.
\end{theorem}

In addition to the above results, in Section~\ref{sec:csc} we explicitly compute $\thiq$, for $K$ a $\delta$-thin knot, or a certain type of L-space knot.

By a classical theorem of Schubert \cite{schubert}, the 3-genus $g(P_0(K))$ of a $0$-twisted satellite knot $P_0(K)$, with nontrivial companion $K \subset S^3$ and pattern $P \subset S^1 \times D^2$, can be expressed in terms of the 3-genus $g(K)$ of $K$, the winding number $w$ of $P$, and a geometrically defined number $g(P)$ that depends only on $P$:
\begin{equation*}
g(P_0(K)) = |w| g(K) + g(P).
\end{equation*}
We give an explicit formula for the 3-genus $g(Q_n(K))$ of an arbitrarily twisted Mazur pattern satellite $Q_n(K)$, in terms of the 3-genus $g(K)$ of the companion $K$ and the twisting $n$. Our result includes the case when the companion $K$ is trivial.

\begin{theorem}\label{thm:g3}
For any nontrivial knot $K \subset S^3$,
\begin{equation*}
g(Q_n(K)) = 
\begin{cases}
                                   g(K) - n & \text{if $n\leq -1$} \\
                                   g(K) + n + 1 & \text{if $n\geq 0$}.                                 
\end{cases}
\end{equation*}
When $K$ is the unknot, 
\begin{equation*}
g(Q_n(K)) = 
\begin{cases}
                                   -n & \text{if $n\leq 0$} \\
                                   n+1 & \text{if $n\geq 1$}.                                  
\end{cases}
\end{equation*}
\end{theorem}

 We remark that there is a 4-dimensional analogue of Theorem~\ref{thm:g3} due to Cochran--Ray in \cite{cr}. They showed that for certain companion knots $K$, the 4-genus $g_4(Q_n(K))$ of $Q_n(K)$ depends only on the $4$-genus $g_4(K)$ of the companion $K$, and not on the framing $n$.

We also fully determine when $Q_n(K)$ is fibered. By a theorem of Hirasawa--Murasugi--Silver \cite{hms}, $0$-twisted satellite knots $P_0(K)$ with nontrivial companions $K$ are fibered if and only if $K$ is fibered and $P$ is fibered in $S^1 \times D^2$. We show the following:

\begin{theorem}\label{thm:fib}
If $K$ is nontrivial, then $Q_n(K)$ is fibered if and only if $K$ is fibered and $n \neq -1,0$. If $K$ is trivial, then $Q_n(K)$ is fibered if and only if $n\neq -1$.
\end{theorem}

Lastly, we consider a question about surgeries on satellite knots. Given a knot $K \subset S^3$, two surgeries $S^3_r(K)$ and $S^3_{r'}(K)$, with $r \neq r'$, are said to be \textit{truly cosmetic} if $S^3_r(K)$ and $S^3_{r'}(K)$ are homeomorphic as oriented manifolds. The Cosmetic Surgery Conjecture predicts that there are no truly cosmetic surgeries on nontrivial knots in $S^3$ \cite{gordon}. The conjecture has been verified for several classes of knots, including genus 1 knots \cite{wang}, nontrivial cables \cite{tao1}, knots with genus at least 3 and $\delta$-thickness at most 5 \cite{hanselman}, and most recently composite knots \cite{tao2}, 3-braids \cite{3braids}, and pretzel knots \cite{cosmeticpretzel}. One might ask whether Mazur pattern satellite knots also satisfy the conjecture. We give the following partial answer.

\begin{theorem}\label{thm:csc}
Suppose $K$ is an L-space knot or a $\delta$-thin knot. 
\begin{itemize}
\item[-] If $K$ is an L-space knot and its Alexander polynomial 
\begin{equation*} 
\Delta_K(t) = (-1)^{k+1} + \sum\limits_{j=0}^k (-1)^j (t^{r_j} + t^{-r_j} )
\end{equation*}
satisfies the property that
\begin{equation*}
r_1-r_2 \geq r_2 - r_3 \geq \ldots \geq r_{k-1}-r_k \geq r_k,
\end{equation*}
then all nontrivial satellites $Q_n(K)$ satisfy the Cosmetic Surgery Conjecture.

\item[-] If $K$ is a $\delta$-thin knot, then all nontrivial satellites $Q_n(K)$ satisfy the Cosmetic Surgery Conjecture, unless one of the following holds:
\begin{itemize}[leftmargin=*]
\item[-] $\{r, r'\} = \{\pm 2\}$, $n=-1$, and $\Delta_K(t)=2t - 5+2t^{-1}$
\item[-] $\{r, r'\} = \{\pm 1\}$, $n=-1$,  and 
\begin{equation*}
\Delta_K(t)= 
\begin{cases}
                                   2t - 5+2t^{-1}, & \hspace{2.7cm} \textrm{or} \\
                                   bt^2 - (4b+2)t + (6b+5) - (4b+2)t^{-1} + bt^{-2} & \quad \textrm{with } b\geq 1, \quad \textrm{ or}\\   
                                   bt^2 - (4b-2)t + (6b-5) - (4b-2)t^{-1} + bt^{-2} & \quad \textrm{with } b\geq 2, \quad \textrm{ or}\\
(b+1)t^2 - (4b+6)t + (6b+11) - (4b+6)t^{-1} + (b+1)t^{-2}  & \quad \textrm{with }  b\geq 0,                              
\end{cases}
\end{equation*}
\item[-] $\{r, r'\} = \{\pm 1\}$, $n=0$, and 
\begin{equation*}
\Delta_K(t)= 
\begin{cases}
                    bt^2 - 4bt + (6b-1) - 4bt^{-1} + bt^{-2} & \quad \textrm{with }  b\geq 1 \quad \textrm{ or}\\ 
bt^2 - 4bt + (6b+1) - 4bt^{-1} + bt^{-2} & \quad \textrm{with }  b\geq 1,     \quad \textrm{and } \tau(K)=-1.      
\end{cases}
\end{equation*}
\end{itemize}
\end{itemize}
\end{theorem}

\subsection*{Organization}
We review the necessary bordered Floer homology background in Section~\ref{sec:prelim}. In Section~\ref{sec:tensor}, we use bordered Floer homology to study relevant properties of the knot Floer homology of $Q_n(K)$. In Section \ref{sec:deltathick}, we prove Theorems \ref{theorem_delta_thick}, \ref{thm:th-unbounded}, and \ref{thm:th-bounded}. In Section~\ref{sec:gf}, we prove Theorems~\ref{thm:g3} and ~\ref{thm:fib}. In Section~\ref{sec:csc}, we prove Theorem~\ref{thm:csc}.

\subsection*{Acknowledgments}
We thank Adam Levine, Tye Lidman, Brendan Owens, and Danny Ruberman for helpful conversations. We also thank the referee for helpful comments. The project began in the summer of 2019 when I.P. was a visitor at CIRGET; we thank CIRGET for its hospitality. I.P. received support from NSF Grant DMS-1711100.



\section{Preliminaries on bordered Floer theory}\label{sec:prelim}

Bordered Floer homology is an extension of Heegaard Floer homology to manifolds with  boundary \cite{bfh2}. To a parametrized surface $F$, one associates a differential algebra $A(F)$, and to a manifold $Y$ whose boundary is identified with $F$, one associates a right $\mathcal A_{\infty}$-module $\cfahat(Y)$ over $A(F)$, or a left type $D$ module $\cfdhat(Y)$ over $A(F)$.  These modules are invariants of the manifolds up to homotopy equivalence, and $\cfahat(Y_1) \boxtimes\cfdhat(Y_2)\simeq \cfhat(Y_1\cup Y_2)$. Another variant of these structures is associated to knots in bordered 3-manifolds, and recovers $\hfkhat$ or $\hfkm$ after gluing. To define these structures, one uses bordered Heegaard diagrams.

The algebra is graded by a certain nonabelian group $G$. Domains on a bordered Heegaard diagram are graded by $G$ as well. Then,  a right (resp.~left) module associated to a Heegaard diagram is graded by the space of right (resp.~left) cosets in $G$ of the subgroup of gradings of periodic domains. The tensor product is then graded by double cosets in $G$, from where one could extract the usual Heegaard Floer grading.

Below, we recall relevant definitions in the case when the boundary $F$ is a torus. For more details, see \cite{bfh2}.

The algebra $\mathcal A$ associated to the torus is generated over $\F_2$ by two idempotents denoted $\iota_0$ and $\iota_1$, and six nontrivial elements denoted $\rho_1$, $\rho_2$, $\rho_3$, $\rho_{12}$, $\rho_{23}$, and $\rho_{123}$. The differential is zero,  the nonzero products are
\[\rho_1\rho_2 = \rho_{12}\qquad \rho_2\rho_3=\rho_{23} \qquad \rho_1\rho_{23} = \rho_{123} \qquad \rho_{12}\rho_3 = \rho_{123}\]
and the compatibility with the idempotents is given by 
\begin{align*}
\rho_1 &= \iota_0 \rho_1\iota_1 \quad &\rho_2 &= \iota_1 \rho_2\iota_0 \quad &\rho_3 &= \iota_0 \rho_3\iota_1\\
\rho_{12} &= \iota_0 \rho_{12}\iota_0 \quad &\rho_{23} &= \iota_1 \rho_{23}\iota_1 \quad &\rho_{123} &= \iota_0 \rho_{123}\iota_1.
\end{align*}

Let $X_{K,n}$ be the $n$-framed knot complement $X_K = S^3\setminus \mathrm{nbhd}(K)$. One can compute the left type $D$ structure $\cfdhat(X_{K,n})$ from $\cfkm(K)$ as follows. 

There exist a pair of bases  $\widetilde{\boldsymbol{\eta}} = \{\widetilde\eta_0, \ldots, \widetilde\eta_{2k}\}$ and $\widetilde{\boldsymbol{\xi}} = \{\widetilde\xi_0, \ldots, \widetilde\xi_{2k}\}$ for $\cfkm(K)$ (over $\F_2[U]$) that are horizontally simplified and vertically simplified, respectively, indexed so that there is a horizontal arrow of length $l_i\geq 1$ from $\widetilde\eta_{2i-1}$ to $\widetilde\eta_{2i}$ and a vertical arrow of length $k_i\geq 1$ from $\widetilde\xi_{2i-1}$ to $\widetilde\xi_{2i}$. There are corresponding bases $\boldsymbol{\eta} = \{\eta_0, \ldots, \eta_{2k}\}$ and $\boldsymbol{\xi} = \{\xi_0, \ldots, \xi_{2k}\}$ for $\iota_0\cfdhat(X_{K,n})$, such that if $\widetilde\xi_p = \sum_{i=0}^{2k}a_{ip}\widetilde\eta_i$ and $\widetilde\eta_p = \sum_{i=0}^{2k}b_{ip}\widetilde\xi_i$, then $\xi_p = \sum_{i=0}^{2k}a_{ip}|_{U=0}\eta_i$ and $\eta_p = \sum_{i=0}^{2k}b_{ip}|_{U=0}\xi_i$. The summand $\iota_1\cfdhat(X_{K,n})$ has basis 
\[\bigcup_{i=1}^k\{\kappa_1^i, \ldots, \kappa_{k_i}^i\}\cup\bigcup_{i=1}^k\{\lambda_1^i, \ldots, \lambda_{l_i}^i\}\cup\{\mu_1, \ldots, \mu_{|2\tau(K)-n|}\}.\]
For each vertical arrow $\widetilde\xi_{2i-1} \rightarrow \widetilde\xi_{2i}$, there are corresponding coefficient maps
\begin{displaymath}
\xymatrix{ 
\xi_{2i-1} \ar[r]^{D_{1}}&\kappa_1^i  &  \ar[l]_{D_{23}} \dots & \ar[l]_{D_{23}} \kappa_{k_i}^i & \ar[l]_{D_{123}}\xi_{2i},
}
\end{displaymath}
and for each horizontal arrow $\widetilde\eta_{2i-1} \rightarrow \widetilde\eta_{2i}$, there are corresponding coefficient maps
\begin{displaymath}
\xymatrix{ 
\eta_{2i-1} \ar[r]^{D_3}&\lambda_1^i   \ar[r]^{D_{23}}& \dots\ar[r]^{D_{23}} & \lambda_{l_i}^i \ar[r]^{D_{2}}& \eta_{2i}.
}
\end{displaymath}
Depending on the framing $n$, there are additional coefficient maps
\begin{displaymath}
\xymatrix@R-14pt{ 
\xi_0  \ar[r]^{D_{12}} & \eta_0 & & & & &  \text{if }  n = 2\tau, \hspace{70pt}\\
\xi_0 \ar[r]^{D_1}&\mu_1  &  \mu_2 \ar[l]_{D_{23}}& \dots\ar[l]_{D_{23}} & \mu_m \ar[l]_{D_{23}}& \ar[l]_{D_3} \eta_0&  \text{if } n < 2\tau, \quad  m=2\tau-n,\\
\xi_0  \ar[r]^{D_{123}}&\mu_1 \ar[r]^{D_{23}} &  \mu_2 \ar[r]^{D_{23}}& \dots\ar[r]^{D_{23}}& \mu_{m} \ar[r]^{D_2}&   \eta_0 &\text{if } n > 2\tau, \quad  m=n-2\tau.
}
\end{displaymath}
We refer to the above chains of coefficient maps as the \emph{vertical chains}, the \emph{horizontal chains}, and the \emph{unstable chain}.

Given a doubly-pointed bordered Heegaard diagram $\mathcal H$ for a knot $Q$ in a solid torus $V$, one can associate right type $A$ structures $\cfam(\mathcal H)$ and $\cfahat(\mathcal H)$. See Section \ref{sec:mazur} for the  specific $\cfam(\mathcal H)$ and $\cfahat(\mathcal H)$ that we will be interested in.

Given any right type $A$ structure $M$ and any left type $D$ structure $N$ over the same algebra, with at least one of $M$ or $N$ bounded (an algebraic condition which our module $\cfam(\mathcal H)$ from Section~\ref{sec:mazur} satisfies), their \emph{box tensor product} is the chain complex  $M\boxtimes N$ defined as follows. As an $\F_2$ vector space, $M\boxtimes N$ is just $\M\otimes_{\mathcal I}N$. 
 The differential  $\bdy^{\boxtimes}(x_1\boxtimes y_1)$ has $x_2\boxtimes y_2$ in the image whenever there is a sequence of coefficient maps $D_{I_1}, \ldots, D_{I_n}$ from $y_1$ to $y_2$ and a multiplication map $m_{n+1}(x_1, \rho_{I_1}, \ldots, \rho_{I_n})$ with $x_2$ in the image, both indexed the same way. Further,  $\bdy^{\boxtimes}(x_1\boxtimes y)$ has $x_2\boxtimes y$ in the image  whenever $x_2$ is in the image of $m_1(x_2)$, and  $\bdy^{\boxtimes}(x\boxtimes y_1)$ has $x\boxtimes y_2$ in the image whenever there is a coefficient map with no label from $y_1$ to $y_2$. See  \cite[Definition 2.26 and Equation (2.29)]{bfh2}. 
 
When the type $A$ structure is $\cfam(\mathcal H)$ or $\cfahat(\mathcal H)$, and the type $D$ structure is $\cfdhat(X_{K,n})$, we have homotopy equivalences $\cfam(\mathcal H)\boxtimes \cfdhat(X_{K,n})\simeq g\cfkm(Q_n(K))$ and $\cfahat(\mathcal H)\boxtimes \cfdhat(X_{K,n})\simeq \widehat{\mathit{gCFK}}(Q_n(K))$.
 
We now briefly recall gradings on these algebraic objects. The algebra  $\mathcal A$ is graded by a group $G$ given by quadruples $(a; b,c; d)$ with $a,b,c\in \frac 1 2 \Z$, $d\in \Z$, and $b+c\in \Z$ and group law
\[(a_1; b_1, c_1; d_1)\cdot (a_2; b_2, c_2; d_2) = \left(a_1+a_2 + \begin{vmatrix} 
b_1 & c_1 \\
b_2 & c_2 
\end{vmatrix}; b_1+b_2, c_1+c_2; d_1+d_2\right).\]
The grading function is then defined by
\begin{align*}
\gr(\rho_1) & = (\textstyle -\frac{1}{2}; \frac{1}{2}, -\frac{1}{2}; 0)\\
\gr(\rho_2) & = (\textstyle -\frac{1}{2}; \frac{1}{2}, \frac{1}{2}; 0)\\
\gr(\rho_3) & = (\textstyle -\frac{1}{2}; -\frac{1}{2}, \frac{1}{2}; 0).
\end{align*}
along with the rule that for homogeneous algebra elements $a,b$, we have $\gr(ab)= \gr(a)\gr(b)$.

The type $D$ module $\cfdhat(X_{K,n})$ is graded by the coset space $G/\hd$, where $h_D =(-\frac{n}{2}-\frac{1}{2}; -1, -n; 0)$. A homogeneous generator $s$ of $\iota_0\cfdhat(X_{K,n})$ has grading 
\begin{equation}\label{eqn:homogencfd}
\gr(s) = \textstyle \left(M(\tilde s)-\frac 3 2 A(
\tilde s); 0, -A(\tilde s);0\right) \in G/\hd,
\end{equation}
where $M(\tilde s)$ and $A(\tilde s)$ are the Maslov grading and Alexander filtration of the corresponding generator $\tilde s$ in $\cfkm(K)$, respectively. In particular, we recall that 
\begin{equation}\label{eqn:gradings}
A(\widetilde\xi_0) =\tau(K) \qquad M(\widetilde\xi_0) = 0 \qquad A(\widetilde\eta_0) = -\tau(K)\qquad M(\widetilde\eta_0) = -2\tau(K).
\end{equation}

 If $D_I$ is a coefficient map from $x$ to $y$ then the gradings of $x$ and $y$ are related by
\begin{equation}\label{eqn:cfd}
\gr(y) = \lambda^{-1}\gr(\rho_I)^{-1}\gr(x)\in G/\hd,
\end{equation}
where $\lambda = (1; 0, 0; 0)$. 

Given a doubly-pointed bordered Heegaard diagram $\mathcal H$ for a knot $Q$ in a solid torus $V$, the module $\cfam(\mathcal H)$ is graded by the coset space $\ha\backslash G$, where $\ha$ is the subgroup of gradings of periodic domains. This subgroup depends on the knot $Q$; for the Mazur pattern, we find a generator $h_A$ in Section~\ref{sec:mazur}. For a multiplication map $m_{l+1} (x, \rho_{I_1}, \ldots, \rho_{I_l}) = U^i y$ we have the formula
\begin{equation}\label{eqn:cfa}
\gr(y) = \gr (x)\lambda^{l-1}\gr(\rho_{I_1})\cdots \gr(\rho_{I_l})(0; 0,0; i)\in \ha\backslash G.
\end{equation}
When the underlying manifold is the solid torus $V$, this is sufficient to obtain a relative grading of all generators. 

The tensor product $\cfam(\mathcal H)\boxtimes \cfdhat(X_{K,n})\simeq g\cfkm(Q_n(K))$ is graded by the double-coset space $\ha\backslash G/\hd$ via $\gr(x\boxtimes y) = \gr(x)\gr(y)$. The double-coset space $\ha\backslash G/\hd$ is isomorphic to $\Z\oplus \Z$, and for a homogeneous element $x\boxtimes y$, there is a unique grading representative of the form $(a; 0,0; d)$ with $a,d\in \Z$. Note that $a$ agrees with the absolute $z$-normalized grading $N$ of $x\boxtimes y$, and up to an overall translation $d$ agrees with the Alexander grading $A$ of $x\boxtimes y$.



\section{The complex $\cfahat(\mathcal H)\boxtimes \cfdhat(X_{K,n})\simeq \cfkhat(Q_n(K))$}\label{sec:tensor}

In this section, we work out general grading formulas for the generators of  $\cfahat(\mathcal H)\boxtimes \cfdhat(X_{K,n}) \simeq \cfkhat(Q_n(K))$, where $\mathcal{H}$ is a doubly-pointed bordered Heegaard diagram for the Mazur pattern in the solid torus. We also make some useful observations about the differential on this  complex. 

\subsection{$\cfahat$ of the Mazur pattern in the solid torus}\label{sec:mazur}
Let $V$ denote the solid torus $S^1 \times D^2$, and let $Q$ denote the Mazur pattern in $V$. Figure~\ref{fig:mazur-hd} is a doubly-pointed bordered Heegaard diagram $\mathcal{H}$ for $(V, Q)$, see also \cite[Figure~9]{levine-satellite}.
\begin{figure}[h]
  \centering
  \labellist
    \pinlabel  $z$ at 9 71
    \pinlabel  $w$ at 26 4 
    \pinlabel  $3$ at 4 4
    \pinlabel  $2$ at 76 4
    \pinlabel  $1$ at 76 76
    \pinlabel  $x_0$ at -4 20
    \pinlabel  $y_2$ at -4 30
    \pinlabel  $y_4$ at -4 40
    \pinlabel  $x_4$ at -4 50
    \pinlabel  $x_2$ at -4 60
    \pinlabel  $x_5$ at 16 -4
    \pinlabel  $x_6$ at 23 -4
    \pinlabel  $y_6$ at 30 -4
    \pinlabel  $y_5$ at 37 -4
    \pinlabel  $y_1$ at 44 -4
    \pinlabel  $y_3$ at 50 -4
    \pinlabel  $x_3$ at 57 -4
    \pinlabel  $x_1$ at 65 -4
  \endlabellist
  \includegraphics[scale=2]{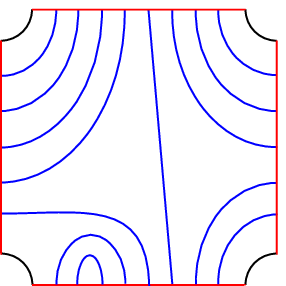}
  \caption{A bordered Heegaard diagram $\mathcal{H}$ for the pair $(V,Q)$.}
  \label{fig:mazur-hd}
\end{figure}

Over $\F_2[U]$, the type $A$ structure $\mathit{CFA}^{-}(\mathcal{H})$ is generated by $x_0$, $x_2$, $x_4$, $y_2$,  $y_4$ in idempotent $\iota_0$, and by $x_1$, $x_3$, $x_5$, $x_6$, $y_1$, $y_3$, $y_5$,  $y_6$ in idempotent $\iota_1$. The multiplication maps are encoded by the  labeled edges in Figure~\ref{fig:cfa}: an arrow from $v_1$ to $v_2$ with label $U^ta_1\cdots a_n$ describes the multiplication map $m_{n+1}(v_1, a_1, \ldots, a_n) = U^tv_2$, while an arrow from $v_1$ to $v_2$ with label $U^ta_1\cdots a_n + U^sb_1\cdots b_p$ describes the multiplication maps $m_{n+1}(v_1, a_1, \ldots, a_n) = U^tv_2$ and $m_{p+1}(v_1, b_1, \ldots, b_p) = U^pv_2$.
  \begin{figure}[h]
    \centering
    \scalefont{0.8}
    \begin{tikzpicture}[xscale=0.97,yscale=0.97]
      \node at (0,2.5) (x0) {$x_0$};
      \node at (2.5,2.5) (x1) {$x_1$};
      \node at (5,2.5) (x2) {$x_2$};
      \node at (7.5,2.5) (x3) {$x_3$};
      \node at (10,2.5) (x4) {$x_4$};
      \node at (12.5,2.5) (x5) {$x_5$};
      \node at (15,2.5) (x6) {$x_6$};
      \node at (2.5,0) (y1) {$y_1$};
      \node at (5,0) (y2) {$y_2$};
      \node at (7.5,0) (y3) {$y_3$};
      \node at (10,0) (y4) {$y_4$};
      \node at (12.5,0) (y5) {$y_5$};
      \node at (15,0) (y6) {$y_6$};
      \draw[algarrow] (x1) to node[above] {$\rho_2$} (x0);
      \draw[algarrow] (x2) to node[above] {$\rho_1$} (x1);
      \draw[algarrow] (x4) to node[below, xshift=15pt] {$\rho_1$} (x3);
      \draw[algarrow] (y4) to node[above] {$\rho_1$} (y3);
      \draw[algarrow] (y2) to node[above] {$U\rho_1$} (y1);
      \draw[al] (x2) to node[above, yshift=2pt] {$\rho_{12}$} (x0);
      \draw[ar] (x4) to node[above, xshift=10pt] {$U\rho_3\rho_2\rho_1$} (x6);
      \draw[ar] (x2) to node[above, yshift=2pt] {$U\rho_3\rho_2\rho_1$} (x5);
      \draw[algarrow] (x1) to node[midway, sloped, below] {$U^2+U\rho_{23}$} (y1);
      \draw[algarrow] (x4) to node[midway, sloped, above] {$U\rho_{123}\rho_2\rho_1$} (y5);
      \draw[algarrow] (x3) to node[midway, sloped, xshift= - 28pt, below] {$U\rho_{23}\rho_2\rho_1$} (y5);
      \draw[algarrow] (x2) to node[left] {$U$} (y2);
      \draw[algarrow] (x3) to node[left] {$U$} (y3);
      \draw[algarrow] (x4) to node[right, yshift= - 20pt] {$U$} (y4);
      \draw[algarrow] (x5) to node[right] {$U$} (y5);
      \draw[algarrow] (x6) to node[right] {$U$} (y6);
      \draw[algarrow] (x3) to node[right, xshift=-10pt, yshift=-15pt] {$\rho_2$} (y2);
      \draw[alb] (x4) to node[above, yshift=5pt] {$\rho_{12}$} (y2);
      \draw[ar] (y3) to node[above, yshift=2pt] {$\rho_2\rho_1$} (y1);
      \draw[ar] (y4) to node[xshift=-15pt, yshift=-2pt, below] {$\rho_{12}\rho_1$} (y1);
      \draw[al] (y4) to node[below, xshift=15pt] {$U\rho_3\rho_2\rho_1$} (y6);
      \draw[al] (y2) to node[below, xshift=10pt] {$U\rho_3\rho_2\rho_1$} (y5);
      \draw[al] (x0) to node[midway, sloped, below, xshift=-15pt] {$U\rho_3$} (y1);
      \draw[algarrow] (x2) to node[midway, sloped, above] {$U\rho_{123}$} (y1);
    \end{tikzpicture}	
    \vspace{-.2cm}	
    \caption{The type $A$ structure $\cfam(\mathcal{H})$.}
    \label{fig:cfa}
  \end{figure}

Consider  the periodic domain $B\in \pi_2(x_0, x_0)$ corresponding to traversing the loop
\begin{displaymath}
\xymatrix{ 
x_0 \ar[r]^{U\rho_3}& y_1  &\ar[l]_{U^2} x_1 \ar[r]^{\rho_2}& x_0.
}
\end{displaymath}
Using Equation~\ref{eqn:cfa}, we compute the following relative gradings in $\ha\backslash G$.
\begin{align*}
\gr(y_1) &= \gr (x_0)\gr(\rho_{3})(0; 0,0; 1)\in \ha\backslash G\\
\gr(y_1) &= \gr (x_1)\lambda^{-1}(0; 0,0; 2)\in \ha\backslash G\\
\gr(x_0) &= \gr (x_1)\gr(\rho_2)\in \ha\backslash G
\end{align*}
The first equation is equivalent to 
\[\gr(x_0) = \gr(y_1)(0; 0,0; -1)\gr(\rho_3)^{-1}.\]
Substituting the left coset $\gr (x_1)\lambda^{-1}(0; 0,0; 2)$ for $\gr(y_1)$, we get
\[\gr(x_0) = \gr (x_1)\lambda^{-1}(0; 0,0; 2)(0; 0,0; -1)\gr(\rho_3)^{-1}=\gr(x_1)\gr(\rho_3)^{-1}(-1;0,0;1).\]
Further substituting $\gr(x_1) = \gr(x_0)\gr(\rho_2)^{-1}$, we get
\[\gr(x_0) =  \gr(x_0)\gr(\rho_2)^{-1}\gr(\rho_3)^{-1}(-1;0,0;1) = \gr(x_0) \textstyle\left(\frac 1 2; 0, -1;1\right).\]
So $\left(\frac 1 2; 0, -1;1\right)\in \ha$, and since $\left(\frac 1 2; 0, -1;1\right)$ is not a positive multiple of another group element, it generates $\ha$. From here on, we will use the generator 
\[h_A =\textstyle \left(-\frac 1 2; 0, 1; -1\right)\]
of $\ha$.

From here on, we abuse notation and denote cosets by their representatives. We normalize the grading by setting 
\[\gr(x_0) = (0;0,0;0).\]
Since $m_2 (x_1, \rho_2) = x_0$, we get $\gr(x_0) = \gr(x_1)\gr(\rho_2)$, so 
\[\gr(x_1) = \gr(x_0)\gr(\rho_2)^{-1} = \textstyle\left(\frac 1 2; -\frac 1 2, -\frac 1 2; 0\right).\]
Since $m_2 (x_2, \rho_1) = x_1$, we get $\gr(x_1) = \gr(x_2)\gr(\rho_1)$, so 
\[\gr(x_2) = \gr(x_1)\gr(\rho_1)^{-1} = \textstyle\left(\frac 1 2; -1, 0; 0\right).\]
Continuing these computations along any spanning tree for the graph in Figure~\ref{fig:cfa}, we obtain the gradings of all generators. We summarize the result below. 
\begin{align*}
\gr(x_0) &= (0;0,0;0) & \gr(y_1) &=\textstyle \left(-\frac 1 2; -\frac 1 2, \frac 1 2; 1 \right) \\
\gr(x_1) &= \textstyle\left(\frac 1 2; -\frac 1 2, -\frac 1 2; 0\right) & \gr(y_2) &=\textstyle \left(-\frac 1 2; -1, 0; 1 \right)\\
\gr(x_2) &= \textstyle \left(\frac 1 2; -1, 0; 0\right) & \gr(y_3) &=\textstyle \left(-\frac 1 2; -\frac 3 2, -\frac 1 2; 2\right)\\
\gr(x_3) &= \textstyle\left(\frac 1 2; -\frac 3 2, -\frac 1 2; 1\right) & \gr(y_4) &= \left(-1; -2, 0; 2\right)\\
\gr(x_4) &= \left(0; -2, 0; 1\right) & \gr(y_5) &= \textstyle \left(-\frac 3 2; -\frac 1 2, \frac 1 2; 2 \right)\\
\gr(x_5) &=\textstyle \left(-\frac 1 2; -\frac 1 2, \frac 1 2; 1 \right) & \gr(y_6) &= \textstyle\left(-\frac 5 2; -\frac 3 2, \frac 1 2; 3 \right)\\
\gr(x_6) &= \textstyle\left(-\frac 3 2; -\frac 3 2, \frac 1 2; 2 \right) & &
\end{align*}
We remind the reader that following a different path to a given generator may result in a different representative of the same coset.

\subsection{The gradings on $\cfdhat(X_{K,n})$}\label{ssec:cfd-gr}
We begin with a discussion of the gradings in $G/\hd$ of the generators of $\cfdhat(X_{K,n})$. Since the last component of the grading is always zero here, we omit it. Recall from Equation~\ref{eqn:homogencfd} that each homogeneous generator $s$ of $\iota_0\cfdhat(X_{K,n})$ is graded by 
\begin{equation*}
\gr(s) = \textstyle \left(M(\tilde s)-\frac 3 2 A(
\tilde s); 0, -A(\tilde s)\right),
\end{equation*}
where $M(\tilde s)$ and $A(\tilde s)$ are the Maslov grading and Alexander filtration of the corresponding generator $\tilde s$ in $\cfkm(K)$, respectively.

Next consider the vertical chain 
\begin{displaymath}
\xymatrix{ 
\xi_{2i-1} \ar[r]^{D_{1}}&\kappa_1^i  &  \ar[l]_{D_{23}} \dots & \ar[l]_{D_{23}} \kappa_{k_i}^i & \ar[l]_{D_{123}}\xi_{2i}.
}
\end{displaymath}
Using Equation~\ref{eqn:cfd}, we  see that
\begin{align*}
\gr(\kappa_1^i) &= \lambda^{-1}\gr(\rho_{1})^{-1}\gr(\xi_{2i-1})\\
& = \lambda^{-1}\textstyle(\frac 1 2; -\frac 1 2, \frac 1 2) (M(\widetilde \xi_{2i-1})-\frac 3 2 A(\widetilde \xi_{2i-1}); 0,-A(\widetilde \xi_{2i-1}))\\
&= \textstyle (M(\widetilde \xi_{2i-1})- A(\widetilde \xi_{2i-1}) - \frac 1 2; -\frac 1 2,-A(\widetilde \xi_{2i-1}) + \frac 1 2).
\end{align*}
Continuing along the chain, we obtain the general formula
\begin{align*}
\gr(\kappa_j^i) &= \lambda^{j-1} \gr(\rho_{23})^{j-1}\gr(\kappa_1^i)\\
& = \textstyle(\frac j 2 - \frac 1 2; 0, j -1) \gr(\kappa_1^i)\\
&= \textstyle (M(\widetilde \xi_{2i-1})- A(\widetilde \xi_{2i-1}) + j- \frac 3 2; -\frac 1 2 , -A(\widetilde \xi_{2i-1}) + j -\frac 1 2).
\end{align*}
Similarly, traversing the horizontal chain 
\begin{displaymath}
\xymatrix{ 
\eta_{2i-1} \ar[r]^{D_3}&\lambda_1^i   \ar[r]^{D_{23}}& \dots\ar[r]^{D_{23}} & \lambda_{l_i}^i \ar[r]^{D_{2}}& \eta_{2i},
}
\end{displaymath}
we get
\[\gr(\lambda_j^i) = \textstyle (M(\widetilde \eta_{2i-1}) - 2A(\widetilde \eta_{2i-1}) - \frac 1 2; \frac 1 2, - A(\widetilde \eta_{2i-1}) - j +\frac 1 2 ).\]
Last, we traverse the unstable chain, starting from $\xi_0$ and working towards $\eta_0$. 
When $n=2\tau(K)$, there are no additional generators. When $n < 2\tau(K)$, the unstable chain takes the form
\begin{equation*}
\xi_0 \xrightarrow{\textrm{$D_1$}} \mu_1 \xleftarrow{\textrm{$D_{23}$}} \cdots \xleftarrow{\textrm{$D_{23}$}} \mu_{2\tau-n} \xleftarrow{\textrm{$D_{3}$}} \eta_0 
\end{equation*}
and we get
\[\gr(\mu_j) = \lambda^{j-2} \gr(\rho_{23})^{j-1}\gr(\rho_1)^{-1}\gr(\xi_0) = \textstyle(- \tau(K) + j - \frac 3 2; -\frac 1 2, -\tau(K) + j - \frac 1 2).\]
When $n>2\tau(K)$, the unstable chain takes the form
\begin{equation*}
\xi_0 \xrightarrow{\textrm{$D_{123}$}} \mu_1 \xrightarrow{\textrm{$D_{23}$}} \cdots \xrightarrow{\textrm{$D_{23}$}} \mu_{n-2\tau} \xrightarrow{\textrm{$D_{2}$}} \eta_0 
\end{equation*}
and we get 
\[\gr(\mu_j) = \lambda^{-j}\gr(\rho_{123})^{-1}\gr(\rho_{23})^{1-j}\gr(\xi_0) = \textstyle( - \tau(K) - j+\frac 1 2; -\frac 1 2, -\tau(K) - j + \frac 1 2).\]

\subsection{The gradings on $\cfahat(\mathcal H)\boxtimes \cfdhat(X_{K,n})$}\label{ssec:tensor-gr}

In this subsection, we compute grading representatives of $\textrm{gr}(x_A\boxtimes x_D)$ in the double-coset space $\ha\backslash G/\hd$ of the form $(a; 0,0; d)$ for all generators $x_A\boxtimes x_D \in \cfahat(\mathcal H)\boxtimes \cfdhat(X_{K,n})$. Note that not all these generators survive in homology; the differential is discussed in the subsequent section. Recall that $h_A = (-\frac 1 2; 0,1; -1)$ and $h_D = (-\frac{n}{2}-\frac{1}{2}; -1, -n; 0)$, where $n$ is the framing of $K$. Recall also that $\cfahat(\mathcal H)\boxtimes \cfdhat(X_{K,n})\simeq g\cfkhat(Q_n(K))$. The procedure is as follows. Given any generator $x_A\boxtimes x_D$, we multiply the coset grading representatives for $x_A$ and $x_D$ to obtain a double coset representative for $x_A\boxtimes x_D$.  Then we multiply the double coset representative by an appropriate power of $h_D$ on the right, to obtain a representative with $0$ in the second coordinate. Last, we multiply the new double coset representative by an appropriate power of  $h_A$ on the left, to obtain a double coset representative with $0$ in the second and third coordinates. In the resulting grading representative $(a; 0,0; d)$, $a$ is the absolute $z$-normalized Maslov grading $N$ of $x_A\boxtimes x_D$ in $g\cfkhat(Q_n(K))$, and  $d$ is the relative Alexander grading  $A_{\textrm{rel}}$  of $x_A\boxtimes x_D$ in $g\cfkhat(Q_n(K))$, that is, the Alexander grading $A$ considered up to an overall translation.

As an example, for generators of $\cfahat(\mathcal H)\boxtimes \cfdhat(X_{K,n})$ of the form $x_1\boxtimes \kappa_j^i$,
\begin{align*}
\gr(x_1\boxtimes \kappa_j^i) &=\gr(x_1)\gr(\kappa_j^i)\\
& = \textstyle(\frac1 2; -\frac 1 2,-\frac 1 2; 0)   (M(\widetilde \xi_{2i-1})- A(\widetilde \xi_{2i-1}) + j- \frac 3 2; -\frac 1 2 , -A(\widetilde \xi_{2i-1}) + j -\frac 1 2; 0)\\
&=  \textstyle(M(\widetilde \xi_{2i-1}) - \frac 1 2 A(\widetilde \xi_{2i-1}) + \frac 1 2 j -1; -1, - A(\widetilde \xi_{2i-1}) + j -1; 0)\\
&=  \textstyle(M(\widetilde \xi_{2i-1}) - \frac 1 2 A(\widetilde \xi_{2i-1}) +\frac 1 2 j -1; -1, - A(\widetilde \xi_{2i-1}) + j -1; 0) h_D^{-1}\\
&=  \textstyle(M(\widetilde \xi_{2i-1}) - \frac 1 2 A(\widetilde \xi_{2i-1}) + \frac 1 2 j -1; -1, - A(\widetilde \xi_{2i-1})+ j -1; 0) (\frac{n}{2} + \frac{1}{2}; 1, n; 0)\\
&= \textstyle(M(\widetilde \xi_{2i-1}) + \frac 1 2 A(\widetilde \xi_{2i-1}) - \frac 1 2 j - \frac n 2 + \frac 1 2; 0, - A(\widetilde \xi_{2i-1})+j +n -1; 0)  \\
&= h_A^{A(\widetilde \xi_{2i-1}) - j -n +1} \textstyle(M(\widetilde \xi_{2i-1}) + \frac 1 2 A(\widetilde \xi_{2i-1}) - \frac 1 2 j - \frac n 2 + \frac 1 2; 0, - A(\widetilde \xi_{2i-1})+j +n -1; 0)\\
&= (M(\widetilde \xi_{2i-1}); 0,0 ; - A(\widetilde \xi_{2i-1})+j +n -1).
\end{align*}
Thus $N(x_1\boxtimes \kappa_j^i) = M(\widetilde \xi_{2i-1})$ and $A_{\textrm{rel}}(x_1\boxtimes \kappa_j^i) = - A(\widetilde \xi_{2i-1})+j +n -1$.

We will also be interested in the $\delta$-grading $\delta = M-A$, where $M$ is the $w$-normalized Maslov grading. Since  $M=N+2A$ (see [Section 11.3]\cite{bfh2}, for example), we can also compute this $\delta$-grading as $\delta = N+A$. Thus, denoting the  $\delta$-grading up to overall translation by  $\delta_{\rm rel}$,  we see that we can compute it as $\delta_{\rm rel} = N + A_{\rm rel}$.  In the above example, $\delta_{\textrm{rel}}(x_1\boxtimes \kappa_j^i) =  M(\widetilde \xi_{2i-1}) - A(\widetilde \xi_{2i-1})+j +n -1$.

Proceeding in this way, we find the gradings $N$, $A_{\textrm{rel}}$, and $\delta_{\textrm{rel}}$ of all the remaining  types of generators of $\cfahat(\mathcal H)\boxtimes \cfdhat(X_{K,n})$. We summarize the results in Table~\ref{tab:tensor-gr}. 

\begin{table}[h]
\begin{center}
\resizebox{\textwidth}{!}{
  \begin{tabular}{ |    c    ||    l      |     l     ||   l   |}
    \hline
    Generator & $N$ & $A_{\rm rel}$ & $\delta_{\rm rel}$ \\ 
    \hline
    \hline
          \multicolumn{4}{|c|}{Generators arising from the vertical and the horizontal chains of $\cfdhat(X_{K,n})$:}\\ \hline
    \hline
    $x_0\boxtimes s$ & $M(\tilde s)- 2 A(
\tilde s)$ & $-A(\tilde s)$  & $M(\tilde s)- 3 A(
\tilde s)$  \\ \hline
    $x_2\boxtimes s$ & $M(\tilde s)+1$ & $-A(\tilde s)+n$  & $M(\tilde s)- A(
\tilde s)+n+1$  \\ \hline
   $y_2\boxtimes s$ & $M(\tilde s)$ & $-A(\tilde s)+n+1$  & $M(\tilde s)- A(
\tilde s)+n+1$  \\ \hline
     $x_4\boxtimes s$ & $M(\tilde s)+ 2A(
\tilde s) -2n+1$ & $- A(
\tilde s)+2n+1$  & $M(\tilde s) + A(
\tilde s)+2$  \\ \hline
     $y_4\boxtimes s$ & $M(\tilde s)+ 2A(
\tilde s) -2n$ & $- A(
\tilde s)+2n+2$  & $M(\tilde s) + A(
\tilde s)+2$  \\ \hline
     $x_1\boxtimes \lambda_j^i$ & $M(\widetilde \eta_{2i-1}) - 2A(\widetilde \eta_{2i-1})$ & $ - A(\widetilde \eta_{2i-1})-j$  & $M(\widetilde \eta_{2i-1}) - 3A(\widetilde \eta_{2i-1})- j$ \\ \hline
     $y_1\boxtimes \lambda_j^i$ & $M(\widetilde \eta_{2i-1}) - 2A(\widetilde \eta_{2i-1})-1$ & $ - A(\widetilde \eta_{2i-1})-j+2$  & $M(\widetilde \eta_{2i-1}) - 3A(\widetilde \eta_{2i-1})- j+1$ \\ \hline
     $x_1\boxtimes \kappa_j^i$ & $M(\widetilde \xi_{2i-1})$ & $ - A(\widetilde \xi_{2i-1})+j +n-1$  & $M(\widetilde \xi_{2i-1}) - A(\widetilde \xi_{2i-1})+j + n-1$ \\ \hline
          $y_1\boxtimes \kappa_j^i$ & $M(\widetilde \xi_{2i-1}) - 1$ & $ - A(\widetilde \xi_{2i-1})+j +n+ 1$  & $M(\widetilde \xi_{2i-1}) - A(\widetilde \xi_{2i-1})+j + n$ \\ \hline
          $x_3\boxtimes \lambda_j^i$ & $M(\widetilde \eta_{2i-1}) +2j$ & $ - A(\widetilde \eta_{2i-1})-j + n+1$  & $M(\widetilde \eta_{2i-1}) - A(\widetilde \eta_{2i-1}) + j +n+1$ \\ \hline
     $y_3\boxtimes \lambda_j^i$ & $M(\widetilde \eta_{2i-1}) +2j-1$ & $ - A(\widetilde \eta_{2i-1})-j + n+2$  & $M(\widetilde \eta_{2i-1}) - A(\widetilde \eta_{2i-1}) + j +n+1$ \\ \hline
     $x_3\boxtimes \kappa_j^i$ & $M(\widetilde \xi_{2i-1})+2A(\widetilde \xi_{2i-1})-2j -2n+2$ & $ - A(\widetilde \xi_{2i-1})+j +2n$  & $M(\widetilde \xi_{2i-1}) + A(\widetilde \xi_{2i-1})-j + 2$ \\ \hline
          $y_3\boxtimes \kappa_j^i$ & $ M(\widetilde \xi_{2i-1})+2A(\widetilde \xi_{2i-1})-2j -2n+1$ & $ - A(\widetilde \xi_{2i-1})+j +2n+1$  & $M(\widetilde \xi_{2i-1}) + A(\widetilde \xi_{2i-1})-j + 2$ \\ \hline
           $x_5\boxtimes \lambda_j^i$ & $M(\widetilde \eta_{2i-1}) - 2A(\widetilde \eta_{2i-1})-1$ & $ - A(\widetilde \eta_{2i-1})-j+2$  & $M(\widetilde \eta_{2i-1}) - 3A(\widetilde \eta_{2i-1})- j+1$ \\ \hline
     $y_5\boxtimes \lambda_j^i$ & $M(\widetilde \eta_{2i-1}) - 2A(\widetilde \eta_{2i-1})-2$ & $ - A(\widetilde \eta_{2i-1})-j+3$  & $M(\widetilde \eta_{2i-1}) - 3A(\widetilde \eta_{2i-1})- j+1$ \\ \hline
     $x_5\boxtimes \kappa_j^i$ & $M(\widetilde \xi_{2i-1})-1$ & $ - A(\widetilde \xi_{2i-1})+j +n+1$  & $M(\widetilde \xi_{2i-1}) - A(\widetilde \xi_{2i-1})+j + n$ \\ \hline
          $y_5\boxtimes \kappa_j^i$ & $M(\widetilde \xi_{2i-1}) - 2$ & $ - A(\widetilde \xi_{2i-1})+j +n+ 2$  & $M(\widetilde \xi_{2i-1}) - A(\widetilde \xi_{2i-1})+j + n$ \\ \hline
$x_6\boxtimes \lambda_j^i$ & $M(\widetilde \eta_{2i-1}) +2j-3$ & $ - A(\widetilde \eta_{2i-1})-j + n+3$  & $M(\widetilde \eta_{2i-1}) - A(\widetilde \eta_{2i-1}) + j +n$ \\ \hline
     $y_6\boxtimes \lambda_j^i$ & $M(\widetilde \eta_{2i-1}) +2j-4$ & $ - A(\widetilde \eta_{2i-1})-j + n+4$  & $M(\widetilde \eta_{2i-1}) - A(\widetilde \eta_{2i-1}) + j +n$ \\ \hline
     $x_6\boxtimes \kappa_j^i$ & $M(\widetilde \xi_{2i-1})+2A(\widetilde \xi_{2i-1})-2j -2n-1$ & $ - A(\widetilde \xi_{2i-1})+j +2n+2$  & $M(\widetilde \xi_{2i-1}) + A(\widetilde \xi_{2i-1})-j + 1$ \\ \hline
          $y_6\boxtimes \kappa_j^i$ & $ M(\widetilde \xi_{2i-1})+2A(\widetilde \xi_{2i-1})-2j -2n-2$ & $ - A(\widetilde \xi_{2i-1})+j +2n+3$  & $M(\widetilde \xi_{2i-1}) + A(\widetilde \xi_{2i-1})-j + 1$ \\ \hline
          \hline
          \multicolumn{4}{|c|}{Generators arising from the unstable chain of $\cfdhat(X_{K,n})$ when $n<2\tau(K)$:}\\ \hline
    \hline
    $x_1\boxtimes \mu_j$ & $0$ & $-\tau(K)+j+n-1$  & $-\tau(K)+j+n-1$  \\ \hline
    $y_1\boxtimes \mu_j$ & $-1$ & $-\tau(K)+j+n+1$  & $-\tau(K)+j+n$  \\ \hline
    $x_3\boxtimes \mu_j$ & $2\tau(K)  - 2j - 2n +2$ & $-\tau(K)+j+2n$  & $\tau(K)-j+2$  \\ \hline
    $y_3\boxtimes \mu_j$ &  $2\tau(K)  - 2j - 2n +1$ & $-\tau(K)+j+2n +1$  & $\tau(K)-j+2$  \\ \hline
    $x_5\boxtimes \mu_j$ & $-1$ & $-\tau(K)+j+n+1$  & $-\tau(K)+j+n$  \\ \hline
    $y_5\boxtimes \mu_j$ & $-2$ & $-\tau(K)+j+n+2$  & $-\tau(K)+j+n$  \\ \hline
    $x_6\boxtimes \mu_j$ & $2\tau(K)  - 2j - 2n -1$ & $-\tau(K)+j+2n+2$  & $\tau(K)-j+1$  \\ \hline
    $y_6\boxtimes \mu_j$ &  $2\tau(K)  - 2j - 2n -2$ & $-\tau(K)+j+2n +3$  & $\tau(K)-j+1$  \\ \hline
              \hline
          \multicolumn{4}{|c|}{Generators arising from the unstable chain of $\cfdhat(X_{K,n})$ when $n>2\tau(K)$:}\\ \hline
    \hline
    $x_1\boxtimes \mu_j$ & $1$ & $-\tau(K)-j+n$  & $-\tau(K)-j+n+1$  \\ \hline
    $y_1\boxtimes \mu_j$ & $0$ & $-\tau(K)-j+n+2$  & $-\tau(K)-j+n+2$  \\ \hline
    $x_3\boxtimes \mu_j$ & $2\tau(K)  + 2j - 2n +1$ & $-\tau(K)-j+2n+1$  & $\tau(K)+j+2$  \\ \hline
    $y_3\boxtimes \mu_j$ &  $2\tau(K)  + 2j - 2n$ & $-\tau(K)-j+2n +2$  & $\tau(K)+j+2$  \\ \hline
    $x_5\boxtimes \mu_j$ & $0$ & $-\tau(K)-j+n+2$  & $-\tau(K)-j+n+2$  \\ \hline
    $y_5\boxtimes \mu_j$ & $-1$ & $-\tau(K)-j+n+3$  & $-\tau(K)-j+n+2$  \\ \hline
    $x_6\boxtimes \mu_j$ & $2\tau(K)  + 2j - 2n -2$ & $-\tau(K)-j+2n+3$  & $\tau(K)+j+1$  \\ \hline
    $y_6\boxtimes \mu_j$ &  $2\tau(K)  + 2j - 2n -3$ & $-\tau(K)-j+2n +4$  & $\tau(K)+j+1$  \\ \hline
  \end{tabular}
  }
\end{center}
  \caption{The $z$-normalized Maslov gradings, relative Alexander gradings, and relative $\delta$-gradings on all  possible types of generators of the complex $\cfahat(\mathcal H)\boxtimes \cfdhat(X_{K,n})$.} 
  \label{tab:tensor-gr}
\end{table}

\subsection{The differential on $\cfahat(\mathcal H)\boxtimes \cfdhat(X_{K,n})$}\label{ssec:tensor-d}

Setting $U=0$ in $\cfam(\mathcal H)$, we obtain $\cfahat(\mathcal H)$, see Figure~\ref{fig:cfahat}.

  \begin{figure}[h]\label{CFA}
    \centering
    {
    \begin{tikzpicture}[xscale=0.9,yscale=0.9]
      \node at (0,2) (x0) {$x_0$};
      \node at (1,0) (x1) {$x_1$};
      \node at (2,2) (x2) {$x_2$};
      \node at (4.5,0) (x3) {$x_3$};
      \node at (5.5,2) (x4) {$x_4$};
      \node at (3.5,2) (y2) {$y_2$};
      \node at (9.5,0) (x5) {$x_5$};
      \node at (10.5,0) (x6) {$x_6$};
      \node at (8.5,0) (y1) {$y_1$};
      \node at (6.5,0) (y3) {$y_3$};
      \node at (7.5,2) (y4) {$y_4$};
      \node at (11.5,0) (y5) {$y_5$};
      \node at (12.5,0) (y6) {$y_6$};
      \draw[algarrow] (x1) to node[below, xshift=-7pt, yshift=5pt] {$\rho_2$} (x0);
      \draw[algarrow] (x2) to node[below ,xshift=7pt, yshift=5pt] {$\rho_1$} (x1);
      \draw[algarrow] (x4) to node[below, xshift=7pt, yshift=5pt] {$\rho_1$} (x3);
      \draw[algarrow] (y4) to node[below, xshift=-9pt, yshift=5pt] {$\rho_1$} (y3);
      \draw[algarrow] (x2) to node[below] {$\rho_{12}$} (x0);
      \draw[algarrow] (x3) to node[below, xshift=-7pt, yshift=5pt] {$\rho_2$} (y2);
      \draw[algarrow] (x4) to node[below] {$\rho_{12}$} (y2);
      \draw[algarrow] (y3) to node[above] {$\rho_2\rho_1$} (y1);
      \draw[algarrow] (y4) to node[below, xshift=18pt, yshift=5pt, below] {$\rho_{12}\rho_1$} (y1);
    \end{tikzpicture}	
    }	
    \caption{The type $A$ structure $\cfahat(\mathcal{H})$}
    \label{fig:cfahat}
  \end{figure}

Since $\cfahat(\mathcal H)$ is bounded, we can  boxtensor of $\cfahat(\mathcal H)$ with any type $D$ structure, and so we use the model  of $\cfdhat(X_{K,n})$ described in Section~\ref{sec:prelim} without having to analyze its boundedness (in general,  $\cfdhat(X_{K,n})$ may be an unbounded structure, and we may have to replace it with an equivalent bounded one).   By \cite[Lemmas 3.2-3.3]{hepsilon}, we may assume that our bases $\boldsymbol{\eta} = \{\eta_0, \ldots, \eta_{2k}\}$ and $\boldsymbol{\xi} = \{\xi_0, \ldots, \xi_{2k}\}$ are indexed so that  when $\epsilon(K)=0$ we have $\eta_0 = \xi_0$, when $\epsilon(K)=1$ we have $\xi_0=\eta_2$, and when $\epsilon(K)=-1$ we have $\eta_0=\xi_1$. 

To compute $\bdy^{\boxtimes}$, we pair the multiplication maps represented by the arrows in Figure~\ref{fig:cfahat} with (sequences of) coefficient maps in $\cfdhat(X_{K,n})$. From Figure~\ref{fig:cfahat}, we see that we only need to consider the length one sequences $D_1$, $D_2$, $D_{12}$, and the length two sequences $D_2, D_1$ and $D_{12}, D_1$. 

The coefficient map $D_1$ is seen once from $\xi_{2i-1}$ to $\kappa_1^i$ in each vertical chain, and once from $\xi_0$ to $\mu_1$ in the unstable chain when $n<2\tau(K)$. The map $D_2$ is seen once from  $\lambda_{l_i}^i$ to $\eta_{2i}$  in each  horizontal chain, and once from $\mu_m$ to $\eta_0$ in the unstable chain when $n>2\tau(K)$. The map $D_{12}$ is only seen once, from $\xi_0$ to $\eta_0$ in the unstable chain, when $n=2\tau(K)$. The sequence $D_2, D_1$ is seen from $\lambda_{l_i}^i$ to $\kappa_1^j$ whenever $a_{ij}|_{U=0}\neq 0$, once when $\epsilon(K) = 1$ and $n<2\tau(K)$ (because we've assumed that $\xi_0 = \eta_2$), and once from $\mu_m$ to $\kappa_1^1$ when $\epsilon(K)= -1$ and $n>2\tau(K)$ (because we've assumed that $\eta_0 = \xi_1$). The sequence $D_{12}, D_1$ appears only once, from $\xi_0$ to $\kappa_1^1$, when $\epsilon= -1$ and $n=2\tau(K)$ (because we've assumed that $\eta_0 = \xi_1$).

The following nontrivial differentials occur regardless of the value of $\epsilon(K)$ and the framing $n$.
\begin{align*}
\dbox (x_1\boxtimes \lambda_{l_i}^i) &= x_0\boxtimes \eta_{2i} & i=1,\ldots, k &\\
\dbox (x_3\boxtimes \lambda_{l_i}^i) &= y_2\boxtimes \eta_{2i}& i=1,\ldots, k &\\
\dbox (x_2\boxtimes \xi_{2i-1}) &= x_1\boxtimes \kappa_1^i & i=1,\ldots, k &\\
\dbox (x_4\boxtimes \xi_{2i-1}) &= x_3\boxtimes \kappa_1^i & i=1,\ldots, k & \\
\dbox (y_4\boxtimes \xi_{2i-1}) &= y_3\boxtimes \kappa_1^i & i=1,\ldots, k &\\
\dbox (y_3\boxtimes \lambda_{l_i}^i) &= y_1\boxtimes \kappa_1^j & \textrm{ whenever } a_{ij}|_{U=0}\neq 0 &
\end{align*}
where the first two rows of differentials come from pairings for $D_2$, the next three rows come from pairings for $D_1$, and the last row comes from pairings for the sequence $D_2, D_1$.

There are also the following additional nontrivial differentials that depend on the framing $n$.  

When $n<2\tau(K)$, we have
\[
\dbox (x_2\boxtimes \xi_0) = x_1\boxtimes \mu_1\qquad
\dbox (x_4\boxtimes \xi_0) = x_3\boxtimes \mu_1 \qquad
 \dbox (y_4\boxtimes \xi_0) = y_3\boxtimes \mu_1
\]
where all three differentials come from pairings for $D_1$, as well as 
\[\dbox(y_3\boxtimes\lambda_{l_1}^1) = y_1\boxtimes \mu_1\]
if $\epsilon(K)=1$, coming from a pairing for the sequence $D_2, D_1$.

When $n = 2\tau(K)$, we have
\[
\dbox (x_2\boxtimes \xi_0) = x_0\boxtimes \eta_0\qquad
\dbox (x_4\boxtimes \xi_0) = y_2\boxtimes \eta_0 
\]
where both differentials come from pairings for $D_{12}$, as well as 
\[ \dbox(y_4\boxtimes\xi_0) = y_1\boxtimes \kappa_1^1\]
if $\epsilon(K)=-1$, coming from a pairing for the sequence $D_{12}, D_1$. 

When $n>2\tau(K)$, we have
\[
\dbox (x_1\boxtimes \mu_m) = x_0\boxtimes \eta_0\qquad
\dbox (x_3\boxtimes \mu_m) = y_2\boxtimes \eta_0
\]
where both differentials come from pairings for $D_2$, as well as 
\[ \dbox(y_3\boxtimes\mu_m) = y_1\boxtimes \kappa_1^1\]
if $\epsilon(K)=-1$, coming from a pairing for the sequence $D_{2}, D_1$.



\section{$\delta$-thickness of $Q_n(K)$}\label{sec:deltathick}

We start by proving that $Q_n(K)$ is $\delta$-thick for all integers $n$ and knots $K$, except for two satellites obtained when $K$ is the unknot and $n$ is $-1$ or $0$,  which are $\delta$-thin.

\begin{proof}[Proof of Theorem~\ref{theorem_delta_thick}]
Recall that we have horizontally and vertically simplified bases $\widetilde{\boldsymbol{\eta}} = \{\widetilde{\eta}_0, \ldots, \widetilde{\eta}_{2k}\}$ and $\widetilde{\boldsymbol{\xi}} = \{\widetilde{\xi}_0, \ldots, \widetilde{\xi}_{2k}\}$, respectively, for $\cfkm(K)$ that induce bases $\boldsymbol{\eta} = \{\eta_0, \ldots, \eta_{2k}\}$ and $\boldsymbol{\xi} = \{\xi_0, \ldots, \xi_{2k}\}$ for the subspace $\iota_0  \cfdhat(X_{K,n})$. Since $\widetilde{\boldsymbol{\eta}}$ and $\widetilde{\boldsymbol{\xi}}$ are simplified bases for  $\cfkm(K)$, we can treat them as bases of $\hfkhat(K)$ as well. In particular, this implies that $\mathrm{rk}\, \hfkhat(K) = 2k+1$.
We will make use of the following simple lemma.

\begin{lemma}\label{lem:etaoddneg}
If $K$ is not the unknot or a trefoil, then there is some $\widetilde\eta_{2t-1} \in \widetilde{\boldsymbol{\eta}}$ with $A(\widetilde\eta_{2t-1}) < 0$.
\end{lemma}
\begin{proof}
Assume, to the contrary, $\widetilde\eta_1,\widetilde\eta_3, \ldots \widetilde\eta_{2k-1}$ all have nonnegative Alexander degree. Recall that the basis $\widetilde{\boldsymbol{\eta}}$ is indexed so that there is a horizontal arrow from $\widetilde\eta_{2i-1}$ to $\widetilde\eta_{2i}$ for each $i$. Since the horizontal arrows strictly increase the Alexander degree, it follows that $\widetilde\eta_2, \widetilde\eta_4 \ldots, \widetilde\eta_{2k}$ all have positive Alexander degree. By symmetry of $\hfkhat$, there must be at least $k$ generators in $\widetilde{\boldsymbol{\eta}}$ with negative Alexander degree. So $\mathrm{rk}\, \hfkhat(K) \geq 3k$.

On the other hand, since the dimension of $\hfkhat$ is always odd and detects both the trefoil \cite{hw} and the unknot \cite{hfkg}, we have $\mathrm{rk} \, \hfkhat(K) = 2k+1 \geq 5$. This is a contradiction.
\end{proof}

Now consider $\cfdhat(X_{K,n})$, and consider the basis $\boldsymbol{\eta}$ for $\iota_0  \cfdhat(X_{K,n})$. At each odd-indexed element $\eta_{2i-1}$, we have the following arrows: 
\begin{enumerate}
\item[-] An outgoing $D_3$-arrow to $\lambda_1^{i}$.
\item[-] An outgoing $D_1$-arrow to $\kappa_1^{j}$, whenever $\xi_{2j-1}$ appears with nonzero coefficient in $\eta_{2i-1}$.
\item[-] An outgoing $D_{123}$-arrow to $\kappa_{k_j}^{j}$, whenever $\xi_{2j}$ appears with nonzero coefficient in $\eta_{2i-1}$.
\item[-] An outgoing arrow labelled $D_1$, $D_{12}$, or $D_{123}$, depending on the framing, if $\xi_0$ appears with nonzero coefficient in $\eta_{2i-1}$. 
\end{enumerate}
Regardless of the framing and how the bases $\boldsymbol{\eta}$ and $\boldsymbol{\xi}$ are related, there are no incoming arrows at $\eta_{2i-1}$. Since there are no outgoing edges at the generators $x_0$ and $y_2$ of $\cfahat(\mathcal{H})$, the elements $x_0\boxtimes \eta_{2i-1}$ and $y_2\boxtimes \eta_{2i-1}$ survive in the homology of $\cfahat(\mathcal{H})\boxtimes \cfdhat(X_{K,n})$, i.e. they represent distinct generators of $\hfkhat(Q_n(K))$. 

Similarly, since the only outgoing arrows at each even-indexed element $\xi_{2i}$ are labeled $D_3$ or $D_{123}$, and there are no matching labels in our model for $\cfahat(V,Q)$, the elements $x_2\boxtimes \xi_{2i}$ and $x_4\boxtimes \xi_{2i}$  are all nonzero in $\hfkhat(Q_n(K))$.

Next, we consider the relative $\delta$-gradings of the above generators of homology. From Table~\ref{tab:tensor-gr}, we have
\begin{align*}
\delta_{\textrm{rel}}(x_0\boxtimes\eta_{2i-1}) &=M(\widetilde\eta_{2i-1})-3 A(\widetilde\eta_{2i-1})\\
\delta_{\textrm{rel}}(y_2\boxtimes\eta_{2i-1}) &=  M(\widetilde\eta_{2i-1}) -A(\widetilde\eta_{2i-1})+n+1\\
\delta_{\textrm{rel}}(x_4\boxtimes\xi_{2i}) &=  M(\widetilde\xi_{2i})+A(\widetilde\xi_{2i})+2.
\end{align*}

{\bf{Case 1:}} Suppose there are two elements $\widetilde\eta_{2t-1}$ and $\widetilde\eta_{2s-1}$ with distinct $\delta$-degrees. Then  $\delta(y_2\boxtimes\eta_{2t-1}) - \delta(y_2\boxtimes\eta_{2s-1}) = \delta_{\textrm{rel}}(y_2\boxtimes\eta_{2t-1}) - \delta_{\textrm{rel}}(y_2\boxtimes\eta_{2s-1}) \neq 0$, so $Q_n(K)$ is $\delta$-thick. 

{\bf{Case 2:}} Suppose all odd-indexed elements in $\widetilde{\boldsymbol{\eta}}$ are in the same $\delta$-degree. 

{\bf{Case 2.1:}}  Suppose there exist two elements $\widetilde\eta_{2t-1}$ and $\widetilde\eta_{2s-1}$  in different Alexander degrees. Then 
\begin{align*}
\delta(x_0\boxtimes\eta_{2t-1}) - \delta(x_0\boxtimes\eta_{2s-1}) &=  \delta_{\textrm{rel}}(x_0\boxtimes\eta_{2t-1}) - \delta_{\textrm{rel}}(x_0\boxtimes\eta_{2s-1}) \\
&= M(\widetilde\eta_{2t-1})-3 A(\widetilde\eta_{2t-1}) - M(\widetilde\eta_{2s-1})+3 A(\widetilde\eta_{2s-1}) \\
&=\delta(\widetilde\eta_{2t-1}) - \delta(\widetilde\eta_{2s-1}) - 2A(\widetilde\eta_{2t-1}) + 2A(\widetilde\eta_{2s-1})\\
&= - 2A(\widetilde\eta_{2t-1}) + 2A(\widetilde\eta_{2s-1})\neq 0,
\end{align*}
so $Q_n(K)$ is $\delta$-thick. 

{\bf{Case 2.2:}} Suppose all odd-indexed elements in $\widetilde{\boldsymbol{\eta}}$ have the same bidegree $(M, A)$.  First, assume $K$ is neither the unknot nor a trefoil. Lemma \ref{lem:etaoddneg} implies that the Alexander degree $A$ in which the odd-indexed elements in $\widetilde{\boldsymbol{\eta}}$ are supported is negative. By symmetry of  $\hfkhat$, there are $k\geq 2$ generators in $\widetilde{\boldsymbol{\eta}}$ with bidegree $(M-2A, -A)$, and the one remaining generator has Alexander degree zero.

Consider the basis $\widetilde{\boldsymbol{\xi}}$. It also has $k$ elements in Alexander degree $A$ and $k$ elements in Alexander degree $-A$. Recall that $k \geq 2$. Since $A <0$ and the vertical arrows strictly decrease the Alexander grading, there is at least one element $\widetilde\xi_{2t}$ in bidegree $(M, A)$ with $t\geq 1$.
We see that $\delta(x_0\boxtimes\eta_{3}) - \delta(x_4\boxtimes\xi_{2t}) = \delta_{\textrm{rel}}(x_0\boxtimes\eta_{3}) - \delta_{\textrm{rel}}(x_4\boxtimes\xi_{2t}) = -4A-2$, which is nonzero, since $A$ is an integer. So  $Q_n(K)$ is $\delta$-thick. 

Now suppose $K$ is the unknot or a trefoil. If $K$ is the unknot, then $Q_n(K)$ is $\delta$-thick exactly when $n\notin\{-1,0\}$, by Proposition~\ref{thm:alt-th}. Note that in this case $Q_0(K)$ is the unknot and $Q_{-1}(K)$ is the $5_2$ knot. If $K$ is a trefoil, then by Proposition~\ref{thm:alt-th}, $Q_n(K)$ is $\delta$-thick for all values of $n$.
\end{proof}

Further, we show that as the number of twists on the Mazur pattern increases, the $\delta$-thickness increases without bound.

\begin{proof}[Proof of Theorem~\ref{thm:th-unbounded}]
Let  $n<2\tau(K)$. Observe that for any generator $\mu_i$ along the unstable chain of $\cfdhat(X_{K,n})$, the tensor product $x_6\boxtimes \mu_i$ survives in the homology of $\cfahat(\mathcal{H}) \boxtimes \cfdhat(X_{K,n})$. Further, the generators $x_6\boxtimes \mu_i$ all have distinct $\delta$-gradings. In particular, 
\[\mathrm{th}(Q_n(K))\geq \delta(x_6\boxtimes \mu_1) - \delta(x_6\boxtimes \mu_{2\tau(K)-n}) = 2\tau(K)-n-1.\]
Similarly, when $n>2\tau(K)$, we have that every tensor product of the form $x_6\boxtimes \mu_i$ survives in the homology of $\cfahat(\mathcal{H}) \boxtimes \cfdhat(X_{K,n})$ and
\[\mathrm{th}(Q_n(K))\geq \delta(x_6\boxtimes \mu_{n-2\tau(K)}) - \delta(x_6\boxtimes \mu_1) = -2\tau(K)+n-1.\]
Thus, 
\[\lim_{n \rightarrow \pm \infty} \mathrm{th}(Q_n(K)) = \infty.\qedhere\]
\end{proof}

Below, we prove Theorem~\ref{thm:th-bounded},  which states that the analogue of Theorem~\ref{thm:th-unbounded} does not hold for all patterns.

\begin{proof}[Proof of Theorem~\ref{thm:th-bounded}]

Let $P$ be the $(2,1)$-cable in the solid torus, and let $K$ be a $\delta$-thin knot.

In \cite[Section 4]{cables}, using a doubly-pointed Heegaard diagram $\mathcal H(P)$, the type $A$ structure $\cfahat(\mathcal H(P))$ was computed. It has three generators: $a$, $b_1$, and $b_2$. Generator $a$ lies in idempotent $\iota_0$, while generators $b_1$ and $b_2$ lie in idempotent $\iota_1$. The only nontrivial multiplication map is $m_2(a, \rho_1) = b_2$. The grading lies in the set $\langle h_A'\rangle\backslash G$, where $h_A' = (-\frac 1 2;0 , 1; 2)$, and the three generators are graded as follows:
\begin{align*}
\gr(a) &= (0; 0,0; 0)\in \langle h_A'\rangle\backslash G\\
\gr(b_1) &= \textstyle (\frac 1 2 ; \frac 1 2,-\frac 1 2 ; -1)\in \langle h_A'\rangle\backslash G\\
\gr(b_2) &= \textstyle (-\frac 1 2 ; \frac 1 2,-\frac 1 2 ; 0) \in \langle h_A'\rangle\backslash G.
\end{align*}

Since $K$ is thin, the type $D$ structure $\cfdhat(X_{K,n})$ is particularly simple. All elements $\xi_i$ in the basis $\boldsymbol{\xi}$ for $\iota_0  \cfdhat(X_{K,n})$ correspond to elements of $\cfkm(K)$ in the same $\delta$-grading. All vertical and horizontal chains have length one, so the generators of $\iota_1 \cfdhat(X_{K,n})$ are of form $\lambda_1^i$, $\kappa_1^i$, and $\mu_j$. Below, we write $\tau$ for $\tau(K)$. By Section~\ref{ssec:cfd-gr}, the gradings in $G/\hd$, where $h_D =(-\frac{n}{2}-\frac{1}{2}; -1, -n; 0)$, are
\begin{align*}
\gr(\xi_i) &= \textstyle \left(M(\widetilde \xi_i)-\frac 3 2 A(
\widetilde \xi_i); 0, -A(\widetilde \xi_i)\right) &\\
\gr(\kappa_1^i) &= \textstyle (M(\widetilde \xi_{2i-1})- A(\widetilde \xi_{2i-1}) - \frac 1 2; -\frac 1 2,-A(\widetilde \xi_{2i-1}) + \frac 1 2)&\\
\gr(\lambda_1^i) &=  \textstyle (M(\widetilde \eta_{2i-1}) - 2A(\widetilde \eta_{2i-1}) - \frac 1 2; \frac 1 2, - A(\widetilde \eta_{2i-1}) -  \frac 1 2 )&\\
\gr(\mu_j) & = \textstyle(- \tau + j - \frac 3 2; -\frac 1 2, -\tau + j - \frac 1 2) &\quad {\textrm{if } n < 2\tau} \\
\gr(\mu_j) &= \textstyle( - \tau - j+\frac 1 2; -\frac 1 2, -\tau - j + \frac 1 2)& \quad {\textrm{if } n > 2\tau}.
\end{align*}

Using the procedure described in the beginning of Section~\ref{ssec:tensor-gr}, we see that 
\begin{align*}
\gr(a\boxtimes\xi_i) & = (-\tau- A(\xi_i); 0,0; 2A(\xi_i)) & \\
\gr(b_1\boxtimes\kappa_1^i) & = (-\tau- A(\xi_{2i-1}); 0,0; 2A(\xi_{2i-1})-1) & \\
\gr(b_2\boxtimes\kappa_1^i) & = (-\tau- A(\xi_{2i-1})-1; 0,0; 2A(\xi_{2i-1})) & \\
\gr(b_1\boxtimes\lambda_1^i) & = (-\tau- 3A(\xi_{2i-1})- 2n-2; 0,0; 2A(\xi_{2i-1})+2n+1) & \\
\gr(b_2\boxtimes\lambda_1^i) & = (-\tau- 3A(\xi_{2i-1})- 2n-3; 0,0; 2A(\xi_{2i-1})+2n+2) & \\
\gr(b_1\boxtimes\mu_j) & = (-2\tau + 2j-2; 0,0; 2\tau-2j+1) & \quad {\textrm{if } n < 2\tau}  \\
\gr(b_2\boxtimes\mu_j) & =(-2\tau + 2j-3; 0,0; 2\tau-2j+2) & \quad {\textrm{if } n < 2\tau} \\
\gr(b_1\boxtimes\mu_j) & = (-2\tau - 2j + 1; 0,0; 2\tau+2j-1)  & \quad {\textrm{if } n > 2\tau}\\
\gr(b_2\boxtimes\mu_j) & =(-2\tau - 2j; 0,0; 2\tau+2j) &   \quad {\textrm{if } n > 2\tau}.
\end{align*}
Note that in the grading computations above, we've used the fact that when $K$ is a thin knot, $M(\xi_{i}) = A(\xi_{i}) - \tau$.

Adding the first and fourth component of each grading, we see that the relative $\delta$-gradings are supported in the set 
\[\{-\tau +A(\xi_i),  -\tau+ A(\xi_{2i-1})-1, -\tau- A(\xi_{2i-1})-1, -1, 0\}.\]
Since $|A(\xi_i)|\leq g(K)$ for all $\xi_i$, the values in this set are bounded above by $-\tau + g$ and below by $-\tau -g-1$. Thus $\mathrm{th}(P_n(K))\leq 2g+1$, regardless of the number of twists $n$ on the pattern. 

Note that to determine the exact thickness for each $n$, we need to consider the differential. We do not need to do this to prove our proposition. The type $D$ structure for the complement of a thin knot is discussed in further detail in Section~\ref{ssec:thin-comp}.
\end{proof}

\section{3-genus and fiberedness of $Q_n(K)$}\label{sec:gf}

In this section, we combine a bordered Floer homology computation with a couple of classical results to calculate the 3-genus of $Q_n(K)$ and to determine when $Q_n(K)$ is fibered, for all $n$ and $K$.

We first focus on the case where $K$ is the right-handed trefoil. To compute the 3-genus of $Q_n(K)$, it suffices to find the extremal Alexander degrees in $\hfkhat(Q_n(K))$ \cite{hfkg}. To determine whether $Q_n(K)$ is fibered, it suffices to compute the rank of $\hfkhat(Q_n(K))$ in the top Alexander degree \cite{pgf, nfibered}. 

Recall that $\widehat{\mathit{HFK}}(Q_n(K)) \cong H_{*}(\cfahat(\mathcal{H}) \boxtimes \cfdhat(X_{K,n}))$, where $\cfdhat(X_{K, n})$ is as follows:
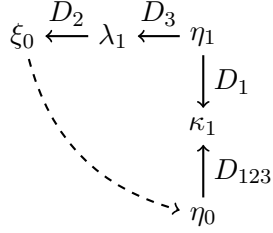
\begin{figure}[h]
  \centering
    {\scalefont{0.9}
    \begin{tikzpicture}[xscale=0.6,yscale=0.6]
      \node at (0,0) (x0) {$\xi_0$};
      \node at (4,-4) (x1) {$\eta_0$};
      \node at (4,0) (x3) {$\eta_1$};
      \node at (2,0) (l1) {$\lambda_1$};
      \node at (4,-2) (k1) {$\kappa_1$};
      \draw[algarrow] (x3) to node[right] {$D_1$} (k1);
      \draw[algarrow] (x1) to node[right] {$D_{123}$} (k1);
      \draw[algarrow] (x3) to node[above] {$D_3$} (l1);
      \draw[algarrow] (l1) to node[above] {$D_2$} (x0);
             \draw[unst1] (x0) to node[below] {} (x1);
    \end{tikzpicture}	
    }	
  \caption{$\cfdhat(X_{K, n})$ for the right-handed trefoil $K$. The dotted arrow represents the unstable chain.}
 \label{fig:cfdrht}
\end{figure}

We use the values from Table~\ref{tab:tensor-gr}, combined with the differential computed in Section~\ref{ssec:tensor-d},  to find the extremal relative Alexander degrees in $\hfkhat(Q_n(K))$ and the generators of $\hfkhat(Q_n(K))$ in those degrees. 

When $n < -1$, the nontrivial differentials  on $\cfahat(\mathcal{H}) \boxtimes \cfdhat(X_{K,n})$ are
\[\dbox (x_2\boxtimes \eta_1) = x_1\boxtimes \kappa_1\quad
\dbox (x_4\boxtimes \eta_1) = x_3\boxtimes \kappa_1 \quad
 \dbox (y_4\boxtimes \eta_1) = y_3\boxtimes \kappa_1\quad 
\dbox(x_1\boxtimes\lambda_1) = x_0\boxtimes \xi_0\]
\[\dbox (x_3\boxtimes \lambda_1) = y_2\boxtimes \xi_0\quad
\dbox (y_3\boxtimes \lambda_1) = y_1\boxtimes \mu_1 \quad
 \dbox (x_2\boxtimes \xi_0) = x_1\boxtimes \mu_1\quad 
\dbox(x_4\boxtimes \xi_0) = x_3\boxtimes \mu_1\]
\[\dbox (y_4\boxtimes \xi_0) = y_3\boxtimes \mu_1.\]
The generators of $\hfkhat(Q_n(K))$, together with their relative Alexander degrees, are given by Table \ref{tab:rhtn<-1}. One can easily verify that the minimum relative Alexander degree is $2n+1$ realized only by generator $x_3 \boxtimes \mu_2$, and the maximum Alexander degree is $3$ realized only by generator $y_5 \boxtimes \mu_{2-n}$.
\begin{table}[h]
\begin{center}
  \begin{tabular}{ | c | c || c | c | }
    \hline
    Generator & $A_\textrm{rel}$ & Generator & $A_\textrm{rel}$  \\ 
    \hline
    \hline
    $x_0\boxtimes \eta_0$  & $1$ & $x_5\boxtimes \kappa_1$  & $n+2$ \\ \hline
    $x_2\boxtimes \eta_0$  & $1+n$ &  $y_5\boxtimes \kappa_1$ & $n+3$\\ \hline
    $y_2\boxtimes \eta_0$  & $2+n$ & $x_6\boxtimes \kappa_1$  & $2n+3$ \\ \hline
    $x_4\boxtimes \eta_0$  & $2n+2$ & $y_6\boxtimes \kappa_1$  & $2n+4$ \\ \hline
    $y_4\boxtimes \eta_0$  & $2n+3$ & $x_1 \boxtimes \mu_j, j \in \{2, \ldots, 2-n\}$  & $n+j-2$ \\ \hline
    $x_0\boxtimes \eta_1$  & $0$ & $y_1 \boxtimes \mu_j, j \in \{2, \ldots, 2-n\}$ & $n+j$ \\ \hline
    $y_2\boxtimes \eta_1$  & $n+1$ & $x_3 \boxtimes \mu_j, j \in \{2, \ldots, 2-n\}$ & $2n+j-1$ \\ \hline
    $y_1\boxtimes \lambda_1$  & $1$ & $y_3 \boxtimes \mu_j, j \in \{2, \ldots, 2-n\}$ & $2n+j$  \\ \hline
    $x_5\boxtimes \lambda_1$  & $1$ &$x_5 \boxtimes \mu_j, j \in \{1, \ldots, 2-n\}$  & $n+j$\\ \hline
    $y_5\boxtimes \lambda_1$  & $2$ & $y_5 \boxtimes \mu_j, j \in \{1, \ldots, 2-n\}$ & $n+j+1$ \\ \hline
    $x_6\boxtimes \lambda_1$  & $n+2$ & $x_6 \boxtimes \mu_j, j \in \{1, \ldots, 2-n\}$ & $2n+j+1$ \\ \hline
    $y_6\boxtimes \lambda_1$  & $n+3$ & $y_6 \boxtimes \mu_j, j \in \{1, \ldots, 2-n\}$  & $2n+j+2$ \\ \hline
    $y_1\boxtimes \kappa_1$ & $n+2$ &  & \\ \hline
  \end{tabular}
\end{center}
 \caption{The generators of $\hfkhat(Q_n(K))$ and their relative Alexander degrees for the right-handed trefoil $K$ and framing $n <-1$.}
  \label{tab:rhtn<-1}
\end{table}

The cases when  $n \geq -1$ are similar. We summarize the results in Table~\ref{tab:rht}.
\begin{table}[h]
\begin{center}
  \begin{tabular}{ | c || c | c || c | c |}
    \hline
    Framing & Min $A_\textrm{rel}$ & Generators  & Max $A_\textrm{rel}$ & Generators \\ 
    \hline
    \hline
    $n<-1$ & $2n+1$ & $x_3\boxtimes \mu_2$  & $3$ & $y_5\boxtimes \mu_{2-n}$ \\ \hline
    $n=-1$ & $-1$ & $x_1\boxtimes \mu_2$, $x_3\boxtimes\mu_2$  & $3$ & $y_5\boxtimes \mu_3$, $y_6\boxtimes \mu_3$ \\ \hline
    $n=0$ & $0$ & $x_0\boxtimes \eta_1$, $x_1\boxtimes\mu_2$  & $4$ & $y_6\boxtimes \kappa_1$, $y_6\boxtimes \mu_2$ \\ \hline
     $n=1$ & $0$ & $x_0\boxtimes \eta_1$  & $6$ & $y_6\boxtimes \kappa_1$ \\ \hline
     $n=2$ & $0$ & $x_0\boxtimes \eta_1$  & $8$ & $y_6\boxtimes \kappa_1$ \\ \hline
     $n>2$ & $0$ & $x_0\boxtimes \eta_1$  & $2n+4$ & $y_6\boxtimes \kappa_1$ \\ \hline
  \end{tabular}
\end{center}
 \caption{The extremal relative Alexander degrees in $\hfkhat(Q_n(K))$, together with the generators of $\hfkhat(Q_n(K))$ in those degrees, for the right-handed trefoil $K$.}
  \label{tab:rht}
\end{table}

Now since the $3$-genus is half the difference between the highest and the lowest Alexander degrees, Table \ref{tab:rht} implies that
\begin{equation*}
g(Q_n(K)) = 
\begin{cases}
                                   1 - n & \text{if $n\leq -1$,} \\
                                   n + 2 & \text{if $n\geq 0$}.                                 
\end{cases}
\end{equation*}

Furthermore, because a knot is fibered if and only if its knot Floer homology has rank $1$ in the highest Alexander degree, we conclude that $Q_n(K)$ is fibered if and only if $n$ is not $-1$ or $0$.

Next we use our work for the right-handed trefoil to calculate the 3-genus of $Q_n(K)$ and to determine when $Q_n(K)$ is fibered, given any nontrivial knot $K$ not equal to the right-handed trefoil and given any number of twists $n$.

\begin{proof}[Proof of Theorem~\ref{thm:g3} for nontrivial knots $K$ not equal to the right-handed trefoil]
We can think of $Q_n(K)$ as the $0$-twisted satellite knot $(Q_n)_0(K)$ with pattern the $n$-twisted Mazur knot $Q_n$ and companion $K$. Since the winding number of $Q_n$ is $1$, by a result attributed to Schubert \cite{schubert},
\begin{equation*}
g(Q_n(K)) = g((Q_n)_0(K)) = g(K) + g(Q_n),
\end{equation*}
where $g(Q_n)$ is a number that depends only on the pattern $Q_n$. This means that we can determine the genus of $Q_n(K)$ if we know the constant $g(Q_n)$. The same equation also tells us that if we know the genus of some test companion $K'$ and the genus of its corresponding satellite $Q_n(K')$, then this constant $g(Q_n)$ is just $g(Q_n(K')) - g(K')$. Take $K'$ to be the right-handed trefoil. From our work above,
\begin{equation*}
g(Q_n) = g(Q_n(K')) - g(K') =
\begin{cases}
                                   - n & \text{if $n\leq -1$} \\
                                    n + 1 & \text{if $n\geq 0$}.                                  
\end{cases} 
\end{equation*}
This completes the proof of Theorem~\ref{thm:g3} for nontrivial $K$.
\end{proof}

\begin{proof}[Proof of Theorem~\ref{thm:fib} for nontrivial knots $K$ not equal to the right-handed trefoil]
Once again, we think of $Q_n(K)$ as the $0$-twisted satellite knot $(Q_n)_0(K)$ with pattern $Q_n$ and companion $K$. By a theorem of Hirasawa-Murasugi-Silver \cite[Theorem 2.1]{hms}, $Q_n(K)$ is fibered if and only if $K$ is fibered and $Q_n$ is fibered in the solid torus. Since the right-handed trefoil is fibered, the above computation shows that the pattern $Q_n$ is fibered in the solid torus if and only if $n$ is not $-1$ or $0$. Therefore, when $n=-1,0$, the satellite $Q_n(K)$ is never fibered, and when $n\neq -1,0$, the satellite $Q_n(K)$ is fibered if and only if $K$ is fibered.
\end{proof}

Lastly, we determine the 3-genus and fiberedness of $Q_n(U)$. 
Again, we use the values from Table~\ref{tab:tensor-gr}. The case analysis is similar to above. 

When $n<-1$, the minimum and the maximum relative Alexander degrees are $2n+2$ and $2$,  realized by $x_3\boxtimes \mu_2$, and $y_5\boxtimes \mu_{-n}$, respectively. 
When $n=-1$, we have $Q_{-1}(U)= 5_2$, which is known to have genus $1$ and not be fibered. 
When $n=0$, we have $Q_0(U) = U$ (genus zero, fibered). 
When $n=1$, the minimum and the maximum relative Alexander degrees are $1$ and $5$,  realized by $x_2\boxtimes \eta_0$, and $y_6\boxtimes \mu_1$, respectively. 
Last, when $n>1$, the minimum and the maximum relative Alexander degrees are $1$ and $2n+3$, realized by $x_1\boxtimes \mu_{n-1}$, and $y_6\boxtimes \mu_1$, respectively. 

It follows that  
\begin{equation*}
g(Q_n(U)) = 
\begin{cases}
                                    - n & \text{if $n\leq 0$,} \\
                                   n + 1 & \text{if $n\geq 1$},                                 
\end{cases}
\end{equation*}
and that $Q_n(U)$ is fibered if and only if $n\neq -1$.

\section{An application to the Cosmetic Surgery Conjecture}\label{sec:csc}
In this section, we prove Theorem~\ref{thm:csc}.

In  \cite[Theorem 2]{hanselman}, Hanselman shows that if $K\subset S^3$ is a nontrivial knot and $S^3_r(K)\cong S^3_{r'}(K)$, for $r\neq r'$, then the pair of surgery slopes $\{r, r'\}$ is either $\{\pm 2\}$ or $\{\pm \frac 1 q\}$ for some positive integer $q$. Further, he shows that if  $\{r, r'\}=\{\pm 2\}$, then $g(K)=2$, and if $\{r, r'\} = \{\pm \frac 1 q\}$, then 
\[q \leq \frac{\mathrm{th}(K)+2g(K)}{2g(K)(g(K)-1)}.\]
In particular, if $g(K) \geq 3$, and 
\begin{equation}\label{eqn:csc}
\frac{\mathrm{th}(K)+2g(K)}{2g(K)(g(K)-1)} <1,
\end{equation}
the knot $K$ automatically satisfies the cosmetic surgery conjecture. Define 
\[f(K) = 2(g(K))^2 - 4g(K) - \mathrm{th}(K),\]  and observe that Inequality~\ref{eqn:csc} is equivalent to the inequality
\begin{equation}\label{eqn:csc2}
f(K)>0. 
\end{equation}
We will show that nontrivial satellites $Q_n(K)$ satisfy the cosmetic surgery conjecture, whenever $K$ is a thin knot (with a small set of unverified exceptions), or an $L$-space knot (of a certain type). 
Except for a few special cases, which we analyze using other tools, we use Inequality~\ref{eqn:csc2}, so we need to combine the genus values from Theorem~\ref{thm:g3} with a computation of the $\delta$-thickness $\mathrm{th}(Q_n(K))$. There are two other tools that we will use for the special cases. The first is an  obstruction of Boyer--Lines, which says that if $\Delta''_J(1)\neq 0$, then the knot $J$ satisfies the cosmetic surgery conjecture; see \cite[Proposition 5.1]{boyerlines}. The second is an obstruction of Ni--Wu, which says that if $\tau(J)\neq 0$, then $J$ satisfies the cosmetic surgery conjecture; see \cite[Theorem 1.2]{niwu}.



\subsection{Thin companions}\label{ssec:thin-comp}

In this subsection, we prove Theorem~\ref{thm:csc} in the case of thin companions. 

We break the argument into cases that depend on $n$ and $K$. In each case, we begin by computing the $\delta$-thickness $\mathrm{th}(Q_n(K))$. We then combine the thickness values with the genus values from Theorem~\ref{thm:g3}, and check whether Inequality~\ref{eqn:csc2} holds. In the isolated cases where Inequality~\ref{eqn:csc2} does not hold, we use other methods to complete the proof.

\subsubsection{The $\delta$-thickness of $Q_n(K)$ when $K$ is thin}\label{sssec:alt-th}
In this subsection, we show that when the companion $K$ is thin,  the $\delta$-thickness of $Q_n(K)$ is as follows. 

\begin{proposition}\label{thm:alt-th}
Suppose the companion $K$ is thin. 
If $K$ is the unknot, then
\begin{equation*}
\mathrm{th}(Q_n(K)) = 
\begin{cases}
                                   -n - 1 & \textrm{if \,} n\leq -1 \\
                                   n & \text{if \,} n\geq 0.                                 
\end{cases}
\end{equation*}
If $K$ is the right-handed trefoil, then
\begin{equation*}
\mathrm{th}(Q_n(K)) = 
\left\{ \begin{array}{ll}
         -n + 1 & \textrm{if \,} n \leq -1\\
         2 & \textrm{if \,} n =0, 1\\
        n+2 & \textrm{if \,} n \geq 2.\end{array} \right.
\end{equation*}
In all other cases, 
\[ 
\mathrm{th}(Q_n(K)) = \begin{cases} 
      2 g(K)-n-1 & \textrm{if \,} n \in (-\infty, -2g(K)] \\
      4 g(K)-2 & \textrm{if \,} n \in [-2g(K)+1, 2 g(K)-2] \\
      2 g(K)+n & \textrm{if \,} n \in [2g(K)-1, \infty).
   \end{cases}
\]
\end{proposition}

\begin{proof}
Throughout, we denote $\tau(K)$ by $\tau$ and $g(K)$ by $g$. Recall from \cite[Lemma 7]{cables} that for a thin knot $K$, the complex $\cfkm(K)$, and hence the module $\cfdhat(X_{K,n})$, is particularly simple. More precisely, there exists a basis $\widetilde{\boldsymbol{\eta}} = \{\widetilde\eta_0, \ldots, \widetilde\eta_{2k}\}$ for $\cfkm(K)$ which is both horizontally and vertically simplified. With respect to that basis, the complex $\cfkm(K)$ decomposes as a direct sum of "squares" and one "staircase" with length-one steps, as in Figure~\ref{fig:thin}. If $\cfkm(K)$ contains squares, then the number of squares with top right corner in any given Alexander degree $a$ is the same as the number of squares with top right corner in Alexander degree $-a$.

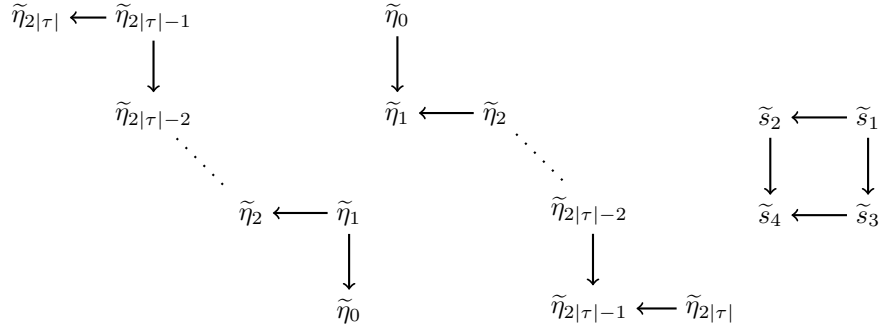
\begin{figure}[h]
    \centering
    {\scalefont{0.8}
    \begin{tikzpicture}[xscale=0.65,yscale=0.65]
      \node at (-0.4,-4) (x1) {$\widetilde{\eta}_{2|\tau|}$};
      \node at (2,-4) (x2) {$\widetilde{\eta}_{2|\tau|-1}$};
      \node at (2,-6) (x3) {$\widetilde{\eta}_{2|\tau|-2}$};
      \node at (4,-8) (x4) {$\widetilde{\eta}_{2}$};
      \node at (6,-8) (x5) {$\widetilde{\eta}_{1}$};
      \node at (6,-10) (x6) {$\widetilde{\eta}_0$}; 
      \draw[algarrow] (x2) to (x1);
      \draw[algarrow] (x2) to  (x3);
      \draw[algarrow] (x5) to  (x4);
      \draw[algarrow] (x5) to (x6);
      \draw[loosely dotted, thick] (x3) to node[right] {} (x4);
    \end{tikzpicture}
    \begin{tikzpicture}[xscale=0.65,yscale=0.65]
      \node at (0,-4) (x1) {$\widetilde{\eta}_{0}$};
      \node at (0,-6) (x2) {$\widetilde{\eta}_{1}$};
      \node at (2,-6) (x3) {$\widetilde{\eta}_{2}$};
      \node at (4,-8) (x4) {\hspace{-.1cm}$\widetilde{\eta}_{2|\tau|-2}$};
      \node at (4,-10) (x5) {\hspace{-.1cm}$\widetilde{\eta}_{2|\tau|-1}$};
      \node at (6.4,-10) (x6) {$\widetilde{\eta}_{2|\tau|}$}; 
      \draw[algarrow] (x1) to  (x2);
      \draw[algarrow] (x3) to  (x2);
      \draw[algarrow] (x4) to (x5);
      \draw[algarrow] (x6) to  (x5);
      \draw[loosely dotted, thick] (x3) to node[right] {} (x4);
    \end{tikzpicture}	
        \begin{tikzpicture}[xscale=0.65,yscale=0.65]
      \node at (0,-2) (x1) {$\widetilde{s}_{2}$};
      \node at (0,-4) (x2) {$\widetilde{s}_{4}$};
      \node at (2,-4) (x3) {$\widetilde{s}_{3}$};
      \node at (2,-2) (x4) {$\widetilde{s}_{1}$};
      \node at (0, -6.2) (x5) {};
      \draw[algarrow] (x1) to  (x2);
      \draw[algarrow] (x3) to  (x2);
      \draw[algarrow] (x4) to (x1);
      \draw[algarrow] (x4) to  (x3);
    \end{tikzpicture}	
    }	
    \caption{The three types of summands of $\cfkm(K)$ for a thin knot $K$.}
    \label{fig:thin}
  \end{figure}

  \begin{figure}[h]
    \centering
    {\scalefont{0.8}
    \begin{tikzpicture}[xscale=0.58,yscale=0.58]
      \node at (0,0) (x0) {$s_2$};
      \node at (0,-4) (x2) {$s_4$};
      \node at (4,-4) (x1) {$s_3$};
      \node at (0,-2) (k) {$\kappa'$};
      \node at (2,-4) (l) {$\lambda'$}; 
      \node at (4,0) (x3) {$s_1$};
      \node at (2,0) (l1) {$\lambda$};
      \node at (4,-2) (k1) {$\kappa$};
      \draw[algarrow] (x0) to node[left] {$D_1$} (k);
      \draw[algarrow] (x2) to node[left] {$D_{123}$} (k);
      \draw[algarrow] (x1) to node[below] {$D_3$} (l);
      \draw[algarrow] (l) to node[below] {$D_2$} (x2);
      \draw[algarrow] (x3) to node[right] {$D_1$} (k1);
      \draw[algarrow] (x1) to node[right] {$D_{123}$} (k1);
      \draw[algarrow] (x3) to node[above] {$D_3$} (l1);
      \draw[algarrow] (l1) to node[above] {$D_2$} (x0);
    \end{tikzpicture}	
    }	
    \caption{A summand $\mathit{Sq}_a$ of $\cfdhat(X_{K, n})$ corresponding to a square summand of $\cfkm(K)$ with top right corner in Alexander degree $a$.}
    \label{fig:sq}
  \end{figure}

Let $\mathit{Sq}_a$ be a square summand of $\cfdhat(X_{K, n})$ with generators labeled as in Figure \ref{fig:sq} and $A(\widetilde s_1) = a$. We will say that $\mathit{Sq}_a$ is \emph{centered at $a$}. Using the differential computation from Section~\ref{ssec:tensor-d}, we see that the  nontrivial differentials on $\cfahat(\mathcal H)\boxtimes \mathit{Sq}_a$ are given by 
\begin{align*}
\dbox (x_1\boxtimes \lambda) &= x_0\boxtimes s_2 &  \dbox (x_1\boxtimes \lambda') &= x_0\boxtimes s_4 &\\
\dbox (x_3\boxtimes \lambda) &= y_2\boxtimes s_2 & \dbox (x_3\boxtimes \lambda') &= y_2\boxtimes s_4 &\\
\dbox (x_2\boxtimes s_1) &= x_1\boxtimes \kappa & \dbox (x_2\boxtimes s_2) &= x_1\boxtimes \kappa' &\\
\dbox (x_4\boxtimes s_1) &= x_3\boxtimes \kappa & \dbox (x_4\boxtimes s_2) &= x_3\boxtimes \kappa' & \\
\dbox (y_4\boxtimes s_1) &= y_3\boxtimes \kappa & \dbox (y_4\boxtimes s_2) &= y_3\boxtimes \kappa' &\\
\dbox (y_3\boxtimes \lambda) &= y_1\boxtimes \kappa'. &  &
\end{align*}

 Using the values from Table~\ref{tab:tensor-gr}, we compute the relative $\delta$-gradings of the generators of $H_{\ast}(\cfahat(\mathcal H)\boxtimes \mathit{Sq}_a)$ in terms of $\tau$, $n$, and $a$. For example, the second row of Table~\ref{tab:tensor-gr} gives 
 \[\delta_{\textrm{rel}}(x_2\boxtimes s) = M(\tilde s) - A(\tilde s) + n+ 1,\]
 so we get
  \[\delta_{\textrm{rel}}(x_2\boxtimes s_3) = M(\widetilde s_3) - A(\widetilde s_3) + n+ 1 = (M(\widetilde s_1)-1) - (A(\widetilde s_1)-1) + n+ 1.\]
  Since $A(\widetilde s_1) = a$ and  $M(\widetilde s_1) = a - \tau$, this simplifies to 
    \[\delta_{\textrm{rel}}(x_2\boxtimes s_3) = - \tau + n+ 1.\]
    We list all generators of $H_{\ast}(\cfahat(\mathcal H)\boxtimes \mathit{Sq}_a)$ and their relative $\delta$ degrees in Table~\ref{tab:sq-deg}.
\begin{table}[h]
\begin{center}
  \begin{tabular}{ | c | c || c | c || c | c | }
    \hline
    Generator & $\delta_\textrm{rel}$ & Generator & $\delta_\textrm{rel}$  & Generator & $\delta_\textrm{rel}$\\ 
    \hline
    \hline
    $x_2\boxtimes s_3$  & $-\tau+n+1$ & $x_5\boxtimes \kappa$  & $- \tau + n+ 1$  & $x_5\boxtimes \lambda$  & $- \tau - 2a$ \\ \hline
    $x_2\boxtimes s_4$  & $-\tau+n+1$ &  $y_5\boxtimes \kappa$ & $- \tau + n+ 1$ &  $y_5\boxtimes \lambda$  & $-\tau - 2a$  \\ \hline
    $x_0\boxtimes s_1$  & $-\tau-2a$ & $x_6\boxtimes \kappa$  & $-\tau + 2a$ &  $x_6\boxtimes \lambda$  & $-\tau + n+ 1$ \\ \hline
    $x_0\boxtimes s_3$  & $-\tau-2a+2$ & $y_6\boxtimes \kappa$  & $-\tau + 2a$ &  $y_6\boxtimes \lambda$   &  $- \tau + n+ 1$ \\ \hline
    $x_4\boxtimes s_3$  & $-\tau+ 2a$ & $x_5\boxtimes \kappa'$  & $-\tau + n+ 1$  & $x_5 \boxtimes \lambda'$ & $-\tau-2a+2$ \\ \hline
    $x_4\boxtimes s_4$  & $-\tau+2a+2$ & $y_5\boxtimes \kappa'$  & $-\tau + n+ 1$  & $y_5 \boxtimes \lambda'$  & $-\tau-2a+2$ \\ \hline
    $y_2\boxtimes s_1$  & $-\tau + n+1$ & $x_6\boxtimes \kappa'$  & $-\tau +2a+2$ & $x_6 \boxtimes \lambda'$  & $- \tau + n+ 1$ \\ \hline
    $y_2\boxtimes s_3$  & $-\tau + n+1$ & $y_6\boxtimes \kappa'$  & $-\tau +2a+2$ &  $y_6 \boxtimes \lambda'$& $- \tau + n+ 1$ \\ \hline
    $y_4\boxtimes s_3$  & $-\tau +2a$ & $y_1\boxtimes\kappa$ &  $- \tau + n+ 1$  & $y_1\boxtimes\lambda$ &  $-\tau-2a$   \\ \hline
    $y_4\boxtimes s_4$  & $-\tau +2a+2$ &  &  & $y_1\boxtimes\lambda'$ & $-\tau-2a+2$  \\ \hline
      &  &  &  & $y_3\boxtimes\lambda'$ & $-\tau + n +2$ \\ \hline
  \end{tabular}
\end{center}
 \caption{The generators of $H_{\ast}(\cfahat(\mathcal H)\boxtimes \mathit{Sq}_a)$ and their relative $\delta$ degrees. Here $\mathit{Sq}_a$  corresponds to a square summand of $\cfkm(K)$ with top-right corner in Alexander degree $a$, and $K$ is a thin knot with $\tau(K) =\tau$.}
  \label{tab:sq-deg}
\end{table}

By symmetry, for every square $\mathit{Sq}$ centered at $a$, there is a square  $\mathit{Sq'}$  centered at $-a$. Table~\ref{tab:sq-deg} shows that $\mathit{Sq}$ is supported in the following six or fewer relative $\delta$ degrees: $- \tau + n+ 1$, $- \tau + n+ 2$, $-\tau-2a$, $-\tau-2a+2$, $-\tau+2a$, and $-\tau+2a+2$; changing $a$ to $-a$, we see that $\mathit{Sq'}$ is supported in the same relative $\delta$ degrees as $\mathit{Sq}$. Hence, to analyze the thickness of $Q_n(K)$, it is enough to consider the squares of $\cfkm(K)$ centered at nonnegative Alexander degrees $a$. 
Now, for $a\geq 0$, we have
\begin{align*}
-\tau-2a &< -\tau-2a+2\leq -\tau+2a+2,\\
-\tau-2a &\leq -\tau+2a < -\tau+2a+2,
\end{align*}
so the minimum relative $\delta$ degree is in the set $\{- \tau + n+ 1, - \tau -2a\}$, and the maximum is in the set $\{- \tau + n+ 2, - \tau +2a+2\}$. Further, if $0\leq a' \leq a$, we have 
\begin{align*}
-\tau-2a &\leq -\tau-2a',\\
-\tau+2a'+2 &\leq -\tau+2a+2,
\end{align*}
so the relative $\delta$ degrees resulting from a square centered at $a'$ are bounded by the minimum and maximum relative $\delta$ degrees resulting from a square centered at $a$. Thus, to analyze the thickness of $Q_n(K)$, it is in fact enough to only consider a highest-centered square of $\cfkm(K)$. Let $A$ be the highest Alexander degree at which a square is centered for our fixed thin knot $K$. Table~\ref{tab:sq-min-max} summarizes the minimum and maximum relative $\delta$ degrees following from the above discussion, depending on the framing $n$ relative to $A$. 
\begin{table}[h]
\begin{center}
  \begin{tabular}{ | c || c | c | c | c | c | }
    \hline
     & $n\leq -2A-2$ & $n = -2A-1$ & $n\in [-2A, 2A-1]$  & $n=2A$ & $n\geq 2A+1$\\ 
    \hline
    \hline
   Min $\delta_\textrm{rel}$  & $-\tau+n+1$ & \makecell{$-\tau+n+1$\\$ = -\tau - 2A$} & $-\tau -2A$  & $- \tau -2A$  & $- \tau - 2A$ \\ \hline
 Max $\delta_\textrm{rel}$   & $-\tau+2A+2$ &  $-\tau+2A+2$ & $-\tau+2A+2$ &  \makecell{$-\tau+2A+2$ \\ $=-\tau+ n+ 2$}  & $-\tau + n+ 2$  \\ \hline
  \end{tabular}
\end{center}
 \caption{Minimum and maximum relative $\delta$ degrees for the generators of $H_{\ast}(\cfahat(\mathcal H)\boxtimes \cfdhat(X_{K, n})) \cong \hfkhat(Q_n(K))$ coming from squares of $\cfkm(K)$ when $K$ is a thin knot with $\tau(K) = \tau$, and $A$ is the highest Alexander degree at which there is a square centered. }
  \label{tab:sq-min-max}
\end{table}

Minimum and maximum relative $\delta$ degrees for the generators of $H_{\ast}(\cfahat(\mathcal H)\boxtimes \cfdhat(X_{K, n})) \cong \hfkhat(Q_n(K))$ coming from the staircase summand of $\cfkm(K)$ are obtained  like in Table \ref{tab:sq-deg}: we first calculate the generators of $\cfahat(\mathcal H)\boxtimes \cfdhat(X_{K, n})$ that survive in homology using the differential formulas from Section~\ref{ssec:tensor-d}, then we compute their relative $\delta$ degrees using the formulas in Table~\ref{tab:tensor-gr}.

For example, consider the case when $\tau>0$ and $n <2\tau$.  Then the staircase summand of $\cfkm(K)$ is given by the left  diagram in Figure~\ref{fig:thin}. We denote the generators of the type $D$ structure $\cfdhat(X_{K, n})$ associated to the staircase summand of $\cfkm(K)$ as follows.  Notation for generators of $\cfdhat(X_{K, n})$ in idempotent $\iota_0$ agrees with Figure~\ref{fig:thin}, i.e.\ we use $\eta_i$ for the generator corresponding to $\widetilde \eta_i$.  Note that the generator $\eta_{2\tau}$ in $\cfdhat(X_{K, n})$ corresponds to the generator $\xi_0$ in Section \ref{sec:prelim}. The generators in idempotent $\iota_1$ coming from horizontal (resp.\ vertical) arrows are labeled $\lambda^1_1, \ldots, \lambda_1^{\tau}$ (resp.\ $\kappa_1^1, \ldots, \kappa_1^{\tau}$) in the order they appear when traversing the staircase starting at $\eta_0$. The generators  in idempotent $\iota_1$ along the unstable chain are denoted as usual.  With this labeling, we have $A(\widetilde\eta_i) = -\tau + i$  and $M(\widetilde \eta_i) = -2\tau + i$. Using these values to simplify the formulas from Table~\ref{tab:tensor-gr}, and computing the differential of $\cfahat(\mathcal H)\boxtimes \cfdhat(X_{K, n})$, we obtain a list of generators of $H_{\ast}(\cfahat(\mathcal H)\boxtimes \cfdhat(X_{K, n})) \cong \hfkhat(Q_n(K))$ coming from the staircase summand of $\cfkm(K)$ and their relative $\delta$ degrees, see Table~\ref{tab:st-deg}. When $n\leq -2\tau$, the lowest relative $\delta$ degree is $-\tau+n+1$ and one example of a homology generator in this degree is $x_5\boxtimes\mu_1$, whereas the highest relative $\delta$ degree is $\tau$ and one example of a homology generator in this degree is $x_0\boxtimes\eta_0$. When $-2\tau + 1\leq n\leq 2\tau-2$,  the lowest relative $\delta$ degree is $-3\tau+2$ and one example of a homology generator in this degree is $x_4\boxtimes\eta_0$, whereas the highest relative $\delta$ degree is $\tau$ and one example of a homology generator in this degree is $x_0\boxtimes\eta_0$. Finally, assume $n =2\tau -1$. Then the lowest relative $\delta$ degree is $-3 \tau +2$ achieved at generator $x_4\boxtimes\eta_0$, and the highest relative $\delta$ degree is 1 achieved at $x_0 \boxtimes \eta_0$ when $\tau =1$, and $\tau +1$ achieved at $y_3 \boxtimes \lambda^1_1$ when $\tau \geq 2$.

\begin{table}[h]
\begin{center}
  \begin{tabular}{ | c | c || c | c   | }
    \hline
    Generator & $\delta_\textrm{rel}$ & Generator & $\delta_\textrm{rel}$  \\ 
    \hline
    \hline
     $x_2 \boxtimes \eta_{2i}, i \in \{0, \ldots, \tau-1\}$   & $-\tau+n+1$ & $y_1\boxtimes \lambda_1^i, i \in \{1, \ldots, \tau\}$  & $\tau-4i+2$   \\ \hline
    $x_4 \boxtimes \eta_{2i}, i \in \{0, \ldots, \tau-1\}$  & $-3\tau+4i+2$ &  $y_1\boxtimes \kappa_1^i, i \in \{1, \ldots, \tau\}$ & $- \tau + n+ 1$     \\ \hline
    $y_4 \boxtimes \eta_{2i}, i \in \{0, \ldots, \tau-1\}$  & $-3\tau+4i+2$ & $y_3\boxtimes \lambda_1^i, i \in \{1, \ldots, \tau -1 \}$    & $-\tau + n + 2$  \\ \hline
    $x_0 \boxtimes \eta_{2i-1}, i \in \{1, \ldots, \tau\}$   & $\tau-4i+2$ & $x_5\boxtimes \lambda_1^i, i \in \{1, \ldots, \tau\}$   & $\tau-4i+2$  \\ \hline
   $x_0 \boxtimes \eta_{0}$   & $\tau$ & $y_5\boxtimes \lambda_1^i, i \in \{1, \ldots, \tau\}$   & $\tau-4i+2$  \\ \hline
    $y_2 \boxtimes \eta_{2i-1}, i \in \{1, \ldots, \tau\}$   & $-\tau + n+1$   & $x_5\boxtimes \kappa_1^i, i \in \{1, \ldots, \tau\}$ & $- \tau + n+ 1$  \\ \hline
     $y_2 \boxtimes \eta_{0}$   & $-\tau + n+ 1$ & $y_5\boxtimes \kappa_1^i, i \in \{1, \ldots, \tau\}$ & $- \tau + n+ 1$   \\ \hline  
          $x_6\boxtimes \lambda_1^i, i \in \{1, \ldots, \tau\}$  & $- \tau  + n + 1$ & $x_6\boxtimes \kappa_1^i, i \in \{1, \ldots, \tau\}$ & $-3\tau + 4i-2$   \\ \hline  
               $y_6\boxtimes \lambda_1^i, i \in \{1, \ldots, \tau\}$  & $- \tau  + n+ 1$ & $y_6\boxtimes \kappa_1^i, i \in \{1, \ldots, \tau\}$ & $-3\tau + 4i-2$   \\ \hline  
     $x_1\boxtimes \mu_j, j \in \{2, \ldots, 2\tau-n\}$  & $- \tau  +j+ n- 1$ & $y_1\boxtimes \mu_j, j \in \{2, \ldots, 2\tau-n\}$ & $- \tau +j+n$   \\ \hline 
          $x_3\boxtimes \mu_j, j \in \{2, \ldots, 2\tau-n\}$  & $\tau  -j+ 2$ & $y_3\boxtimes \mu_j, j \in \{2, \ldots, 2\tau-n\}$ & $\tau -j+2$   \\ \hline 
       $x_5\boxtimes \mu_j, j \in \{1, \ldots, 2\tau-n\}$  & $-\tau  +j+ n$ & $y_5\boxtimes \mu_j, j \in \{1, \ldots, 2\tau-n\}$ & $-\tau + j+n$   \\ \hline 
        $x_6\boxtimes \mu_j, j \in \{1, \ldots, 2\tau-n\}$  & $\tau  -j+ 1$ & $y_6\boxtimes \mu_j, j \in \{1, \ldots, 2\tau-n\}$ & $\tau -j+1$   \\ \hline 
  \end{tabular}
\end{center}
 \caption{The generators of $H_{\ast}(\cfahat(\mathcal H)\boxtimes \cfdhat(X_{K, n})) \cong \hfkhat(Q_n(K))$ coming from the staircase summand of $\cfkm(K)$ and their relative $\delta$ degrees, when $\tau>0$ and $n < 2\tau$.} 
  \label{tab:st-deg}
\end{table}
The other cases can be worked out similarly. We summarize the results in three tables below, also listing the minimum and maximum relative $\delta$ degrees coming from squares in each case. Note that the $n < 2\tau$ case discussed above falls into the first  four columns of Table~\ref{tab:alt-min-max1} and  all columns of Table~\ref{tab:alt-min-max2}, when $\tau$ is positive.

When $0< |\tau|<g$, the highest Alexander degree for a staircase generator of $\cfkm(K)$ is $|\tau|$, and the highest overall Alexander degree of a generator of $\cfkm(K)$ is $g$. Hence, the highest Alexander degree is attained by a square generator, so there is at least one square, and $A=g-1$. The highest and lowest relative $\delta$ degrees of $\hfkhat(Q_n(K))$ arising from tensoring $\cfahat(\mathcal H)$ with squares and with the staircase are summarized in Table~\ref{tab:alt-min-max1}. For all values of $n$, the staircase degrees are bounded by the highest-square degrees, so the thickness of $Q_n(K)$ is the difference between the extremal square degrees; see the last row of Table~\ref{tab:alt-min-max1} for $\mathrm{th}(Q_n(K))$.  
\begin{savenotes}
\begin{table}[h]
\begin{center}
\resizebox{\textwidth}{!}{
  \begin{tabular}{ | c || c | c | c | c | c | }
    \hline
  $n\in$   & $\left(-\infty, -2g\right]$ & $\left[-2g+1, -2|\tau|\right]$ & $\left[-2|\tau|+1, 2|\tau|-2\right]$  & $ \left[2|\tau|-1, 2g-2\right]$ & $\left[2g-1, \infty\right)$\\ 
    \hline
    \hline
\makecell{Min $\delta_\textrm{rel}$ from\\ squares}& $-\tau+n+1$ & $-\tau-2g+2$ & $-\tau -2g+2$  & $- \tau -2g+2$  & $- \tau - 2g+2$ \\ \hline
\makecell{Min $\delta_\textrm{rel}$  from\\ staircase}& $-\tau+n+1$ &  $-\tau+n+1$ & $-\tau -2|\tau|+2$  & $-\tau -2|\tau|+2$  & $-\tau -2|\tau|+2$ \\ \hline
\hline
\makecell{Max $\delta_\textrm{rel}$ from\\ squares}  & $-\tau+2g$ &  $-\tau+2g$ & $-\tau+2g$ &  $-\tau+2g$ & $-\tau + n+ 2$  \\ \hline
\makecell{Max $\delta_\textrm{rel}$ from \\staircase}  & $-\tau +2|\tau|$ &  $-\tau +2|\tau|$ & $-\tau +2|\tau|$ & $-\tau + n+ 2$\footnote{Except when $\tau=1$ and $n = 1$ the maximum $\delta_{\mathrm{rel}}$ degree coming from the staircase is $1$, not $2$.}   & $-\tau + n+ 2$  \\ \hline
\hline
$\mathrm{th}(Q_n(K))$   & $2g-n-1$ & $4g-2$ & $4g-2$  & $4g-2$ & $2g+n$\\ 
\hline
  \end{tabular}
  }
\end{center}
 \caption{Minimum and maximum relative $\delta$ degrees for the  generators of $H_{\ast}(\cfahat(\mathcal H)\boxtimes \cfdhat(X_{K, n})) \cong \hfkhat(Q_n(K))$ when $K$ is thin with $\tau(K) = \tau\neq 0$ and $g(K)=g>|\tau|$, along with the resulting thickness $\mathrm{th}(Q_n(K))$.}
  \label{tab:alt-min-max1}
\end{table}
\end{savenotes}

When $0< |\tau|=g$, there may or may not be squares in $\cfkm(K)$. If there are squares, the highest one is centered at some $A\in \{0, \ldots,  g-1\}$.
The highest and lowest relative $\delta$ degrees of $\hfkhat(Q_n(K))$ arising from tensoring $\cfahat(\mathcal H)$ with squares and with the staircase are summarized in Table~\ref{tab:alt-min-max2}. For all values of $n$, except when $\tau=1$ and $n=1$,  the square degrees are bounded by the extremal staircase degrees, so the thickness of $Q_n(K)$ is the difference between the extremal staircase degrees. When  $\tau=1$ and $n=1$,  the maximum relative $\delta$ degree coming from the staircase is $1$, and the maximum relative $\delta$ degree coming from squares, if there are any, is $2$. The resulting thickness $\mathrm{th}(Q_n(K))$ is then $2$ if there are no squares, or $3$ if there are squares. See the last row of Table~\ref{tab:alt-min-max2} for $\mathrm{th}(Q_n(K))$.
\begin{savenotes}
\begin{table}[h]
\begin{center}
\resizebox{\textwidth}{!}{
  \begin{tabular}{ | c || c | c | c | c | c | }
    \hline
  $n\in$   & $\left(-\infty, -2g\right]$ & $\left[-2g+1, -2A-1\right]$ & $\left[-2A, 2A-1\right]$  & $ \left[2A, 2g-2\right]$ & $\left[2g-1, \infty\right)$\\ 
    \hline
    \hline
\makecell{Min $\delta_\textrm{rel}$  from\\ squares}& $-\tau+n+1$ & $-\tau+n+1$ & $-\tau -2A$  & $- \tau -2A$  & $- \tau - 2A$ \\ \hline
\makecell{Min $\delta_\textrm{rel}$  from\\ staircase}& $-\tau+n+1$ &  $-\tau -2g+2$ & $-\tau -2g+2$  & $-\tau -2g+2$  & $-\tau -2g+2$ \\ \hline
\hline
\makecell{Max $\delta_\textrm{rel}$ from \\ squares}  & $-\tau+2A+2$ &  $-\tau+2A+2$ & $-\tau+2A+2$ &  $-\tau+n+2$ & $-\tau + n+ 2$  \\ \hline
\makecell{Max $\delta_\textrm{rel}$ from \\ staircase}  & $-\tau +2g$ &  $-\tau +2g$ & $-\tau +2g$ & $-\tau +2g$   & $-\tau + n+ 2$ \footnote{Except when $\tau=1$ and $n = 1$ the maximum $\delta_{\mathrm{rel}}$ degree coming from the staircase is $1$, not $2$.} \\ \hline
\hline
$\mathrm{th}(Q_n(K))$   & $2g-n-1$ & $4g-2$ & $4g-2$  & $4g-2$ & $2g+n$ \footnote{Except when $\tau=1$, $n = 1$ and there are no squares, and hence by \cite{hw} $K$ is the right-handed trefoil, the value of $\mathrm{th}(Q_n(K))$ is $2$, not $3$.}\\ 
\hline
  \end{tabular}
  }
\end{center}
 \caption{Minimum and maximum relative $\delta$ degrees for the generators of $H_{\ast}(\cfahat(\mathcal H)\boxtimes \cfdhat(X_{K, n})) \cong \hfkhat(Q_n(K))$ when $K$ is thin with $\tau(K) = \tau\neq 0$ and $g(K)=g=|\tau|$, along with the resulting thickness $\mathrm{th}(Q_n(K))$.}
  \label{tab:alt-min-max2}
\end{table}
\end{savenotes}

Last, we consider the case $\tau = 0$. The staircase of $\cfkm(K)$ consists of just one element $\widetilde\eta_0$, and the corresponding summand of $\cfdhat(X_{K,n})$ consists only of the unstable chain, which starts and ends at the corresponding element $\eta_0$. Analysis analogous to the above yields the extremal $\delta_{\mathrm{rel}}$ degrees  for the staircase summand of $\cfdhat(X_{K,n})$ listed in Table~\ref{tab:alt-min-max-3}. To compute the thickness of $Q_n(K)$, we need to also consider squares. 

If $g\geq2$ and $\tau=0$, then there are squares, and the highest one is centered at $A = g-1$. For all values of $n$, the staircase degrees are bounded by the highest-square degrees, so the thickness of $Q_n(K)$ is the difference between the extremal square degrees. See Table~\ref{tab:alt-min-max-3}.

\begin{table}[h]
\begin{center}
  \begin{tabular}{ | c || c | c | c | c |}
    \hline
$n\in$  & $\left(-\infty, -2g\right]$ & $\left[-2g+1, -1\right]$ & $\left[0, 2g-2\right]$ & $\left[ 2g-1, \infty\right)$\\ 
    \hline
    \hline
\makecell{Min $\delta_\textrm{rel}$  from\\ squares}& $n+1$ & $-2g+2$ & $-2g+2$ & $-2g+2$ \\ \hline
\makecell{Min $\delta_\textrm{rel}$  from\\ staircase}& $n+1$ &  $n+1$ & $2$ & $2$ \\ \hline
\hline
\makecell{Max $\delta_\textrm{rel}$ from \\ squares}  & $2g$ &  $2g$ & $2g$ & $n+2$ \\ \hline
\makecell{Max $\delta_\textrm{rel}$ from \\ staircase}  & $0$ &  $0$ & $n+ 2$ & $n+ 2$ \\ \hline
\hline
$\mathrm{th}(Q_n(K))$   & $2g-n-1$ & $4g-2$ & $4g-2$  & $2g+n$  \\ 
\hline
  \end{tabular}
\end{center}
 \caption{Minimum and maximum relative $\delta$ degrees for the generators of $H_{\ast}(\cfahat(\mathcal H)\boxtimes \cfdhat(X_{K, n})) \cong \hfkhat(Q_n(K))$ when $K$ is thin with $\tau(K)=0$ and $g(K)=g\geq 2$, along with the resulting thickness $\mathrm{th}(Q_n(K))$.}
  \label{tab:alt-min-max-3}
\end{table}

If $g=1$ and $\tau=0$, then there are squares, all centered at $A=0$. Combining the staircase $\delta_{\mathrm{rel}}$ values from Table~\ref{tab:alt-min-max-3} (which do not depend on any assumptions for $g(K)$) with the square values from Table~\ref{tab:sq-min-max}, we see that the staircase degrees are bounded by the extremal square degrees.  Then thickness of $Q_n(K)$ is the difference between the extremal square degrees. When $n\leq -1$, $\mathrm{th}(Q_n(K)) = 1-n$; when $n\geq 0$, $\mathrm{th}(Q_n(K)) = n+2$.

If $g=0$, then $K$ is the unknot. The staircase values from Table~\ref{tab:alt-min-max-3} are all we need to compute $\mathrm{th}(Q_n(K))$.  When $n\leq -2$, $\mathrm{th}(Q_n(K)) = -n-1$; when $n\in\{-1,0\}$, $\mathrm{th}(Q_n(K)) = 0$; when $n\geq 1$, $\mathrm{th}(Q_n(K)) = n$.    

This completes the proof of Proposition~\ref{thm:alt-th}.
\end{proof}

\subsubsection{The cosmetic surgery conjecture for $Q_n(K)$ when $K$ is thin}\label{sssec:alt-csc}

We are now ready to prove Theorem~\ref{thm:csc} in the case of thin companions. Throughout, $g=g(K)$.
\begin{proof}[Proof of Theorem~\ref{thm:csc} for thin companions]

Apart from a few special cases, we will use Inequality~\ref{eqn:csc2}. 

Let $S =  \{(0,-2), (0,-1), (0,0), (0,1), (1,-1), (1,0), (2,-1), (2,0)\}$.
 \begin{casesp}
 \item Assume $(g,n)\notin S$. Theorem~\ref{thm:g3} implies that $\gq\geq 3$. We will combine the thickness values from Proposition~\ref{thm:alt-th} with the genus values from Theorem~\ref{thm:g3} to show that Inequality~\ref{eqn:csc2} holds for all pairs $(g,n)\notin S$.
 
 \begin{casesp}
 
\item Suppose $g=0$, i.e.\ $K$ is the unknot.
\begin{itemize}
\item[-] If $n\leq -3$, then  $\thiq = -n-1$,  $\gq = -n$, and  $f(Q_n(K)) = 2(-n)^2-4(-n)-(-n-1)>0$. 
\item[-] If $n\geq 2$, then $\thiq = n$,  $\gq = n+1$, and  $f(Q_n(K)) = 2(n+1)^2-4(n+1)-n>0$. 
\end{itemize}

 \item \label{case:12} Suppose  $g\geq 1$.
 \begin{itemize}
 \item[-] If $n\leq -2g$, then $\thiq = 2g-n-1$ and $\gq = g-n$.  One can verify (by hand, or plugging into a calculator) that $f(Q_n(K)) = 2(g-n)^2-4(g-n)-(2g-n-1)$ is always positive on the domain $\{(g,n)\in \Z \times \Z \mid g\geq 1, n\leq -2g, (g,n)\notin S \}$. 
 \item[-] If $n\in\left[-2g+1, -1\right]$, then $\thiq = 4g-2$ and $\gq = g-n$. Again one sees that $f(Q_n(K)) = 2(g-n)^2-4(g-n)-(4g-2)$ is always positive.
\item[-] Suppose  $n\in\left[0,2g-2\right]$. 
Then $\thiq = 4g-2$ and $\gq = g+n+1$, so  $f(Q_n(K)) = 2(g+n+1)^2-4(g+n+1)-(4g-2)>0$.
\item[-] Suppose $n\geq 2g-1$. If $K$ is the right handed trefoil, we have $\thiq = 2$ and $\gq =3$, so $f(Q_n(K)) = 4>0$. Otherwise,  $\thiq = 2g+n$ and $\gq =g+n+1$, so $f(Q_n(K)) = 2(g+n+1)^2-4(g+n+1)-(2g+n)>0$. 
 \end{itemize}

Thus, the satellite $Q_n(K)$ satisfies the cosmetic surgery conjecture whenever $K$ is thin and $(g,n)\notin S$.
\end{casesp}

 \item Assume $(g,n)\in S$. We cannot use Inequality~\ref{eqn:csc2} here, so we use \cite[Proposition 5.1]{boyerlines}, \cite[Theorem 1.2]{niwu}, and \cite[Theorem 2]{hanselman}.

\begin{casesp} 
 \item Suppose $g=0$.
 \begin{itemize}
 \item[-] If $(g,n)=(0,-2)$, then $Q_{-2}(K) = 12n121$, so  $\Delta_{12n121}''(1) = 8\neq 0$.
\item[-] If $(g,n) =(0,-1)$, then $Q_{-1}(K)=5_2$, so  $\Delta_{Q_{-1}(K)}''(1) = 4\neq 0$.
\item[-] If $(g,n)=(0,0)$, then $Q_{0}(K)$ is the trivial knot. 
\item[-] If $(g,n) =(0,1)$, then $Q_1(K)= 9_{42}$, so $\Delta_{Q_{1}(K)}''(1) = -4\neq 0$. 
\end{itemize}
Thus, all three nontrivial satellites above satisfy the cosmetic surgery conjecture. 

 \item Suppose $g\geq 1$.
\begin{itemize}
\item[-] \label{deriv} Suppose $(g,n) = (1,-1)$. If $|\tau(K)|=1$, the complex $\cfkm(K)$ consists of one $3$-element staircase and possibly some squares centered at $0$. Thus, 
\[\Delta_K(t) = (s+1)t - (2s+1) + (s+1)t^{-1},\]
where $s\geq 0$ is the number of squares.
Since the $n$-twisted Mazur pattern has winding number $1$ in the solid torus, we have
\[\Delta_{Q_n(K)} (t)= \Delta_{Q_n(U)}(t)\Delta_K(t).\]
Specifically, for $n=-1$ we have $Q_{-1}(U)= 5_2$, so
\[\Delta_{Q_{-1}(K)} (t)= \Delta_{5_2}(t)\Delta_K(t) = (2t-3+2t^{-1})((s+1)t - (2s+1) + (s+1)t^{-1}).\]
So  $\Delta_{Q_{-1}(K)}''(1) = 2s+6$, which is nonzero for any nonnegative $s$. By \cite[Proposition 5.1]{boyerlines}, $Q_{-1}(K)$ satisfies the  cosmetic surgery conjecture.

If $\tau(K)=0$, following the above reasoning, we see that  $Q_{-1}(K)$ satisfies the  cosmetic surgery conjecture whenever $\Delta_K(t) \neq 2t-5+2t^{-1}$. We further obstruct possible cosmetic surgeries using \cite[Theorem 2]{hanselman}. Since here $\frac{\mathrm{th}(Q_{-1}(K))+2g(Q_{-1}(K))}{2g(Q_{-1}(K))(g(Q_{-1}(K))-1)} =\frac{6}{4}<2$, it follows that  $S^3_{\frac{1}{q}}(Q_{-1}(K)) \not \cong S^3_{-\frac{1}{q}}(Q_{-1}(K))$ for $q \geq 2$. Hence, the only pairs of surgery manifolds we cannot distinguish are $\{S^3_{1}(Q_{-1}(K)), S^3_{-1}(Q_{-1}(K))\}$ and $\{S^3_{2}(Q_{-1}(K)), S^3_{-2}(Q_{-1}(K))\}$, where the companion $K$ satisfies $\Delta_K(t) = 2t-5+2t^{-1}$, for example $K=6_1$.

\item[-] \label{case13} 
Suppose $(g,n) = (1,0)$. Similar to the previous case, if $|\tau(K)|=1$, we see that
\[\Delta_K(t) = (s+1)t - (2s+1) + (s+1)t^{-1},\]
where $s\geq 0$.  So
\[\Delta_{Q_{0}(K)} (t)= \Delta_{U}(t)\Delta_K(t) = (s+1)t - (2s+1) + (s+1)t^{-1},\]
and we obtain $\Delta_{Q_{0}(K)}''(1) = 2s+2\neq 0$. 
If $\tau(K)=0$, the complex $\cfkm(K)$ consists of a one-element staircase and $s\geq 1$ squares centered at $0$. Thus, 
\[\Delta_{Q_{0}(K)} (t)=st - (2s+1) + st^{-1},\]
so $\Delta_{Q_{0}(K)}''(1) = 2s\neq 0$. 
Thus,  $Q_0(K)$ satisfies the cosmetic surgery conjecture.

\item[-] Suppose $(g,n) = (2,-1)$, and suppose $S^3_{r}(Q_{-1}(K))\cong S^3_{r'}(Q_{-1}(K))$.  Since $\mathrm{th}(Q_{-1}(K)) = 6$, $g(Q_{-1}(K)) = 3$, and $\frac{\mathrm{th}(Q_{-1}(K))+2g(Q_{-1}(K))}{2g(Q_{-1}(K))(g(Q_{-1}(K))-1)} =1$, we see that $\{r, r'\} = \{\pm 1\}$. Further, by an argument analogous to the one when $(g,n) = (1,-1)$, $\Delta_K(t)$ must be of the form 
\begin{align*}
bt^2 - (4b+2)t + (6b+5) - (4b+2)t^{-1} + bt^{-2} & \quad \textrm{with } b\geq 1, \quad \textrm{ or}\\ 
bt^2 - (4b-2)t + (6b-5) - (4b-2)t^{-1} + bt^{-2} & \quad \textrm{with } b\geq 2, \quad \textrm{ or}\\
(b+1)t^2 - (4b+6)t + (6b+11) - (4b+6)t^{-1} + (b+1)t^{-2}  & \quad \textrm{with } b\geq 0.
\end{align*} 
An example is $K=8_6$.

\item[-]  Suppose $(g,n) = (2,0)$, and suppose $S^3_{r}(Q_{0}(K))\cong S^3_{r'}(Q_{0}(K))$. 
Since $\thiq = 6$, $\gq = 3$, and $\frac{\mathrm{th}(Q_{0}(K))+2g(Q_{0}(K))}{2g(Q_{0}(K))(g(Q_{0}(K))-1)} =1$, we see that $\{r, r'\} = \{\pm 1\}$. Further, combining \cite[Theorem 1.2]{niwu}   with  \cite[Theorem 1.4]{levine-satellite}, we see that $\tau(K)\in\{-1,0\}$. Similar to above, we then compute that $\Delta_K(t)$ is of the form 
\begin{align*}
bt^2 - 4bt + (6b-1) - 4bt^{-1} + bt^{-2} & \quad \textrm{with }  b\geq 1 \textrm{ when } \tau(K)=-1 \quad \textrm{ or}\\ 
bt^2 - 4bt + (6b+1) - 4bt^{-1} + bt^{-2} & \quad \textrm{with } b\geq 1 \textrm{ when } \tau(K)=-1. 
\end{align*}
An example here is $K=m8_{14}$.
 \end{itemize}

\end{casesp}
\end{casesp}

 This completes the proof of Theorem~\ref{thm:csc} for thin companions.
 \end{proof}



\subsection{L-space companions} Recall that an \textit{L-space} is a rational homology sphere $Y$ with the smallest possible Heegaard Floer homology in the sense that $ \textrm{dim}\ \hfhat(Y) = |H_1(Y)|$. Knots that admit nontrivial L-space surgeries are referred to as \textit{L-space knots}. 

In this subsection, we prove the L-space portion of Theorem~\ref{thm:csc}. We start by computing their $\delta$-thickness values.

\subsubsection{$\delta$-thickness values for L-space companions}\label{sec:lspacethickness}

By \cite[Corollary 1.3]{oszlens} and \cite[Corollaries 8 and 9]{hw}, the Alexander polynomial $\Delta_K(t)$ of any L-space knot $K$ takes the form
\begin{equation*}
\Delta_K(t) = t^{-r_0} - t^{-r_1} + \ldots + (-1)^k t^{-r_k} + (-1)^{k+1} + (-1)^k t^{r_k} + \ldots - t^{r_1} + t^{r_0},
\end{equation*}
for some integers $r_0, r_1, \ldots, r_k$ satisfying
\begin{itemize}
\item[-] $r_0 =  g(K)$
\item[-] $ 0 < r_k < \ldots < r_1 < g(K)$
\item[-] If $k=0$, then $K$ is either the unknot or a trefoil knot
\item[-] If $k \geq 1$, then $r_1 = g(K)-1$.
\end{itemize}

Let $\ell_{i} = r_i - r_{i+1}$, for $i \in \{1, \ldots, k-1\}$. With the above notation, we have the following theorem.

\begin{proposition}\label{prop:lspacethickness}
If $K$ is the unknot or a trefoil, then $\mathrm{th}(Q_n(K))$ is given by Proposition \ref{thm:alt-th}. For all other L-space knots $K$, we have that $k \geq 1$. If $k =1$, then
\begin{equation*}
\mathrm{th}(Q_n(K)) = 
\left\{ \begin{array}{ll}
         2g(K)-n-1 & \textrm{if \,} n \leq -2g(K)\\
         4g(K) -2 & \textrm{if \,} n \in [-2g(K)+1, g(K)]\\
      3g(K)+n-2 & \textrm{if \,} n \geq g(K)+1. \end{array} \right.
\end{equation*}
For $k \geq 2$, suppose that $\ell_1 \geq \ell_2 \geq \ldots \geq \ell_{k-1} \geq r_k$. Then
\begin{equation*}
\mathrm{th}(Q_n(K)) = 
\left\{ \begin{array}{ll}
         2g(K)-n-1 & \textrm{if \,} n \leq -2g(K)\\
         4g(K) -2 & \textrm{if \,} n \in [-2g(K)+1, g(K) + r_2]\\
        3g(K)-r_2+n-2 & \textrm{if \,} n \geq g(K) +r_2 +1. \end{array} \right.
\end{equation*}

\end{proposition}

\begin{proof}
Let $K$ be neither the unknot, nor a trefoil, and let $g=g(K)$.

 \begin{figure}[h]
    \centering
    {\scalefont{0.8}
    \begin{tikzpicture}[xscale=0.62,yscale=0.62]
      \node at (0,-4) (x1) {$\widetilde{\xi}_{0}$};
      \node at (2,-4) (x2) {$\widetilde{\omega}_{1}$};
      \node at (2,-6) (x3) {$\widetilde{\omega}_{2}$};
      \node at (4,-8) (x4) {$\widetilde{\omega}_{k-1}$};
      \node at (6,-8) (x5) {$\widetilde{\omega}_{k}$};
      \node at (6,-10) (x6) {$\widetilde{\omega}_{k+1}$}; 
      \node at (8,-10) (x7) {$\widetilde{\theta}_{k}$};
      \node at (8,-12) (x8) {$\widetilde{\theta}_{k-1}$};
      \node at (10,-14) (x9) {$\widetilde{\theta}_{2}$};
      \node at (12,-14) (x10) {$\widetilde{\theta}_{1}$};
      \node at (12,-16) (x11) {$\widetilde{\eta}_{0}$};
      \draw[algarrow] (x2) to node[above] {$1$} (x1);
      \draw[algarrow] (x2) to node[right] {$\ell_1$} (x3);
      \draw[algarrow] (x5) to node[above] {$\ell_{k-1}$} (x4);
      \draw[algarrow] (x5) to node[right] {$r_k$} (x6);
      \draw[algarrow] (x7) to node[above] {$r_k$} (x6);
      \draw[algarrow] (x7) to node[right] {$\ell_{k-1}$} (x8);
      \draw[loosely dotted, thick] (x3) to node[right] {} (x4);
      \draw[loosely dotted, thick] (x8) to node[right] {} (x9);
      \draw[algarrow] (x10) to node[above] {$\ell_{1}$} (x9);
      \draw[algarrow] (x10) to node[right] {$1$} (x11);
    \end{tikzpicture}
    \begin{tikzpicture}[xscale=0.62,yscale=0.62]
      \node at (0,-4) (x1) {$\widetilde{\xi}_{0}$};
      \node at (2,-4) (x2) {$\widetilde{\omega}_{1}$};
      \node at (2,-6) (x3) {$\widetilde{\omega}_{2}$};
      \node at (5,-7) (x4) {$\widetilde{\omega}_{k-1}$};
      \node at (5,-9) (x5) {$\widetilde{\omega}_{k}$};
      \node at (7,-9) (x6) {$\widetilde{\omega}_{k+1}$}; 
      \node at (7,-11) (x7) {$\widetilde{\theta}_{k}$};
      \node at (9,-11) (x8) {$\widetilde{\theta}_{k-1}$};
      \node at (10,-14) (x9) {$\widetilde{\theta}_{2}$};
      \node at (12,-14) (x10) {$\widetilde{\theta}_{1}$};
      \node at (12,-16) (x11) {$\widetilde{\eta}_{0}$};
      \draw[algarrow] (x2) to node[above] {$1$} (x1);
      \draw[algarrow] (x2) to node[right] {$\ell_1$} (x3);
      \draw[algarrow] (x4) to node[right] {$\ell_{k-1}$} (x5);
      \draw[algarrow] (x6) to node[above] {$r_k$} (x5);
      \draw[algarrow] (x6) to node[right] {$r_k$} (x7);
      \draw[algarrow] (x8) to node[above] {$\ell_{k-1}$} (x7);
      \draw[loosely dotted, thick] (x3) to node[right] {} (x5);
      \draw[loosely dotted, thick] (x7) to node[right] {} (x9);
      \draw[algarrow] (x10) to node[above] {$\ell_{1}$} (x9);
      \draw[algarrow] (x10) to node[right] {$1$} (x11);
    \end{tikzpicture}	
    }	
    \caption{$\cfkm(K)$ for L-space knots $K$ that admit positive L-space surgeries. The left staircase is for $k$ odd, while the right staircase is for $k$ even.}
    \label{fig:cfkmlspace}
  \end{figure}

Assume $K$ admits a positive L-space surgery.  By \cite[Theorem 1.2]{oszlens} and \cite[Remark 6.6]{homsconcord}, there exists a basis $\{ \widetilde{\xi}_0, \widetilde{\omega}_1, \ldots, \widetilde{\omega}_k, \widetilde{\omega}_{k+1}, \widetilde{\theta}_k, \ldots, \widetilde{\theta}_1,  \widetilde{\eta}_0 \}$ for $\cfkm(K)$ with respect to which $\cfkm(K)$ looks like a right-handed staircase where the heights and widths of the steps are given by $1, \ell_i , \textrm{and}\ r_k$. See Figure \ref{fig:cfkmlspace}. We first consider the case where $k$ is odd. Then $\cfkm(K)$ is given by the left staircase in Figure~\ref{fig:cfkmlspace}, and the basis elements $\widetilde{\xi}_0$, $\widetilde{\omega}_1$, $\ldots, \widetilde{\omega}_k$, $\widetilde{\omega}_{k+1}$, $\widetilde{\theta}_k$, \ldots, $\widetilde{\theta}_1$,  and $\widetilde{\eta}_0$ have Alexander $A$ and Maslov $M$ gradings given by Table \ref{tab:amlspace}. The Alexander gradings come from the powers of the Alexander polynomial $\Delta_K(t)$. The Maslov gradings for $\widetilde{\xi}_0$ and $\widetilde{\eta}_0$ come from Equation \ref{eqn:gradings}. The rest of the Maslov gradings come from the following facts:
\begin{itemize}
\item[-] If $\widetilde{z}_i \rightarrow \widetilde{z}_{i+1}$ is a vertical arrow in $\cfkm(K)$ of length $\ell_i$, then
\begin{equation}\label{eq:vertMas}
M(\widetilde{z}_{i+1}) = M(\widetilde{z}_i) -1.
\end{equation}
\item[-] If $\widetilde{z}_i \rightarrow \widetilde{z}_{i+1}$ is a horizontal arrow in $\cfkm(K)$ of length $\ell_i$, then 
\begin{equation}\label{eq:horMas}
M(\widetilde{z}_{i+1}) = M(\widetilde{z}_i) +2\ell_i -1.
\end{equation}
\end{itemize}
We will also need for later the analogous relations for the Alexander grading:
\begin{itemize}
\item[-] If $\widetilde{z}_i \rightarrow \widetilde{z}_{i+1}$ is a vertical arrow in $\cfkm(K)$ of length $\ell_i$, then
\begin{equation}\label{eq:vertAlex}
A(\widetilde{z}_{i+1}) = A(\widetilde{z}_i) -\ell_i.
\end{equation}
\item[-] If $\widetilde{z}_i \rightarrow \widetilde{z}_{i+1}$ is a horizontal arrow in $\cfkm(K)$ of length $\ell_i$, then 
\begin{equation}\label{eq:horAlex}
A(\widetilde{z}_{i+1}) = A(\widetilde{z}_i) +\ell_i.
\end{equation}
\end{itemize}

Equations \ref{eq:vertMas} - \ref{eq:horAlex} follow from the definition of vertical and horizontal arrows in \cite[Chapter 11.5]{bfh2}, along with \cite[Equations 11.10 and 11.11]{bfh2}: 
\begin{itemize}
\item[-]If $\widetilde{z}_i \rightarrow \widetilde{z}_{i+1}$ is a vertical arrow in $\cfkm(K)$ of length $\ell_i$, then by definition $\bdy \widetilde{z}_i \equiv \widetilde{z}_{i+1} \mod (U\cdot \cfkm(K))$ and $A(\widetilde{z}_i) - A(\widetilde{z}_{i+1}) = \ell_i$, the former implying $M(\widetilde{z}_{i+1}) = M(\widetilde{z}_i)-1$. 
\item[-]If $\widetilde{z}_i \rightarrow \widetilde{z}_{i+1}$ is a horizontal  arrow in $\cfkm(K)$ of length $\ell_i$, then  by definition $\bdy \widetilde{z}_i  = U^{\ell_i}\widetilde{z}_{i+1}$ and $A(\widetilde{z}_i) = A(U^{\ell_i}\widetilde{z}_{i+1})$, the former implying $M(\widetilde{z}_{i+1})- 2\ell_i = M(U^{\ell_i}\widetilde{z}_{i+1}) = M(\widetilde{z}_i)-1$, and the latter implying $A(\widetilde{z}_i) = A(U^{\ell_i}\widetilde{z}_{i+1}) = A(\widetilde{z}_{i+1})- \ell_i$.
\end{itemize}

\begin{table}[h]
\begin{center}
  \begin{tabular}{ | c | c | c | }
    \hline
    Basis element & $A$ & $M$  \\
    \hline
    \hline
    $\widetilde{\xi}_0$ & $r_0=g$ & 0  \\ \hline
    $\widetilde{\omega}_i, i \in \{1, \ldots, k\} \, \mathrm{odd}$ & $r_i$ & $-1 - 2\sum\limits_{j=1}^{\frac{i-1}{2}} \ell_{2j}$  \\ \hline
    $\widetilde{\omega}_i, i \in \{1, \ldots, k\} \, \mathrm{even}$ & $r_i$ & $-2 - 2\sum\limits_{j=1}^{\frac{i-2}{2}} \ell_{2j} $ \\ \hline
    $\widetilde{\omega}_{k+1}$ & $0$ & $-2 - 2\sum\limits_{j=1}^{\frac{k-1}{2}} \ell_{2j} $ \\ \hline
    $\widetilde{\theta}_i, i \in \{1, \ldots, k\} \, \mathrm{odd}$ & $-r_i$ & $-2g + 1 + 2\sum\limits_{j=1}^{\frac{i-1}{2}} \ell_{2j-1}$ \\ \hline
    $\widetilde{\theta}_i, i \in \{1, \ldots, k\} \, \mathrm{even}$ & $-r_i$ &  $-2g + 2\sum\limits_{j=1}^{\frac{i}{2}} \ell_{2j-1} $  \\ \hline
    $\widetilde{\eta}_0$ & $-r_0=-g$ & $-2g$  \\ \hline
     \end{tabular}
\end{center}
\caption {The Alexander and Maslov gradings of the basis elements $\widetilde{\xi}_0$, $\widetilde{\omega}_i$, $\widetilde{\theta}_i$, and $\widetilde{\eta}_0$ of $\cfkm(K)$, when $K$ admits a positive L-space surgery and $k$ is odd.}
\label{tab:amlspace} 
\end{table}

Now since the basis $\{ \widetilde{\xi}_0, \widetilde{\omega}_1, \ldots, \widetilde{\omega}_k, \widetilde{\omega}_{k+1}, \widetilde{\theta}_k, \ldots, \widetilde{\theta}_1,  \widetilde{\eta}_0 \}$ is horizontally and vertically simplified, the invariant $\cfdhat(X_{K,n})$ is given by Figure \ref{fig:cfdlspace}. Note that in Figure \ref{fig:cfdlspace} we're using slightly different labels for the vertical and horizontal chain generators of $\iota_1\cfdhat(X_{K,n})$ than we had used earlier in the paper. Namely, the superscript of a generator in a given vertical or horizontal chain refers to the starting generator for that chain.

To compute $\mathrm{th}(Q_n(K))$, we need the relative $\delta$-gradings of the generators of $\hfkhat(Q_n(K)) \cong H_{*}(\cfahat(\mathcal{H}) \boxtimes \cfdhat(X_{K,n}))$.

  \begin{figure}[h]
    \centering
    {\scalefont{0.8}
    \begin{tikzpicture}[xscale=0.62,yscale=0.62]
      \node at (0,-4) (x1) {$\xi_{0}$};
      \node at (2,-4) (x2) {$\lambda_{1}^{\omega_1}$};
      \node at (4,-4) (x3) {$\omega_{1}$};
      \node at (4,-6) (x4) {$\kappa_{1}^{\omega_1}$};
      \node at (4,-8) (x7) {$\kappa_{\ell_1}^{\omega_1}$};
      \node at (4,-10) (x8) {$\omega_{2}$};
      \node at (6,-12) (x9) {$\theta_{2}$};
      \node at (8,-12) (x10) {$\lambda_{\ell_1}^{\theta_1}$};
      \node at (10,-12) (x13) {$\lambda_{1}^{\theta_1}$};
      \node at (12,-12) (x14) {$\theta_{1}$};
      \node at (12,-14) (x15) {$\kappa_{1}^{\theta_1}$};
      \node at (12,-16) (x16) {$\eta_{0}$};
      \draw[algarrow] (x2) to node[above] {$\rho_2$} (x1);
      \draw[algarrow] (x3) to node[above] {$\rho_3$} (x2);
      \draw[algarrow] (x3) to node[right] {$\rho_1$} (x4);
      \draw[algarrow] (x8) to node[right] {$\rho_{123}$} (x7);
      \draw[algarrow] (x10) to node[above] {$\rho_{2}$} (x9);
      \draw[algarrow] (x14) to node[above] {$\rho_{3}$} (x13);
      \draw[algarrow] (x14) to node[right] {$\rho_{1}$} (x15);
      \draw[algarrow] (x16) to node[right] {$\rho_{123}$} (x15);
      \draw[loosely dotted, thick] (x10) to node[above] {} (x13);
      \draw[loosely dotted, thick] (x4) to node[right] {} (x7);
      \draw[loosely dotted, thick] (x9) to node[right] {} (x8);
      \draw[unst1] (x1) to node[below] {} (x16);
    \end{tikzpicture}	
    }	
\caption{$\cfdhat(X_{K,n})$ for L-space knots $K$ that admit positive L-space surgeries. The dotted arrow represents the unstable chain.}
\label{fig:cfdlspace}
  \end{figure}
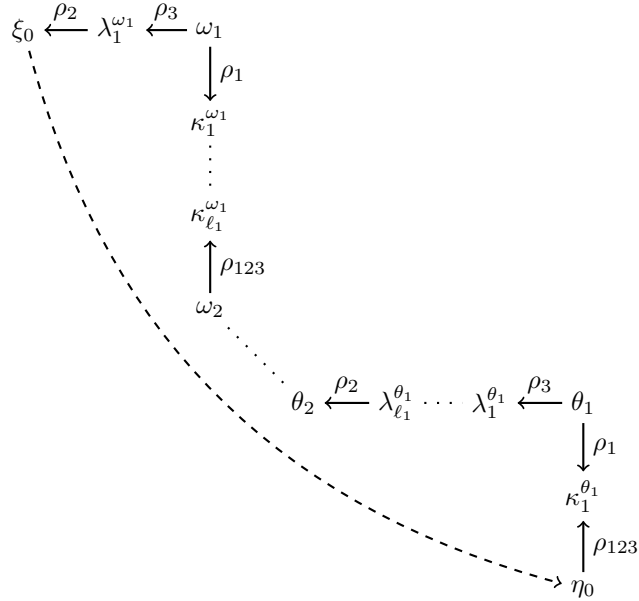

We consider several cases, depending on the framing $n$ relative to $2\tau(K) = 2g$.
\begin{itemize}
\item[-] When $n = 2g$, the unstable chain in $\cfdhat(X_{K,n})$ takes the form
\begin{equation*}
\xi_0  \xrightarrow{\textrm{$D_{12}$}} \eta_0. 
\end{equation*}
Since we're using different notation for the generators of $\cfdhat(X_{K,n})$, below we list in this notation the nontrivial differentials in $\cfahat(\mathcal H)\boxtimes \cfdhat(X_{K,n})$ that were computed in Section~\ref{ssec:tensor-d}:
\begin{align*}
\dbox (x_2\boxtimes \xi_0) &= x_0\boxtimes \eta_0 & \dbox (x_4\boxtimes \xi_0) &= y_2\boxtimes \eta_0  &  &\\
\dbox (x_1\boxtimes \lambda_{1}^{\omega_1}) &= x_0\boxtimes \xi_{0} & \dbox (x_3\boxtimes \lambda_{1}^{\omega_1}) &= y_2\boxtimes \xi_{0} &  &\\
\dbox (x_1\boxtimes \lambda_{r_k}^{\theta_k}) &= x_0\boxtimes \omega_{k+1} &  \dbox (x_3\boxtimes \lambda_{r_k}^{\theta_k}) &= y_2\boxtimes \omega_{k+1} &  &\\
\dbox (x_1\boxtimes \lambda_{l_{i-1}}^{\omega_i}) &= x_0\boxtimes \omega_{i-1} &  \dbox (x_3\boxtimes \lambda_{l_{i-1}}^{\omega_i}) &= y_2\boxtimes \omega_{i-1} & i\in \{3,\ldots, k\}  \mathrm{\, odd} &\\
\dbox (x_1\boxtimes \lambda_{l_{i}}^{\theta_i}) &= x_0\boxtimes \theta_{i+1} &  \dbox (x_3\boxtimes \lambda_{l_{i}}^{\theta_i}) &= y_2 \boxtimes \theta_{i+1} &  i\in \{1,\ldots, k-2\}  \mathrm{\, odd} &\\
\dbox (x_2\boxtimes \omega_{i}) &= x_1\boxtimes \kappa_1^{\omega_i} & \dbox (x_2\boxtimes \theta_{i}) &= x_1\boxtimes \kappa_1^{\theta_i} & i\in\{1,\ldots, k\} \mathrm{\, odd} &\\
\dbox (x_4\boxtimes \omega_i) &= x_3\boxtimes \kappa_1^{\omega_i} & \dbox (x_4\boxtimes \theta_i) &= x_3\boxtimes \kappa_1^{\theta_i} &  i\in\{1,\ldots, k\} \mathrm{\, odd}  & \\
\dbox (y_4\boxtimes \omega_i) &= y_3\boxtimes \kappa_1^{\omega_i} & \dbox (y_4\boxtimes \theta_i) &= y_3\boxtimes \kappa_1^{\theta_i}  &  i\in\{1,\ldots, k\} \mathrm{\, odd}  &
\end{align*}
Thus, the generators not listed above (either as the source or target of the differential $\partial^{\boxtimes}$) generate the homology $H_{\ast}(\cfahat(\mathcal H)\boxtimes \cfdhat(X_{K, n})) \cong \hfkhat(Q_n(K))$. Table \ref{tab:generalcasen=2g} gives these generators of $\hfkhat(Q_n(K))$, together with their relative $\delta$-gradings (which were obtained using Tables \ref{tab:tensor-gr} and \ref{tab:amlspace}).

To get $\mathrm{th}(Q_n(K))$, we need the minimum and maximum relative $\delta$-gradings of all generators in $\hfkhat(Q_n(K))$. The minimal relative $\delta$-grading can be found as follows: first, for each row in Table \ref{tab:generalcasen=2g}, we get the minimal relative $\delta$-grading of all generators lying in that row, then we take the absolute minimum over all of the rows in Table \ref{tab:generalcasen=2g}. The maximum relative $\delta$-grading of all generators in $\hfkhat(Q_n(K))$ can be computed in a similar way.

For concreteness, we now demonstrate how to compute the minimum and maximum relative $\delta$-gradings of the generators of $\hfkhat(Q_n(K))$ of the form $y_3 \boxtimes \lambda_{j}^{\omega_i}$, where $i \in \{3, \ldots, k\} \mathrm{\, odd} \,$ and $j \in \{1, \ldots, \ell_{i-1} \}$ (these are all the generators in Row 10 of Table \ref{tab:generalcasen=2g}). Note that these generators $y_3 \boxtimes \lambda_{j}^{\omega_i}$ only exist when $k \geq 3$. We will need the following two lemmas:

\begin{lemma}\label{lem:lspacethickness1}
For every $i \in \{3, \ldots, k\}$ odd, $M(\widetilde{\omega}_i) - A(\widetilde{\omega}_i) \geq M(\widetilde{\omega}_{i-2}) - A(\widetilde{\omega}_{i-2})$.
\end{lemma}

 \begin{proof}
The generators $\widetilde{\omega}_i$ and $\widetilde{\omega}_{i-1}$ are connected by a horizontal arrow $\widetilde{\omega}_{i-1} \leftarrow \widetilde{\omega}_{i}$ of length $\ell_{i-1}$. By Equations \ref{eq:horMas} and \ref{eq:horAlex}, we have that
\begin{equation}\label{eq:MAdiffhorizontal}
\begin{split}
M(\widetilde{\omega}_{i}) - A(\widetilde{\omega}_{i}) & = M(\widetilde{\omega}_{i-1}) -2\ell_{i-1} +1 - (A(\widetilde{\omega}_{i-1}) - \ell_{i-1}) \\
& = M(\widetilde{\omega}_{i-1}) - A(\widetilde{\omega}_{i-1}) - \ell_{i-1} +1.
\end{split}
\end{equation}
Now consider the generators $\widetilde{\omega}_{i-1}$ and $\widetilde{\omega}_{i-2}$. They are connected by a vertical arrow $\widetilde{\omega}_{i-2} \rightarrow \widetilde{\omega}_{i-1}$ of length $\ell_{i-2}$. By Equations \ref{eq:vertMas} and \ref{eq:vertAlex}, we have that
\begin{equation}\label{eq:MAdiffvertical}
\begin{split}
M(\widetilde{\omega}_{i-1}) - A(\widetilde{\omega}_{i-1})  & = M(\widetilde{\omega}_{i-2}) - 1 - (A(\widetilde{\omega}_{i-2}) - \ell_{i-2}) \\
& = M(\widetilde{\omega}_{i-2}) - A(\widetilde{\omega}_{i-2}) + \ell_{i-2} -1.
\end{split} 
\end{equation}
Plugging Equation \ref{eq:MAdiffvertical} into Equation \ref{eq:MAdiffhorizontal} gives:
\begin{equation*}
\begin{split}
M(\widetilde{\omega}_{i}) - A(\widetilde{\omega}_{i}) & = M(\widetilde{\omega}_{i-2}) - A(\widetilde{\omega}_{i-2}) + \ell_{i-2} -1 - \ell_{i-1} +1 \\
& = M(\widetilde{\omega}_{i-2}) - A(\widetilde{\omega}_{i-2}) + \ell_{i-2} - \ell_{i-1}.
\end{split}
\end{equation*}
By assumption, $\ell_1 \geq \ell_2 \geq \ldots \geq \ell_{k-1}$. This tells us that $\ell_{i-2} - \ell_{i-1} \geq 0$, which means that $M(\widetilde{\omega}_i) - A(\widetilde{\omega}_i) \geq M(\widetilde{\omega}_{i-2}) - A(\widetilde{\omega}_{i-2})$, as desired.
\end{proof}

\begin{lemma}\label{lem:lspacethickness2}
For every $i \in \{5, \ldots, k\}$ odd, $M(\widetilde{\omega}_i) - A(\widetilde{\omega}_i) + \ell_{i-1} \leq M(\widetilde{\omega}_{i-2}) - A(\widetilde{\omega}_{i-2}) + \ell_{i-3}$.
\end{lemma}

\begin{proof}
Using Equations \ref{eq:MAdiffhorizontal} and \ref{eq:MAdiffvertical}, we get that
\begin{equation*}
\begin{split}
M(\widetilde{\omega}_{i}) - A(\widetilde{\omega}_{i}) + \ell_{i-1} & = M(\widetilde{\omega}_{i-1}) - A(\widetilde{\omega}_{i-1}) +1 \\
& =  M(\widetilde{\omega}_{i-2}) - A(\widetilde{\omega}_{i-2}) + \ell_{i-2}.
\end{split}
\end{equation*}
Since $\ell_1 \geq \ell_2 \geq \ldots \geq \ell_{k-1}$, we have that $\ell_{i-3} \geq \ell_{i-2}$. This implies that
\begin{equation*}
M(\widetilde{\omega}_{i}) - A(\widetilde{\omega}_{i}) + \ell_{i-1} \leq M(\widetilde{\omega}_{i-2}) - A(\widetilde{\omega}_{i-2}) + \ell_{i-3},
\end{equation*}
which concludes the proof.
\end{proof}

With Lemmas \ref{lem:lspacethickness1} and \ref{lem:lspacethickness2}, we can now calculate the minimum and maximum relative $\delta$-gradings of the generators of $\hfkhat(Q_n(K))$ of the form $y_3 \boxtimes \lambda_{j}^{\omega_i}$, where $i \in \{3, \ldots, k\} \mathrm{\, odd} \,$ and $j \in \{1, \ldots, \ell_{i-1} \}$:
\begin{equation}\label{absmindelta}
\begin{split}
\min_{\substack{i \in \{3, \ldots, k\} \mathrm{\, odd} \\ j \in \{1, \ldots, \ell_{i-1} \}}} \{\delta_{\mathrm{rel}}(y_3 \boxtimes \lambda_{j}^{\omega_i}) \} & =  \min_{\substack{i}} \{\min_{\substack{j }} \{\delta_{\mathrm{rel}}(y_3 \boxtimes \lambda_{j}^{\omega_i}) \} \} \\
& = \min_{\substack{i}} \{\min_{\substack{j }} \{M(\tilde{\omega}_i) -A(\tilde{\omega}_i) +j + 2g+1 \} \} \\
& = \min_{\substack{i}} \{M(\tilde{\omega}_i) -A(\tilde{\omega}_i) + 2g+2 \} \\
& = M(\tilde{\omega}_3) -A(\tilde{\omega}_3) + 2g+2 \, \, \, \, \textrm{(by Lemma \ref{lem:lspacethickness1})} \\
& = -2r_2 + r_3 + 2g+1 \, \, \, \, \textrm{(by Table \ref{tab:amlspace})}
\end{split}
\end{equation}

\begin{equation}\label{absmaxdelta}
\begin{split}
\max_{\substack{i \in \{3, \ldots, k\} \mathrm{\, odd} \\ j \in \{1, \ldots, \ell_{i-1} \}}} \{\delta_{\mathrm{rel}}(y_3 \boxtimes \lambda_{j}^{\omega_i}) \} & =  \max_{\substack{i}} \{\max_{\substack{j }} \{\delta_{\mathrm{rel}}(y_3 \boxtimes \lambda_{j}^{\omega_i}) \} \} \\
& = \max_{\substack{i}} \{\max_{\substack{j }} \{M(\tilde{\omega}_i) -A(\tilde{\omega}_i) +j + 2g+1 \} \} \\
& = \max_{\substack{i}} \{M(\tilde{\omega}_i) -A(\tilde{\omega}_i) + \ell_{i-1}+ 2g+1 \} \\
& = M(\tilde{\omega}_3) -A(\tilde{\omega}_3) + \ell_2 + 2g+1  \, \, \, \, \textrm{(by Lemma \ref{lem:lspacethickness2})} \\
& = -r_2 + 2g  \, \, \, \, \textrm{(by Table \ref{tab:amlspace})}.
\end{split}
\end{equation}

With a similar approach, we can also calculate the minimum and maximum relative $\delta$-gradings of the generators of $\hfkhat(Q_n(K))$ lying in each of the other rows of Table \ref{tab:generalcasen=2g}.

One can check that when $k$ is odd and $\geq 3$, the minimum over all these minimum relative $\delta$-gradings is $-3g+2$, and the maximum over all these maximum relative $\delta$-gradings is $-r_2 + 2g$. Hence, when $k$ is odd and $\geq 3$, $\mathrm{th}(Q_n(K)) =-r_2 + 5g-2$. In a similar way, one can check that when $k=1$, the minimum relative $\delta$-grading over all the generators of $\hfkhat(Q_n(K))$ is $-3g+2$, and the maximum relative $\delta$-grading over all generators of $\hfkhat(Q_n(K))$ is $2g$. Hence, when $k=1$, $\mathrm{th}(Q_n(K)) =5g-2$. This concludes the $n=2g$ case.

\begin{table}[h]
\begin{center}
\resizebox{\textwidth}{!}{
  \begin{tabular}{ | c | c || c | c | }
    \hline
    Generator & $\delta_{\mathrm{rel}}$ & Generator & $\delta_{\mathrm{rel}}$\\ 
    \hline
    \hline
    
    $y_1 \boxtimes \lambda_{1}^{\omega_1}$ & $-3g+2$ & $y_1 \boxtimes \kappa_{1}^{\theta_1}$ & $g+1$ \\ \hline
    $y_3 \boxtimes \lambda_{1}^{\omega_1}$ & $g+2$ & $x_5 \boxtimes \kappa_{1}^{\theta_1}$ & $g+1$ \\ \hline
    $x_5 \boxtimes \lambda_{1}^{\omega_1}$ & $-3g+2$ & $y_5 \boxtimes \kappa_{1}^{\theta_1}$ & $g+1$ \\ \hline
    $y_5 \boxtimes \lambda_{1}^{\omega_1}$ & $-3g+2$ & $x_6 \boxtimes \kappa_{1}^{\theta_1}$ & $-3g+2$ \\ \hline
    $x_6 \boxtimes \lambda_{1}^{\omega_1}$ & $g+1$ & $y_6 \boxtimes \kappa_{1}^{\theta_1}$ & $-3g+2$ \\ \hline
    $y_6 \boxtimes \lambda_{1}^{\omega_1}$ & $g+1$ & $x_1 \boxtimes \kappa_{j}^{\omega_i}, i \in \{1, \ldots, k-2\} \mathrm{\, odd},\, j \in \{2, \ldots, \ell_{i}\}$ & $M(\tilde{\omega}_i) - A(\tilde{\omega}_i) + j + 2g -1$ \\ \hline
    
    $x_1 \boxtimes \lambda_{j}^{\omega_i}, i \in \{3, \ldots, k\} \mathrm{\, odd}, \, j \in \{1, \ldots, \ell_{i-1} -1\}$ & $M(\tilde{\omega}_i) -3A(\tilde{\omega}_i) -j$ & $y_1 \boxtimes \kappa_{j}^{\omega_i}, i \in \{1, \ldots, k-2\} \mathrm{\, odd}, \, j \in \{1, \ldots, \ell_{i}\}$ & $\delta_\textrm{rel}(x_1 \boxtimes \kappa_{j}^{\omega_i}) +1$ \\ \hline
    $y_1 \boxtimes \lambda_{j}^{\omega_i}, i \in \{3, \ldots, k\} \mathrm{\, odd}, \, j \in \{1, \ldots, \ell_{i-1} \}$ & $\delta_\textrm{rel}(x_1 \boxtimes \lambda_{j}^{\omega_i})+1$ & $x_3 \boxtimes \kappa_{j}^{\omega_i}, i \in \{1, \ldots, k-2\} \mathrm{\, odd}, \, j \in \{2, \ldots, \ell_{i}\}$ & $M(\tilde{\omega}_i) + A(\tilde{\omega}_i) - j + 2$ \\ \hline
    $x_3 \boxtimes \lambda_{j}^{\omega_i}, i \in \{3, \ldots, k\} \mathrm{\, odd}, \, j \in \{1, \ldots, \ell_{i-1} -1\}$ & $M(\tilde{\omega}_i) -A(\tilde{\omega}_i) +j + 2g+1$ & $y_3 \boxtimes \kappa_{j}^{\omega_i}, i \in \{1, \ldots, k-2\} \mathrm{\, odd}, \, j \in \{2, \ldots, \ell_{i}\}$ & $\delta_\textrm{rel}(x_3 \boxtimes \kappa_{j}^{\omega_i})$ \\ \hline
    $y_3 \boxtimes \lambda_{j}^{\omega_i}, i \in \{3, \ldots, k\} \mathrm{\, odd}, \, j \in \{1, \ldots, \ell_{i-1} \}$ & $\delta_\textrm{rel}(x_3 \boxtimes \lambda_{j}^{\omega_i})$ & $x_5 \boxtimes \kappa_{j}^{\omega_i}, i \in \{1, \ldots, k-2\} \mathrm{\, odd}, \, j \in \{1, \ldots, \ell_{i}\}$ & $\delta_\textrm{rel}(y_1 \boxtimes \kappa_{j}^{\omega_i}) $ \\ \hline
    $x_5 \boxtimes \lambda_{j}^{\omega_i}, i \in \{3, \ldots, k\} \mathrm{\, odd}, \, j \in \{1, \ldots, \ell_{i-1} \}$ & $\delta_\textrm{rel}(y_1 \boxtimes \lambda_{j}^{\omega_i})$ & $y_5 \boxtimes \kappa_{j}^{\omega_i}, i \in \{1, \ldots, k-2\} \mathrm{\, odd}, \, j \in \{1, \ldots, \ell_{i}\}$ & $\delta_\textrm{rel}(x_5 \boxtimes \kappa_{j}^{\omega_i}) $ \\ \hline
    $y_5 \boxtimes \lambda_{j}^{\omega_i}, i \in \{3, \ldots, k\} \mathrm{\, odd}, \, j \in \{1, \ldots, \ell_{i-1} \}$ & $\delta_\textrm{rel}(x_5 \boxtimes \lambda_{j}^{\omega_i})$ & $x_6 \boxtimes \kappa_{j}^{\omega_i}, i \in \{1, \ldots, k-2\} \mathrm{\, odd}, \, j \in \{1, \ldots, \ell_{i}\}$ & $\delta_\textrm{rel}(x_3 \boxtimes \kappa_{j}^{\omega_i}) -1$ \\ \hline
    $x_6 \boxtimes \lambda_{j}^{\omega_i}, i \in \{3, \ldots, k\} \mathrm{\, odd}, \, j \in \{1, \ldots, \ell_{i-1} \}$ & $\delta_\textrm{rel}(x_3 \boxtimes \lambda_{j}^{\omega_i})-1$ & $y_6 \boxtimes \kappa_{j}^{\omega_i}, i \in \{1, \ldots, k-2\} \mathrm{\, odd}, \, j \in \{1, \ldots, \ell_{i}\}$ & $\delta_\textrm{rel}(x_6 \boxtimes \kappa_{j}^{\omega_i})$ \\ \hline
    $y_6 \boxtimes \lambda_{j}^{\omega_i}, i \in \{3, \ldots, k\} \mathrm{\, odd}, \, j \in \{1, \ldots, \ell_{i-1} \}$ & $\delta_\textrm{rel}(x_6 \boxtimes \lambda_{j}^{\omega_i})$ & $x_1 \boxtimes \kappa_{j}^{\omega_k}, j \in \{2, \ldots, r_k \}$ & $M(\tilde{\omega}_k) - A(\tilde{\omega}_k) + j + 2g -1$ \\ \hline
    
    $x_1 \boxtimes \lambda_{j}^{\theta_k}, j \in \{1, \ldots, r_k -1\}$ & $M(\tilde{\theta}_k) -3A(\tilde{\theta}_k) -j$ & $y_1 \boxtimes \kappa_{j}^{\omega_k}, j \in \{1, \ldots, r_k \}$ & $\delta_\textrm{rel}(x_1 \boxtimes \kappa_{j}^{\omega_k})+1$ \\ \hline
    $y_1 \boxtimes \lambda_{j}^{\theta_k}, j \in \{1, \ldots, r_k \}$ & $\delta_\textrm{rel}(x_1 \boxtimes \lambda_{j}^{\theta_k})+1$ & $x_3 \boxtimes \kappa_{j}^{\omega_k}, j \in \{2, \ldots, r_k \}$ & $M(\tilde{\omega}_k) + A(\tilde{\omega}_k) - j + 2$ \\ \hline
    $x_3 \boxtimes \lambda_{j}^{\theta_k}, j \in \{1, \ldots, r_k -1\}$ & $M(\tilde{\theta}_k) - A(\tilde{\theta}_k) + j +2g +1$ & $y_3 \boxtimes \kappa_{j}^{\omega_k}, j \in \{2, \ldots, r_k \}$ & $\delta_\textrm{rel}(x_3 \boxtimes \kappa_{j}^{\omega_k})$ \\ \hline
    $y_3 \boxtimes \lambda_{j}^{\theta_k}, j \in \{1, \ldots, r_k \}$ & $\delta_\textrm{rel}(x_3 \boxtimes \lambda_{j}^{\theta_k})$ & $x_5 \boxtimes \kappa_{j}^{\omega_k}, j \in \{1, \ldots, r_k \}$ & $\delta_\textrm{rel}(y_1 \boxtimes \kappa_{j}^{\omega_k})$ \\ \hline
    $x_5 \boxtimes \lambda_{j}^{\theta_k}, j \in \{1, \ldots, r_k \}$ & $\delta_\textrm{rel}(y_1 \boxtimes \lambda_{j}^{\theta_k})$ & $y_5 \boxtimes \kappa_{j}^{\omega_k}, j \in \{1, \ldots, r_k \}$ & $\delta_\textrm{rel}(x_5 \boxtimes \kappa_{j}^{\omega_k})$ \\ \hline
    $y_5 \boxtimes \lambda_{j}^{\theta_k}, j \in \{1, \ldots, r_k \}$ & $\delta_\textrm{rel}(x_5 \boxtimes \lambda_{j}^{\theta_k})$ & $x_6 \boxtimes \kappa_{j}^{\omega_k}, j \in \{1, \ldots, r_k \}$ & $\delta_\textrm{rel}(x_3 \boxtimes \kappa_{j}^{\omega_k})-1$ \\ \hline
    $x_6 \boxtimes \lambda_{j}^{\theta_k}, j \in \{1, \ldots, r_k \}$ & $\delta_\textrm{rel}(x_3 \boxtimes \lambda_{j}^{\theta_k})-1$ & $y_6 \boxtimes \kappa_{j}^{\omega_k}, j \in \{1, \ldots, r_k \}$ & $\delta_\textrm{rel}(x_6 \boxtimes \kappa_{j}^{\omega_k})$ \\ \hline
    $y_6 \boxtimes \lambda_{j}^{\theta_k}, j \in \{1, \ldots, r_k \}$ & $\delta_\textrm{rel}(x_6 \boxtimes \lambda_{j}^{\theta_k})$ & $x_1 \boxtimes \kappa_{j}^{\theta_i}, i \in \{3, \ldots, k\} \mathrm{\, odd},\, j \in \{2, \ldots, \ell_{i-1}\}$ & $M(\tilde{\theta}_i) -A(\tilde{\theta}_i) + j +2g -1$ \\ \hline
    
    $x_1 \boxtimes \lambda_{j}^{\theta_i}, i \in \{1, \ldots, k-2\} \mathrm{\, odd},\, j \in \{1, \ldots, \ell_{i} -1\}$ & $M(\tilde{\theta}_i) -3A(\tilde{\theta}_i) -j$ & $y_1 \boxtimes \kappa_{j}^{\theta_i}, i \in \{3, \ldots, k\} \mathrm{\, odd},\, j \in \{1, \ldots, \ell_{i-1}\}$ & $\delta_\textrm{rel}(x_1 \boxtimes \kappa_{j}^{\theta_i}) + 1$ \\ \hline
    $y_1 \boxtimes \lambda_{j}^{\theta_i}, i \in \{1, \ldots, k-2\} \mathrm{\, odd},\, j \in \{1, \ldots, \ell_{i} \}$ & $\delta_\textrm{rel}(x_1 \boxtimes \lambda_{j}^{\theta_i})+1$ & $x_3 \boxtimes \kappa_{j}^{\theta_i}, i \in \{3, \ldots, k\} \mathrm{\, odd},\, j \in \{2, \ldots, \ell_{i-1}\}$ & $M(\tilde{\theta}_i) + A(\tilde{\theta}_i) - j + 2$ \\ \hline
    $x_3 \boxtimes \lambda_{j}^{\theta_i}, i \in \{1, \ldots, k-2\} \mathrm{\, odd},\, j \in \{1, \ldots, \ell_{i} -1\}$ & $M(\tilde{\theta}_i) -A(\tilde{\theta}_i) +j + 2g +1$ & $y_3 \boxtimes \kappa_{j}^{\theta_i}, i \in \{3, \ldots, k\} \mathrm{\, odd},\, j \in \{2, \ldots, \ell_{i-1}\}$ & $\delta_\textrm{rel}(x_3 \boxtimes \kappa_{j}^{\theta_i})$ \\ \hline
    $y_3 \boxtimes \lambda_{j}^{\theta_i}, i \in \{1, \ldots, k-2\} \mathrm{\, odd},\, j \in \{1, \ldots, \ell_{i} \}$ & $\delta_\textrm{rel}(x_3 \boxtimes \lambda_{j}^{\theta_i})$ & $x_5 \boxtimes \kappa_{j}^{\theta_i}, i \in \{3, \ldots, k\} \mathrm{\, odd},\, j \in \{1, \ldots, \ell_{i-1}\}$ & $\delta_\textrm{rel}(y_1 \boxtimes \kappa_{j}^{\theta_i})$ \\ \hline
    $x_5 \boxtimes \lambda_{j}^{\theta_i}, i \in \{1, \ldots, k-2\} \mathrm{\, odd},\, j \in \{1, \ldots, \ell_{i} \}$ & $\delta_\textrm{rel}(y_1 \boxtimes \lambda_{j}^{\theta_i})$ & $y_5 \boxtimes \kappa_{j}^{\theta_i}, i \in \{3, \ldots, k\} \mathrm{\, odd},\, j \in \{1, \ldots, \ell_{i-1}\}$ & $\delta_\textrm{rel}(x_5 \boxtimes \kappa_{j}^{\theta_i})$ \\ \hline
    $y_5 \boxtimes \lambda_{j}^{\theta_i}, i \in \{1, \ldots, k-2\} \mathrm{\, odd},\, j \in \{1, \ldots, \ell_{i} \}$ & $\delta_\textrm{rel}(x_5 \boxtimes \lambda_{j}^{\theta_i})$ & $x_6 \boxtimes \kappa_{j}^{\theta_i}, i \in \{3, \ldots, k\} \mathrm{\, odd},\, j \in \{1, \ldots, \ell_{i-1}\}$ & $\delta_\textrm{rel}(x_3 \boxtimes \kappa_{j}^{\theta_i}) -1$ \\ \hline
    $x_6 \boxtimes \lambda_{j}^{\theta_i}, i \in \{1, \ldots, k-2\} \mathrm{\, odd},\, j \in \{1, \ldots, \ell_{i} \}$ & $\delta_\textrm{rel}(x_3 \boxtimes \lambda_{j}^{\theta_i}) -1$ & $y_6 \boxtimes \kappa_{j}^{\theta_i}, i \in \{3, \ldots, k\} \mathrm{\, odd},\, j \in \{1, \ldots, \ell_{i-1}\}$ & $\delta_\textrm{rel}(x_6 \boxtimes \kappa_{j}^{\theta_i})$ \\ \hline
    $y_6 \boxtimes \lambda_{j}^{\theta_i}, i \in \{1, \ldots, k-2\} \mathrm{\, odd},\, j \in \{1, \ldots, \ell_{i} \}$ & $\delta_\textrm{rel}(x_6 \boxtimes \lambda_{j}^{\theta_i})$ & $x_2 \boxtimes \theta_i, i \in \{1, \ldots, k\} \mathrm{\, even}$ & $M(\tilde{\theta}_i) -A(\tilde{\theta}_i) + 2g+1$ \\ \hline
    
    $x_2 \boxtimes \omega_i, i \in \{1, \ldots, k\} \mathrm{\, even}$ & $M(\tilde{\omega}_i) -A(\tilde{\omega}_i) + 2g+1$  & $x_4 \boxtimes \theta_i, i \in \{1, \ldots, k\} \mathrm{\, even}$ & $M(\tilde{\theta}_i) + A(\tilde{\theta}_i) + 2$ \\ \hline
    $x_4 \boxtimes \omega_i, i \in \{1, \ldots, k\} \mathrm{\, even}$ & $M(\tilde{\omega}_i) +A(\tilde{\omega}_i) + 2$  & $y_4 \boxtimes \theta_i, i \in \{1, \ldots, k\} \mathrm{\, even}$ & $\delta_\textrm{rel}(x_4 \boxtimes \theta_i)$ \\ \hline
    $y_4 \boxtimes \omega_i, i \in \{1, \ldots, k\} \mathrm{\, even}$ & $\delta_\textrm{rel}(x_4 \boxtimes \omega_i)$  & $x_0 \boxtimes \theta_i, i \in \{1, \ldots, k\} \mathrm{\, odd}$ & $M(\tilde{\theta}_i) - 3A(\tilde{\theta}_i) $ \\ \hline
    
    $x_0 \boxtimes \omega_i, i \in \{1, \ldots, k\} \mathrm{\, odd}$ & $M(\tilde{\omega}_i) -3A(\tilde{\omega}_i) $  & $y_2 \boxtimes \theta_i, i \in \{1, \ldots, k\} \mathrm{\, odd}$ & $M(\tilde{\theta}_i) - A(\tilde{\theta}_i) + 2g +1$ \\ \hline
    $y_2 \boxtimes \omega_i, i \in \{1, \ldots, k\} \mathrm{\, odd}$ & $M(\tilde{\omega}_i) -A(\tilde{\omega}_i) + 2g+1$  & $x_2 \boxtimes \eta_{0}$ & $g+1$ \\ \hline

    $x_2 \boxtimes \omega_{k+1}$ & $M(\tilde{\omega}_{k+1}) -A(\tilde{\omega}_{k+1}) + 2g+1$  & $x_4 \boxtimes \eta_{0}$ & $-3g+2$ \\ \hline
    $x_4 \boxtimes \omega_{k+1}$ & $M(\tilde{\omega}_{k+1}) +A(\tilde{\omega}_{k+1}) + 2$  & $y_4 \boxtimes \eta_{0}$ & $-3g+2$ \\ \hline
    $y_4 \boxtimes \omega_{k+1}$ & $\delta_\textrm{rel}(x_4 \boxtimes \omega_{k+1})$  & $y_4 \boxtimes \xi_{0}$ & $g+2$ \\ \hline
   
  \end{tabular}
  }
\end{center}
\caption {The $\delta$-gradings of the generators of $\hfkhat(Q_n(K))$,  for the case when $K$ admits a positive L-space surgery, $k$ is odd, and $n=2g$.}
\label{tab:generalcasen=2g} 
\end{table}

\item[-] When $n < 2g$, the unstable chain in $\cfdhat(X_{K,n})$ takes the form
\begin{equation*}
\xi_0 \xrightarrow{\textrm{$D_1$}} \mu_1 \xleftarrow{\textrm{$D_{23}$}} \cdots \xleftarrow{\textrm{$D_{23}$}} \mu_{2g-n} \xleftarrow{\textrm{$D_{3}$}} \eta_0. 
\end{equation*}
Similar to the $n =2g$ case, we can compute the generators of $\hfkhat(Q_n(K))$, together with their relative $\delta$-gradings; see Table \ref{tab:generalcasenless2g}. One can verify that
\[ 
\min \, \{ \delta_{\mathrm{rel}}(u \boxtimes v) \mid u \boxtimes v \textrm{\, is a nonzero generator of \,} \hfkhat(Q_n(K)) \} = \begin{cases} 
      n-g+1 & \textrm{if \,} n \in (-\infty, -2g+1] \\
      -3g+2 & \textrm{if \,} n \in  [-2g+1, 2g)
   \end{cases}
\]
and
\[ 
\max \, \{ \delta_{\mathrm{rel}}(u \boxtimes v) \mid u \boxtimes v\textrm{\, is a nonzero generator of \,} \hfkhat(Q_n(K)) \} = \begin{cases} 
      g & \textrm{if \,} n \in (-\infty, g+ r_2]  \\
    -r_2 + n & \textrm{if \,} n \in  [g+ r_2, 2g).
   \end{cases}
\]
Since $-2g +1 <  g+ r_2$, we have that
\[ 
\mathrm{th}(Q_n(K)) = \begin{cases} 
      2g-n-1 & \textrm{if \,} n \in (-\infty, -2g] \\
      4g-2 & \textrm{if \,} n \in [-2g+1, g+ r_2] \\
      3g-r_2+ n-2 & \textrm{if \,} n \in [ g+ r_2 +1, 2g).
   \end{cases}
\]
This concludes the $n<2g$ case.

\begin{table}[h]
\begin{center}
\resizebox{\textwidth}{!}{
  \begin{tabular}{ | c | c || c | c | }
    \hline
    Generator & $\delta_{\mathrm{rel}}$ & Generator & $\delta_{\mathrm{rel}}$\\ 
    \hline
    \hline
    
    $y_1 \boxtimes \lambda_{1}^{\omega_1}$ & $-3g+2$ & $y_1 \boxtimes \kappa_{1}^{\theta_1}$ & $n - g + 1$ \\ \hline
    $x_5 \boxtimes \lambda_{1}^{\omega_1}$ & $-3g+2$ & $x_5 \boxtimes \kappa_{1}^{\theta_1}$ & $n - g + 1$ \\ \hline
    $y_5 \boxtimes \lambda_{1}^{\omega_1}$ & $-3g+2$ & $y_5 \boxtimes \kappa_{1}^{\theta_1}$ & $n - g + 1$ \\ \hline
    $x_6 \boxtimes \lambda_{1}^{\omega_1}$ & $n - g + 1$ & $x_6 \boxtimes \kappa_{1}^{\theta_1}$ & $-3g+2$ \\ \hline
    $y_6 \boxtimes \lambda_{1}^{\omega_1}$ & $n - g + 1$ & $y_6 \boxtimes \kappa_{1}^{\theta_1}$ & $-3g+2$ \\ \hline
    
    $x_1 \boxtimes \lambda_{j}^{\omega_i}, i \in \{3, \ldots, k\} \mathrm{\, odd}, \, j \in \{1, \ldots, \ell_{i-1} -1\}$ & $M(\tilde{\omega}_i) -3A(\tilde{\omega}_i) -j$ & $x_1 \boxtimes \kappa_{j}^{\omega_i}, i \in \{1, \ldots, k-2\} \mathrm{\, odd},\, j \in \{2, \ldots, \ell_{i}\}$ & $M(\tilde{\omega}_i) - A(\tilde{\omega}_i) + j + n -1$ \\ \hline
    $y_1 \boxtimes \lambda_{j}^{\omega_i}, i \in \{3, \ldots, k\} \mathrm{\, odd}, \, j \in \{1, \ldots, \ell_{i-1} \}$ & $\delta_\textrm{rel}(x_1 \boxtimes \lambda_{j}^{\omega_i})+1$ & $y_1 \boxtimes \kappa_{j}^{\omega_i}, i \in \{1, \ldots, k-2\} \mathrm{\, odd}, \, j \in \{1, \ldots, \ell_{i}\}$ & $\delta_\textrm{rel}(x_1 \boxtimes \kappa_{j}^{\omega_i}) +1$ \\ \hline
    $x_3 \boxtimes \lambda_{j}^{\omega_i}, i \in \{3, \ldots, k\} \mathrm{\, odd}, \, j \in \{1, \ldots, \ell_{i-1} -1\}$ & $M(\tilde{\omega}_i) -A(\tilde{\omega}_i) +j + n+1$ & $x_3 \boxtimes \kappa_{j}^{\omega_i}, i \in \{1, \ldots, k-2\} \mathrm{\, odd}, \, j \in \{2, \ldots, \ell_{i}\}$ & $M(\tilde{\omega}_i) + A(\tilde{\omega}_i) - j + 2$ \\ \hline
    $y_3 \boxtimes \lambda_{j}^{\omega_i}, i \in \{3, \ldots, k\} \mathrm{\, odd}, \, j \in \{1, \ldots, \ell_{i-1} \}$ & $\delta_\textrm{rel}(x_3 \boxtimes \lambda_{j}^{\omega_i})$ & $y_3 \boxtimes \kappa_{j}^{\omega_i}, i \in \{1, \ldots, k-2\} \mathrm{\, odd}, \, j \in \{2, \ldots, \ell_{i}\}$ & $\delta_\textrm{rel}(x_3 \boxtimes \kappa_{j}^{\omega_i})$ \\ \hline
    $x_5 \boxtimes \lambda_{j}^{\omega_i}, i \in \{3, \ldots, k\} \mathrm{\, odd}, \, j \in \{1, \ldots, \ell_{i-1} \}$ & $\delta_\textrm{rel}(y_1 \boxtimes \lambda_{j}^{\omega_i})$ & $x_5 \boxtimes \kappa_{j}^{\omega_i}, i \in \{1, \ldots, k-2\} \mathrm{\, odd}, \, j \in \{1, \ldots, \ell_{i}\}$ & $\delta_\textrm{rel}(y_1 \boxtimes \kappa_{j}^{\omega_i}) $ \\ \hline
    $y_5 \boxtimes \lambda_{j}^{\omega_i}, i \in \{3, \ldots, k\} \mathrm{\, odd}, \, j \in \{1, \ldots, \ell_{i-1} \}$ & $\delta_\textrm{rel}(x_5 \boxtimes \lambda_{j}^{\omega_i})$ & $y_5 \boxtimes \kappa_{j}^{\omega_i}, i \in \{1, \ldots, k-2\} \mathrm{\, odd}, \, j \in \{1, \ldots, \ell_{i}\}$ & $\delta_\textrm{rel}(x_5 \boxtimes \kappa_{j}^{\omega_i}) $ \\ \hline
    $x_6 \boxtimes \lambda_{j}^{\omega_i}, i \in \{3, \ldots, k\} \mathrm{\, odd}, \, j \in \{1, \ldots, \ell_{i-1} \}$ & $\delta_\textrm{rel}(x_3 \boxtimes \lambda_{j}^{\omega_i})-1$ & $x_6 \boxtimes \kappa_{j}^{\omega_i}, i \in \{1, \ldots, k-2\} \mathrm{\, odd}, \, j \in \{1, \ldots, \ell_{i}\}$ & $\delta_\textrm{rel}(x_3 \boxtimes \kappa_{j}^{\omega_i}) -1$ \\ \hline
    $y_6 \boxtimes \lambda_{j}^{\omega_i}, i \in \{3, \ldots, k\} \mathrm{\, odd}, \, j \in \{1, \ldots, \ell_{i-1} \}$ & $\delta_\textrm{rel}(x_6 \boxtimes \lambda_{j}^{\omega_i})$ & $y_6 \boxtimes \kappa_{j}^{\omega_i}, i \in \{1, \ldots, k-2\} \mathrm{\, odd}, \, j \in \{1, \ldots, \ell_{i}\}$ & $\delta_\textrm{rel}(x_6 \boxtimes \kappa_{j}^{\omega_i})$ \\ \hline
    
    $x_1 \boxtimes \lambda_{j}^{\theta_k}, j \in \{1, \ldots, r_k -1\}$ & $M(\tilde{\theta}_k) -3A(\tilde{\theta}_k) -j$ & $x_1 \boxtimes \kappa_{j}^{\omega_k}, j \in \{2, \ldots, r_k \}$ & $M(\tilde{\omega}_k) - A(\tilde{\omega}_k) + j + n -1$ \\ \hline
    $y_1 \boxtimes \lambda_{j}^{\theta_k}, j \in \{1, \ldots, r_k \}$ & $\delta_\textrm{rel}(x_1 \boxtimes \lambda_{j}^{\theta_k})+1$ & $y_1 \boxtimes \kappa_{j}^{\omega_k}, j \in \{1, \ldots, r_k \}$ & $\delta_\textrm{rel}(x_1 \boxtimes \kappa_{j}^{\omega_k})+1$ \\ \hline
    $x_3 \boxtimes \lambda_{j}^{\theta_k}, j \in \{1, \ldots, r_k -1\}$ & $M(\tilde{\theta}_k) - A(\tilde{\theta}_k) + j +n +1$ & $x_3 \boxtimes \kappa_{j}^{\omega_k}, j \in \{2, \ldots, r_k \}$ & $M(\tilde{\omega}_k) + A(\tilde{\omega}_k) - j + 2$ \\ \hline
    $y_3 \boxtimes \lambda_{j}^{\theta_k}, j \in \{1, \ldots, r_k \}$ & $\delta_\textrm{rel}(x_3 \boxtimes \lambda_{j}^{\theta_k})$ & $y_3 \boxtimes \kappa_{j}^{\omega_k}, j \in \{2, \ldots, r_k \}$ & $\delta_\textrm{rel}(x_3 \boxtimes \kappa_{j}^{\omega_k})$ \\ \hline
    $x_5 \boxtimes \lambda_{j}^{\theta_k}, j \in \{1, \ldots, r_k \}$ & $\delta_\textrm{rel}(y_1 \boxtimes \lambda_{j}^{\theta_k})$ & $x_5 \boxtimes \kappa_{j}^{\omega_k}, j \in \{1, \ldots, r_k \}$ & $\delta_\textrm{rel}(y_1 \boxtimes \kappa_{j}^{\omega_k})$ \\ \hline
    $y_5 \boxtimes \lambda_{j}^{\theta_k}, j \in \{1, \ldots, r_k \}$ & $\delta_\textrm{rel}(x_5 \boxtimes \lambda_{j}^{\theta_k})$ & $y_5 \boxtimes \kappa_{j}^{\omega_k}, j \in \{1, \ldots, r_k \}$ & $\delta_\textrm{rel}(x_5 \boxtimes \kappa_{j}^{\omega_k})$ \\ \hline
    $x_6 \boxtimes \lambda_{j}^{\theta_k}, j \in \{1, \ldots, r_k \}$ & $\delta_\textrm{rel}(x_3 \boxtimes \lambda_{j}^{\theta_k})-1$ & $x_6 \boxtimes \kappa_{j}^{\omega_k}, j \in \{1, \ldots, r_k \}$ & $\delta_\textrm{rel}(x_3 \boxtimes \kappa_{j}^{\omega_k})-1$ \\ \hline
    $y_6 \boxtimes \lambda_{j}^{\theta_k}, j \in \{1, \ldots, r_k \}$ & $\delta_\textrm{rel}(x_6 \boxtimes \lambda_{j}^{\theta_k})$ & $y_6 \boxtimes \kappa_{j}^{\omega_k}, j \in \{1, \ldots, r_k \}$ & $\delta_\textrm{rel}(x_6 \boxtimes \kappa_{j}^{\omega_k})$ \\ \hline
    
    $x_1 \boxtimes \lambda_{j}^{\theta_i}, i \in \{1, \ldots, k-2\} \mathrm{\, odd},\, j \in \{1, \ldots, \ell_{i} -1\}$ & $M(\tilde{\theta}_i) -3A(\tilde{\theta}_i) -j$ & $x_1 \boxtimes \kappa_{j}^{\theta_i}, i \in \{3, \ldots, k\} \mathrm{\, odd},\, j \in \{2, \ldots, \ell_{i-1}\}$ & $M(\tilde{\theta}_i) -A(\tilde{\theta}_i) + j +n -1$ \\ \hline
    $y_1 \boxtimes \lambda_{j}^{\theta_i}, i \in \{1, \ldots, k-2\} \mathrm{\, odd},\, j \in \{1, \ldots, \ell_{i} \}$ & $\delta_\textrm{rel}(x_1 \boxtimes \lambda_{j}^{\theta_i})+1$ & $y_1 \boxtimes \kappa_{j}^{\theta_i}, i \in \{3, \ldots, k\} \mathrm{\, odd},\, j \in \{1, \ldots, \ell_{i-1}\}$ & $\delta_\textrm{rel}(x_1 \boxtimes \kappa_{j}^{\theta_i}) + 1$ \\ \hline
    $x_3 \boxtimes \lambda_{j}^{\theta_i}, i \in \{1, \ldots, k-2\} \mathrm{\, odd},\, j \in \{1, \ldots, \ell_{i} -1\}$ & $M(\tilde{\theta}_i) -A(\tilde{\theta}_i) +j + n +1$ & $x_3 \boxtimes \kappa_{j}^{\theta_i}, i \in \{3, \ldots, k\} \mathrm{\, odd},\, j \in \{2, \ldots, \ell_{i-1}\}$ & $M(\tilde{\theta}_i) + A(\tilde{\theta}_i) - j + 2$ \\ \hline
    $y_3 \boxtimes \lambda_{j}^{\theta_i}, i \in \{1, \ldots, k-2\} \mathrm{\, odd},\, j \in \{1, \ldots, \ell_{i} \}$ & $\delta_\textrm{rel}(x_3 \boxtimes \lambda_{j}^{\theta_i})$ & $y_3 \boxtimes \kappa_{j}^{\theta_i}, i \in \{3, \ldots, k\} \mathrm{\, odd},\, j \in \{2, \ldots, \ell_{i-1}\}$ & $\delta_\textrm{rel}(x_3 \boxtimes \kappa_{j}^{\theta_i})$ \\ \hline
    $x_5 \boxtimes \lambda_{j}^{\theta_i}, i \in \{1, \ldots, k-2\} \mathrm{\, odd},\, j \in \{1, \ldots, \ell_{i} \}$ & $\delta_\textrm{rel}(y_1 \boxtimes \lambda_{j}^{\theta_i})$ & $x_5 \boxtimes \kappa_{j}^{\theta_i}, i \in \{3, \ldots, k\} \mathrm{\, odd},\, j \in \{1, \ldots, \ell_{i-1}\}$ & $\delta_\textrm{rel}(y_1 \boxtimes \kappa_{j}^{\theta_i})$ \\ \hline
    $y_5 \boxtimes \lambda_{j}^{\theta_i}, i \in \{1, \ldots, k-2\} \mathrm{\, odd},\, j \in \{1, \ldots, \ell_{i} \}$ & $\delta_\textrm{rel}(x_5 \boxtimes \lambda_{j}^{\theta_i})$ & $y_5 \boxtimes \kappa_{j}^{\theta_i}, i \in \{3, \ldots, k\} \mathrm{\, odd},\, j \in \{1, \ldots, \ell_{i-1}\}$ & $\delta_\textrm{rel}(x_5 \boxtimes \kappa_{j}^{\theta_i})$ \\ \hline
    $x_6 \boxtimes \lambda_{j}^{\theta_i}, i \in \{1, \ldots, k-2\} \mathrm{\, odd},\, j \in \{1, \ldots, \ell_{i} \}$ & $\delta_\textrm{rel}(x_3 \boxtimes \lambda_{j}^{\theta_i}) -1$ & $x_6 \boxtimes \kappa_{j}^{\theta_i}, i \in \{3, \ldots, k\} \mathrm{\, odd},\, j \in \{1, \ldots, \ell_{i-1}\}$ & $\delta_\textrm{rel}(x_3 \boxtimes \kappa_{j}^{\theta_i}) -1$ \\ \hline
    $y_6 \boxtimes \lambda_{j}^{\theta_i}, i \in \{1, \ldots, k-2\} \mathrm{\, odd},\, j \in \{1, \ldots, \ell_{i} \}$ & $\delta_\textrm{rel}(x_6 \boxtimes \lambda_{j}^{\theta_i})$ & $y_6 \boxtimes \kappa_{j}^{\theta_i}, i \in \{3, \ldots, k\} \mathrm{\, odd},\, j \in \{1, \ldots, \ell_{i-1}\}$ & $\delta_\textrm{rel}(x_6 \boxtimes \kappa_{j}^{\theta_i})$ \\ \hline

    $x_2 \boxtimes \omega_i, i \in \{1, \ldots, k\} \mathrm{\, even}$ & $M(\tilde{\omega}_i) -A(\tilde{\omega}_i) + n+1$ & $x_2 \boxtimes \theta_i, i \in \{1, \ldots, k\} \mathrm{\, even}$ & $M(\tilde{\theta}_i) -A(\tilde{\theta}_i) + n+1$ \\ \hline
    $x_4 \boxtimes \omega_i, i \in \{1, \ldots, k\} \mathrm{\, even}$ & $M(\tilde{\omega}_i) +A(\tilde{\omega}_i) + 2$ & $x_4 \boxtimes \theta_i, i \in \{1, \ldots, k\} \mathrm{\, even}$ & $M(\tilde{\theta}_i) + A(\tilde{\theta}_i) + 2$ \\ \hline
    $y_4 \boxtimes \omega_i, i \in \{1, \ldots, k\} \mathrm{\, even}$ & $\delta_\textrm{rel}(x_4 \boxtimes \omega_i)$  & $y_4 \boxtimes \theta_i, i \in \{1, \ldots, k\} \mathrm{\, even}$ & $\delta_\textrm{rel}(x_4 \boxtimes \theta_i)$ \\ \hline
    $x_0 \boxtimes \omega_i, i \in \{1, \ldots, k\} \mathrm{\, odd}$ & $M(\tilde{\omega}_i) -3A(\tilde{\omega}_i) $ & $x_0 \boxtimes \theta_i, i \in \{1, \ldots, k\} \mathrm{\, odd}$ & $M(\tilde{\theta}_i) - 3A(\tilde{\theta}_i)$ \\ \hline
    $y_2 \boxtimes \omega_i, i \in \{1, \ldots, k\} \mathrm{\, odd}$ & $M(\tilde{\omega}_i) -A(\tilde{\omega}_i) + n+1$ & $y_2 \boxtimes \theta_i, i \in \{1, \ldots, k\} \mathrm{\, odd}$ & $M(\tilde{\theta}_i) - A(\tilde{\theta}_i) + n +1$ \\ \hline
    
    $x_2 \boxtimes \omega_{k+1}$ & $M(\tilde{\omega}_{k+1}) -A(\tilde{\omega}_{k+1}) + n+1$  & $x_2 \boxtimes \eta_{0}$ & $n-g+1$ \\ \hline
    $x_4 \boxtimes \omega_{k+1}$ & $M(\tilde{\omega}_{k+1}) +A(\tilde{\omega}_{k+1}) + 2$ & $y_2 \boxtimes \eta_{0}$ & $n-g+1$ \\ \hline
    $y_4 \boxtimes \omega_{k+1}$ & $\delta_\textrm{rel}(x_4 \boxtimes \omega_{k+1})$ & $x_4 \boxtimes \eta_{0}$ & $-3g+2$ \\ \hline
    $x_0 \boxtimes \eta_{0}$ & $g$ & $y_4 \boxtimes \eta_{0}$ & $-3g+2$ \\ \hline

    $x_1 \boxtimes \mu_{i}, i \in \{2, \ldots, 2g-n \}$ & $n - g + i -1$ & $y_1 \boxtimes \mu_{i}, i \in \{2, \ldots, 2g-n \}$ & $n - g + i$ \\ \hline
    $x_3 \boxtimes \mu_{i}, i \in \{2, \ldots, 2g-n \}$ & $g - i + 2$ & $y_3 \boxtimes \mu_{i}, i \in \{2, \ldots, 2g-n \}$ & $g - i + 2$ \\ \hline
    $x_5 \boxtimes \mu_{i}, i \in \{1, \ldots, 2g-n \}$ & $n - g + i $ & $y_5 \boxtimes \mu_{i}, i \in \{1, \ldots, 2g-n \}$ & $n - g + i $ \\ \hline
    $x_6 \boxtimes \mu_{i}, i \in \{1, \ldots, 2g-n \}$ & $g - i + 1$ & $y_6 \boxtimes \mu_{i}, i \in \{1, \ldots, 2g-n \}$ & $g - i + 1$ \\ \hline
   
  \end{tabular}
  }
\end{center}
\caption {The $\delta$-gradings of the generators of $\hfkhat(Q_n(K))$, for the casewhen $K$ admits a positive L-space surgery, $k$ is odd, and $n<2g$.}
\label{tab:generalcasenless2g} 
\end{table}

\item[-] When $n > 2g$, the unstable chain in $\cfdhat(X_{K,n})$ takes the form
\begin{equation*}
\xi_0 \xrightarrow{\textrm{$D_{123}$}} \mu_1 \xrightarrow{\textrm{$D_{23}$}} \cdots \xrightarrow{\textrm{$D_{23}$}} \mu_{n-2g} \xrightarrow{\textrm{$D_{2}$}} \eta_0 
\end{equation*}
As in the $n = 2g$ case, we compute the generators of $\hfkhat(Q_n(K))$, together with their relative $\delta$-gradings; see Table \ref{tab:generalcasengreater2g}. One can verify that
\[ 
\min \, \{ \delta_{\mathrm{rel}}(u \boxtimes v) \mid u \boxtimes v \textrm{\, is a nonzero generator of \,} \hfkhat(Q_n(K)) \} = -3g+2
\]
and
\[ 
\max \, \{ \delta_{\mathrm{rel}}(u \boxtimes v) \mid u \boxtimes v \textrm{\, is a nonzero generator of \,} \hfkhat(Q_n(K)) \} = -r_2 + n.
\]
Then $\mathrm{th}(Q_n(K)) = -r_2 + n + 3g- 2$, and this concludes the proof of the $n>2g$ case, and the proof of Proposition \ref{prop:lspacethickness} in the case when $K$ admits a positive L-space surgery and $k$ is odd.

\begin{table}[h]
\begin{center}
\resizebox{\textwidth}{!}{
  \begin{tabular}{ | c | c || c | c | }
    \hline
    Generator & $\delta_{\mathrm{rel}}$ & Generator & $\delta_{\mathrm{rel}}$\\ 
    \hline
    \hline
    
    $y_1 \boxtimes \lambda_{1}^{\omega_1}$ & $-3g+2$ & $y_1 \boxtimes \kappa_{1}^{\theta_1}$ & $n - g + 1$ \\ \hline
    $y_3 \boxtimes \lambda_{1}^{\omega_1}$ & $n-g+2$ & $x_5 \boxtimes \kappa_{1}^{\theta_1}$ & $n - g + 1$ \\ \hline
    $x_5 \boxtimes \lambda_{1}^{\omega_1}$ & $-3g+2$ & $y_5 \boxtimes \kappa_{1}^{\theta_1}$ & $n - g + 1$ \\ \hline
    $y_5 \boxtimes \lambda_{1}^{\omega_1}$ & $-3g+2$ & $x_6 \boxtimes \kappa_{1}^{\theta_1}$ & $-3g+2$ \\ \hline
    $x_6 \boxtimes \lambda_{1}^{\omega_1}$ & $n - g + 1$ & $y_6 \boxtimes \kappa_{1}^{\theta_1}$ & $-3g+2$ \\ \hline
    $y_6 \boxtimes \lambda_{1}^{\omega_1}$ & $n - g + 1$  & $x_1 \boxtimes \kappa_{j}^{\omega_i}, i \in \{1, \ldots, k-2\} \mathrm{\, odd},\, j \in \{2, \ldots, \ell_{i}\}$ & $M(\tilde{\omega}_i) - A(\tilde{\omega}_i) + j + n -1$ \\ \hline
    
    $x_1 \boxtimes \lambda_{j}^{\omega_i}, i \in \{3, \ldots, k\} \mathrm{\, odd}, \, j \in \{1, \ldots, \ell_{i-1} -1\}$ & $M(\tilde{\omega}_i) -3A(\tilde{\omega}_i) -j$ & $y_1 \boxtimes \kappa_{j}^{\omega_i}, i \in \{1, \ldots, k-2\} \mathrm{\, odd}, \, j \in \{1, \ldots, \ell_{i}\}$ & $\delta_\textrm{rel}(x_1 \boxtimes \kappa_{j}^{\omega_i}) +1$ \\ \hline
    $y_1 \boxtimes \lambda_{j}^{\omega_i}, i \in \{3, \ldots, k\} \mathrm{\, odd}, \, j \in \{1, \ldots, \ell_{i-1} \}$ & $\delta_\textrm{rel}(x_1 \boxtimes \lambda_{j}^{\omega_i})+1$ & $x_3 \boxtimes \kappa_{j}^{\omega_i}, i \in \{1, \ldots, k-2\} \mathrm{\, odd}, \, j \in \{2, \ldots, \ell_{i}\}$ & $M(\tilde{\omega}_i) + A(\tilde{\omega}_i) - j + 2$ \\ \hline
    $x_3 \boxtimes \lambda_{j}^{\omega_i}, i \in \{3, \ldots, k\} \mathrm{\, odd}, \, j \in \{1, \ldots, \ell_{i-1} -1\}$ & $M(\tilde{\omega}_i) -A(\tilde{\omega}_i) +j + n+1$ & $y_3 \boxtimes \kappa_{j}^{\omega_i}, i \in \{1, \ldots, k-2\} \mathrm{\, odd}, \, j \in \{2, \ldots, \ell_{i}\}$ & $\delta_\textrm{rel}(x_3 \boxtimes \kappa_{j}^{\omega_i})$ \\ \hline
    $y_3 \boxtimes \lambda_{j}^{\omega_i}, i \in \{3, \ldots, k\} \mathrm{\, odd}, \, j \in \{1, \ldots, \ell_{i-1} \}$ & $\delta_\textrm{rel}(x_3 \boxtimes \lambda_{j}^{\omega_i})$ & $x_5 \boxtimes \kappa_{j}^{\omega_i}, i \in \{1, \ldots, k-2\} \mathrm{\, odd}, \, j \in \{1, \ldots, \ell_{i}\}$ & $\delta_\textrm{rel}(y_1 \boxtimes \kappa_{j}^{\omega_i}) $ \\ \hline
    $x_5 \boxtimes \lambda_{j}^{\omega_i}, i \in \{3, \ldots, k\} \mathrm{\, odd}, \, j \in \{1, \ldots, \ell_{i-1} \}$ & $\delta_\textrm{rel}(y_1 \boxtimes \lambda_{j}^{\omega_i})$ & $y_5 \boxtimes \kappa_{j}^{\omega_i}, i \in \{1, \ldots, k-2\} \mathrm{\, odd}, \, j \in \{1, \ldots, \ell_{i}\}$ & $\delta_\textrm{rel}(x_5 \boxtimes \kappa_{j}^{\omega_i}) $ \\ \hline
    $y_5 \boxtimes \lambda_{j}^{\omega_i}, i \in \{3, \ldots, k\} \mathrm{\, odd}, \, j \in \{1, \ldots, \ell_{i-1} \}$ & $\delta_\textrm{rel}(x_5 \boxtimes \lambda_{j}^{\omega_i})$ & $x_6 \boxtimes \kappa_{j}^{\omega_i}, i \in \{1, \ldots, k-2\} \mathrm{\, odd}, \, j \in \{1, \ldots, \ell_{i}\}$ & $\delta_\textrm{rel}(x_3 \boxtimes \kappa_{j}^{\omega_i}) -1$ \\ \hline
    $x_6 \boxtimes \lambda_{j}^{\omega_i}, i \in \{3, \ldots, k\} \mathrm{\, odd}, \, j \in \{1, \ldots, \ell_{i-1} \}$ & $\delta_\textrm{rel}(x_3 \boxtimes \lambda_{j}^{\omega_i})-1$ & $y_6 \boxtimes \kappa_{j}^{\omega_i}, i \in \{1, \ldots, k-2\} \mathrm{\, odd}, \, j \in \{1, \ldots, \ell_{i}\}$ & $\delta_\textrm{rel}(x_6 \boxtimes \kappa_{j}^{\omega_i})$ \\ \hline
    $y_6 \boxtimes \lambda_{j}^{\omega_i}, i \in \{3, \ldots, k\} \mathrm{\, odd}, \, j \in \{1, \ldots, \ell_{i-1} \}$ & $\delta_\textrm{rel}(x_6 \boxtimes \lambda_{j}^{\omega_i})$ & $x_1 \boxtimes \kappa_{j}^{\omega_k}, j \in \{2, \ldots, r_k \}$ & $M(\tilde{\omega}_k) - A(\tilde{\omega}_k) + j + n -1$ \\ \hline
    
    $x_1 \boxtimes \lambda_{j}^{\theta_k}, j \in \{1, \ldots, r_k -1\}$ & $M(\tilde{\theta}_k) -3A(\tilde{\theta}_k) -j$ & $y_1 \boxtimes \kappa_{j}^{\omega_k}, j \in \{1, \ldots, r_k \}$ & $\delta_\textrm{rel}(x_1 \boxtimes \kappa_{j}^{\omega_k})+1$ \\ \hline
    $y_1 \boxtimes \lambda_{j}^{\theta_k}, j \in \{1, \ldots, r_k \}$ & $\delta_\textrm{rel}(x_1 \boxtimes \lambda_{j}^{\theta_k})+1$ & $x_3 \boxtimes \kappa_{j}^{\omega_k}, j \in \{2, \ldots, r_k \}$ & $M(\tilde{\omega}_k) + A(\tilde{\omega}_k) - j + 2$ \\ \hline
    $x_3 \boxtimes \lambda_{j}^{\theta_k}, j \in \{1, \ldots, r_k -1\}$ & $M(\tilde{\theta}_k) - A(\tilde{\theta}_k) + j +n +1$ & $y_3 \boxtimes \kappa_{j}^{\omega_k}, j \in \{2, \ldots, r_k \}$ & $\delta_\textrm{rel}(x_3 \boxtimes \kappa_{j}^{\omega_k})$ \\ \hline
    $y_3 \boxtimes \lambda_{j}^{\theta_k}, j \in \{1, \ldots, r_k \}$ & $\delta_\textrm{rel}(x_3 \boxtimes \lambda_{j}^{\theta_k})$ & $x_5 \boxtimes \kappa_{j}^{\omega_k}, j \in \{1, \ldots, r_k \}$ & $\delta_\textrm{rel}(y_1 \boxtimes \kappa_{j}^{\omega_k})$ \\ \hline
    $x_5 \boxtimes \lambda_{j}^{\theta_k}, j \in \{1, \ldots, r_k \}$ & $\delta_\textrm{rel}(y_1 \boxtimes \lambda_{j}^{\theta_k})$ & $y_5 \boxtimes \kappa_{j}^{\omega_k}, j \in \{1, \ldots, r_k \}$ & $\delta_\textrm{rel}(x_5 \boxtimes \kappa_{j}^{\omega_k})$ \\ \hline
    $y_5 \boxtimes \lambda_{j}^{\theta_k}, j \in \{1, \ldots, r_k \}$ & $\delta_\textrm{rel}(x_5 \boxtimes \lambda_{j}^{\theta_k})$ & $x_6 \boxtimes \kappa_{j}^{\omega_k}, j \in \{1, \ldots, r_k \}$ & $\delta_\textrm{rel}(x_3 \boxtimes \kappa_{j}^{\omega_k})-1$ \\ \hline
    $x_6 \boxtimes \lambda_{j}^{\theta_k}, j \in \{1, \ldots, r_k \}$ & $\delta_\textrm{rel}(x_3 \boxtimes \lambda_{j}^{\theta_k})-1$ & $y_6 \boxtimes \kappa_{j}^{\omega_k}, j \in \{1, \ldots, r_k \}$ & $\delta_\textrm{rel}(x_6 \boxtimes \kappa_{j}^{\omega_k})$ \\ \hline
    $y_6 \boxtimes \lambda_{j}^{\theta_k}, j \in \{1, \ldots, r_k \}$ & $\delta_\textrm{rel}(x_6 \boxtimes \lambda_{j}^{\theta_k})$ & $x_1 \boxtimes \kappa_{j}^{\theta_i}, i \in \{3, \ldots, k\} \mathrm{\, odd},\, j \in \{2, \ldots, \ell_{i-1}\}$ & $M(\tilde{\theta}_i) -A(\tilde{\theta}_i) + j +n -1$ \\ \hline
    
    $x_1 \boxtimes \lambda_{j}^{\theta_i}, i \in \{1, \ldots, k-2\} \mathrm{\, odd},\, j \in \{1, \ldots, \ell_{i} -1\}$ & $M(\tilde{\theta}_i) -3A(\tilde{\theta}_i) -j$ & $y_1 \boxtimes \kappa_{j}^{\theta_i}, i \in \{3, \ldots, k\} \mathrm{\, odd},\, j \in \{1, \ldots, \ell_{i-1}\}$ & $\delta_\textrm{rel}(x_1 \boxtimes \kappa_{j}^{\theta_i}) + 1$ \\ \hline
    $y_1 \boxtimes \lambda_{j}^{\theta_i}, i \in \{1, \ldots, k-2\} \mathrm{\, odd},\, j \in \{1, \ldots, \ell_{i} \}$ & $\delta_\textrm{rel}(x_1 \boxtimes \lambda_{j}^{\theta_i})+1$ & $x_3 \boxtimes \kappa_{j}^{\theta_i}, i \in \{3, \ldots, k\} \mathrm{\, odd},\, j \in \{2, \ldots, \ell_{i-1}\}$ & $M(\tilde{\theta}_i) + A(\tilde{\theta}_i) - j + 2$ \\ \hline
    $x_3 \boxtimes \lambda_{j}^{\theta_i}, i \in \{1, \ldots, k-2\} \mathrm{\, odd},\, j \in \{1, \ldots, \ell_{i} -1\}$ & $M(\tilde{\theta}_i) -A(\tilde{\theta}_i) +j + n +1$ & $y_3 \boxtimes \kappa_{j}^{\theta_i}, i \in \{3, \ldots, k\} \mathrm{\, odd},\, j \in \{2, \ldots, \ell_{i-1}\}$ & $\delta_\textrm{rel}(x_3 \boxtimes \kappa_{j}^{\theta_i})$ \\ \hline
    $y_3 \boxtimes \lambda_{j}^{\theta_i}, i \in \{1, \ldots, k-2\} \mathrm{\, odd},\, j \in \{1, \ldots, \ell_{i} \}$ & $\delta_\textrm{rel}(x_3 \boxtimes \lambda_{j}^{\theta_i})$ & $x_5 \boxtimes \kappa_{j}^{\theta_i}, i \in \{3, \ldots, k\} \mathrm{\, odd},\, j \in \{1, \ldots, \ell_{i-1}\}$ & $\delta_\textrm{rel}(y_1 \boxtimes \kappa_{j}^{\theta_i})$ \\ \hline
    $x_5 \boxtimes \lambda_{j}^{\theta_i}, i \in \{1, \ldots, k-2\} \mathrm{\, odd},\, j \in \{1, \ldots, \ell_{i} \}$ & $\delta_\textrm{rel}(y_1 \boxtimes \lambda_{j}^{\theta_i})$ & $y_5 \boxtimes \kappa_{j}^{\theta_i}, i \in \{3, \ldots, k\} \mathrm{\, odd},\, j \in \{1, \ldots, \ell_{i-1}\}$ & $\delta_\textrm{rel}(x_5 \boxtimes \kappa_{j}^{\theta_i})$ \\ \hline
    $y_5 \boxtimes \lambda_{j}^{\theta_i}, i \in \{1, \ldots, k-2\} \mathrm{\, odd},\, j \in \{1, \ldots, \ell_{i} \}$ & $\delta_\textrm{rel}(x_5 \boxtimes \lambda_{j}^{\theta_i})$ & $x_6 \boxtimes \kappa_{j}^{\theta_i}, i \in \{3, \ldots, k\} \mathrm{\, odd},\, j \in \{1, \ldots, \ell_{i-1}\}$ & $\delta_\textrm{rel}(x_3 \boxtimes \kappa_{j}^{\theta_i}) -1$ \\ \hline
    $x_6 \boxtimes \lambda_{j}^{\theta_i}, i \in \{1, \ldots, k-2\} \mathrm{\, odd},\, j \in \{1, \ldots, \ell_{i} \}$ & $\delta_\textrm{rel}(x_3 \boxtimes \lambda_{j}^{\theta_i}) -1$ & $y_6 \boxtimes \kappa_{j}^{\theta_i}, i \in \{3, \ldots, k\} \mathrm{\, odd},\, j \in \{1, \ldots, \ell_{i-1}\}$ & $\delta_\textrm{rel}(x_6 \boxtimes \kappa_{j}^{\theta_i})$ \\ \hline
    $y_6 \boxtimes \lambda_{j}^{\theta_i}, i \in \{1, \ldots, k-2\} \mathrm{\, odd},\, j \in \{1, \ldots, \ell_{i} \}$ & $\delta_\textrm{rel}(x_6 \boxtimes \lambda_{j}^{\theta_i})$ & $x_2 \boxtimes \theta_i, i \in \{1, \ldots, k\} \mathrm{\, even}$ & $M(\tilde{\theta}_i) -A(\tilde{\theta}_i) + n+1$ \\ \hline

    $x_2 \boxtimes \omega_i, i \in \{1, \ldots, k\} \mathrm{\, even}$ & $M(\tilde{\omega}_i) -A(\tilde{\omega}_i) + n+1$ & $x_4 \boxtimes \theta_i, i \in \{1, \ldots, k\} \mathrm{\, even}$ & $M(\tilde{\theta}_i) + A(\tilde{\theta}_i) + 2$ \\ \hline
    $x_4 \boxtimes \omega_i, i \in \{1, \ldots, k\} \mathrm{\, even}$ & $M(\tilde{\omega}_i) +A(\tilde{\omega}_i) + 2$ & $y_4 \boxtimes \theta_i, i \in \{1, \ldots, k\} \mathrm{\, even}$ & $\delta_\textrm{rel}(x_4 \boxtimes \theta_i)$ \\ \hline
    $y_4 \boxtimes \omega_i, i \in \{1, \ldots, k\} \mathrm{\, even}$ & $\delta_\textrm{rel}(x_4 \boxtimes \omega_i)$  & $x_0 \boxtimes \theta_i, i \in \{1, \ldots, k\} \mathrm{\, odd}$ & $M(\tilde{\theta}_i) - 3A(\tilde{\theta}_i)$ \\ \hline
    $x_0 \boxtimes \omega_i, i \in \{1, \ldots, k\} \mathrm{\, odd}$ & $M(\tilde{\omega}_i) -3A(\tilde{\omega}_i) $ & $y_2 \boxtimes \theta_i, i \in \{1, \ldots, k\} \mathrm{\, odd}$ & $M(\tilde{\theta}_i) - A(\tilde{\theta}_i) + n +1$ \\ \hline
    $y_2 \boxtimes \omega_i, i \in \{1, \ldots, k\} \mathrm{\, odd}$ & $M(\tilde{\omega}_i) -A(\tilde{\omega}_i) + n+1$ & $x_2 \boxtimes \eta_{0}$ & $n-g+1$ \\ \hline
    
    $x_2 \boxtimes \omega_{k+1}$ & $M(\tilde{\omega}_{k+1}) -A(\tilde{\omega}_{k+1}) + n+1$  & $x_4 \boxtimes \eta_{0}$ & $-3g+2$ \\ \hline
    $x_4 \boxtimes \omega_{k+1}$ & $M(\tilde{\omega}_{k+1}) +A(\tilde{\omega}_{k+1}) + 2$ & $y_4 \boxtimes \eta_{0}$ & $-3g+2$ \\ \hline
    $y_4 \boxtimes \omega_{k+1}$ & $\delta_\textrm{rel}(x_4 \boxtimes \omega_{k+1})$ & $x_4 \boxtimes \xi_{0}$ & $g+2$ \\ \hline
    $x_2 \boxtimes \xi_{0}$ & $n-g+1$ & $y_4 \boxtimes \xi_{0}$ & $g+2$ \\ \hline
     
    $x_1 \boxtimes \mu_{i}, i \in \{1, \ldots, n-2g-1 \}$ & $n - g - i +1$ & $y_1 \boxtimes \mu_{i}, i \in \{1, \ldots, n-2g \}$ & $n - g - i + 2$ \\ \hline
    $x_3 \boxtimes \mu_{i}, i \in \{1, \ldots, n-2g-1 \}$ & $g + i + 2$ & $y_3 \boxtimes \mu_{i}, i \in \{1, \ldots, n-2g \}$ & $g + i + 2$ \\ \hline
    $x_5 \boxtimes \mu_{i}, i \in \{1, \ldots, n-2g \}$ & $n - g - i +2 $ & $y_5 \boxtimes \mu_{i}, i \in \{1, \ldots, n-2g \}$ & $n - g - i +2 $ \\ \hline
    $x_6 \boxtimes \mu_{i}, i \in \{1, \ldots, n-2g \}$ & $g + i + 1$ & $y_6 \boxtimes \mu_{i}, i \in \{1, \ldots, n-2g \}$ & $g + i + 1$ \\ \hline
   
  \end{tabular}
  }
\end{center}
\caption {The $\delta$-gradings of the generators of $\hfkhat(Q_n(K))$, for the case when $K$ admits a positive L-space surgery, $k$ is odd, and $n>2g$.}
\label{tab:generalcasengreater2g} 
\end{table}

\end{itemize}

When $K$ admits a positive L-space surgery and $k$ is even, $\cfkm(K)$ is given by the right staircase in Figure \ref{fig:cfkmlspace}. The Alexander and Maslov gradings of the basis elements $\widetilde{\xi}_0$, $\widetilde{\omega}_i$, $\widetilde{\theta}_i$, and $\widetilde{\eta}_0$ are given by Table \ref{tab:amlspacekeven}. Note that with the exception of the Maslov grading of $\widetilde{\omega}_{k+1}$, these gradings agree with the gradings in Table \ref{tab:amlspace} for the basis elements $\widetilde{\xi}_0$, $\widetilde{\omega}_i$, $\widetilde{\theta}_i$, and $\widetilde{\eta}_0$ in the $k$ odd case. One can verify that the thickness values of $Q_n(K)$ in this case are given by the formula in the $k$ odd case. This concludes the proof of Proposition  \ref{prop:lspacethickness} in the case when $K$ admits a positive L-space surgery.

\begin{table}[h]
\begin{center}
  \begin{tabular}{ | c | c | c | }
    \hline
    Basis element & $A$ & $M$  \\
    \hline
    \hline
    $\widetilde{\xi}_0$ & $r_0=g$ & 0  \\ \hline
    $\widetilde{\omega}_i, i \in \{1, \ldots, k\} \, \mathrm{odd}$ & $r_i$ & $-1 - 2\sum\limits_{j=1}^{\frac{i-1}{2}} \ell_{2j}$  \\ \hline
    $\widetilde{\omega}_i, i \in \{1, \ldots, k\} \, \mathrm{even}$ & $r_i$ & $-2 - 2\sum\limits_{j=1}^{\frac{i-2}{2}} \ell_{2j} $ \\ \hline
    $\widetilde{\omega}_{k+1}$ & $0$ & $-1 - 2(\sum\limits_{j=1}^{\frac{k-2}{2}} \ell_{2j} + r_k) $ \\ \hline
    $\widetilde{\theta}_i, i \in \{1, \ldots, k\} \, \mathrm{odd}$ & $-r_i$ & $-2g + 1 + 2\sum\limits_{j=1}^{\frac{i-1}{2}} \ell_{2j-1}$ \\ \hline
    $\widetilde{\theta}_i, i \in \{1, \ldots, k\} \, \mathrm{even}$ & $-r_i$ &  $-2g + 2\sum\limits_{j=1}^{\frac{i}{2}} \ell_{2j-1} $  \\ \hline
    $\widetilde{\eta}_0$ & $-r_0=-g$ & $-2g$  \\ \hline
     \end{tabular}
\end{center}
\caption {The Alexander and Maslov gradings of the basis elements $\widetilde{\xi}_0$, $\widetilde{\omega}_i$, $\widetilde{\theta}_i$, and $\widetilde{\eta}_0$ of $\cfkm(K)$, when $K$ admits a positive L-space surgery and $k$ is even.}
\label{tab:amlspacekeven}
\end{table}

Now suppose $K$ admits a negative L-space surgery. Recall that $k \geq 1$, since by assumption $K$ is neither the unknot nor the left-handed trefoil. Then $\cfkm(K)$ is given by the staircases in Figure \ref{fig:cfkmlspacenegsurgery}. The left staircase is for the case when $k$ is odd, while the right staircase is for the case when $k$ is even (compare with Figure \ref{fig:cfkmlspace} from the positive L-space surgery case).

 \begin{figure}[h]
    \centering
    {\scalefont{0.8}
    \begin{tikzpicture}[xscale=0.62,yscale=0.62]
      \node at (0,-4) (x1) {$\widetilde{\eta}_{0}$};
      \node at (0,-6) (x2) {$\widetilde{\omega}_{1}$};
      \node at (2,-6) (x3) {$\widetilde{\omega}_{2}$};
      \node at (4,-8) (x4) {$\widetilde{\omega}_{k-1}$};
      \node at (4,-10) (x5) {$\widetilde{\omega}_{k}$};
      \node at (6,-10) (x6) {$\widetilde{\omega}_{k+1}$}; 
      \node at (6,-12) (x7) {$\widetilde{\theta}_{k}$};
      \node at (8,-12) (x8) {$\widetilde{\theta}_{k-1}$};
      \node at (10,-14) (x9) {$\widetilde{\theta}_{2}$};
      \node at (10,-16) (x10) {$\widetilde{\theta}_{1}$};
      \node at (12,-16) (x11) {$\widetilde{\xi}_{0}$};
      \draw[algarrow] (x1) to node[left] {$1$} (x2);
      \draw[algarrow] (x3) to node[below] {$\ell_1$} (x2);
      \draw[algarrow] (x4) to node[left] {$\ell_{k-1}$} (x5);
      \draw[algarrow] (x6) to node[below] {$r_k$} (x5);
      \draw[algarrow] (x6) to node[left] {$r_k$} (x7);
      \draw[algarrow] (x8) to node[below] {$\ell_{k-1}$} (x7);
      \draw[loosely dotted, thick] (x3) to node[right] {} (x4);
      \draw[loosely dotted, thick] (x8) to node[right] {} (x9);
      \draw[algarrow] (x9) to node[left] {$\ell_{1}$} (x10);
      \draw[algarrow] (x11) to node[below] {$1$} (x10);
    \end{tikzpicture}
    \begin{tikzpicture}[xscale=0.62,yscale=0.62]
      \node at (0,-4) (x1) {$\widetilde{\eta}_{0}$};
      \node at (0,-6) (x2) {$\widetilde{\omega}_{1}$};
      \node at (2,-6) (x3) {$\widetilde{\omega}_{2}$};
      \node at (3,-9) (x4) {$\widetilde{\omega}_{k-1}$};
      \node at (5,-9) (x5) {$\widetilde{\omega}_{k}$};
      \node at (5,-11) (x6) {$\widetilde{\omega}_{k+1}$}; 
      \node at (7,-11) (x7) {$\widetilde{\theta}_{k}$};
      \node at (7,-13) (x8) {$\widetilde{\theta}_{k-1}$};
      \node at (10,-14) (x9) {$\widetilde{\theta}_{2}$};
      \node at (10,-16) (x10) {$\widetilde{\theta}_{1}$};
      \node at (12,-16) (x11) {$\widetilde{\xi}_{0}$};
      \draw[algarrow] (x1) to node[left] {$1$} (x2);
      \draw[algarrow] (x3) to node[below] {$\ell_1$} (x2);
      \draw[algarrow] (x5) to node[below] {$\ell_{k-1}$} (x4);
      \draw[algarrow] (x5) to node[left] {$r_k$} (x6);
      \draw[algarrow] (x7) to node[below] {$r_k$} (x6);
      \draw[algarrow] (x7) to node[left] {$\ell_{k-1}$} (x8);
      \draw[loosely dotted, thick] (x3) to node[right] {} (x5);
      \draw[loosely dotted, thick] (x7) to node[right] {} (x9);
      \draw[algarrow] (x9) to node[left] {$\ell_{1}$} (x10);
      \draw[algarrow] (x11) to node[below] {$1$} (x10);
    \end{tikzpicture}	
    }	
    \caption{$\cfkm(K)$ for L-space knots $K$ that admit negative L-space surgeries. The left staircase is for $k$ odd, while the right staircase is for $k$ even.}
    \label{fig:cfkmlspacenegsurgery}
  \end{figure}
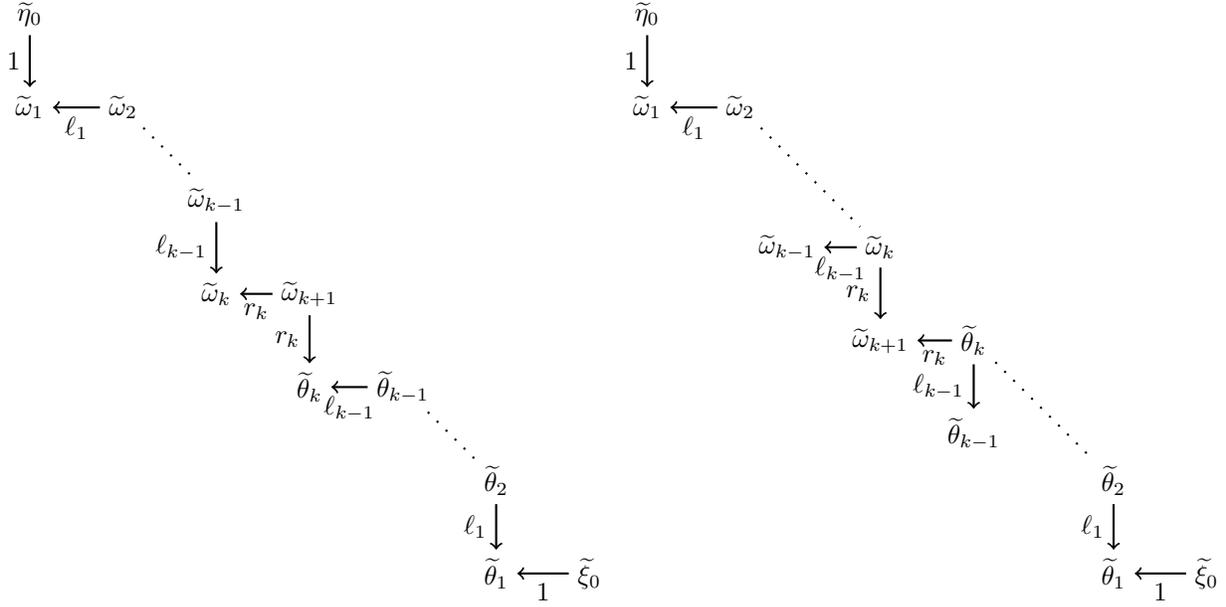

  The Alexander and Maslov gradings of the basis elements $\widetilde{\xi}_0$, $\widetilde{\omega}_i$, $\widetilde{\theta}_i$, and $\widetilde{\eta}_0$ are given by Tables \ref{tab:amlspacenegsurgery} and \ref{tab:amlspacekevennegsurgery}. Table \ref{tab:amlspacenegsurgery} is for the case when $k$ is odd, while Table \ref{tab:amlspacekevennegsurgery} is for the case when $k$ is even (compare with Tables \ref{tab:amlspace}  and \ref{tab:amlspacekeven} from the positive L-space surgery case).

\begin{table}[h]
\begin{center}
  \begin{tabular}{ | c | c | c | }
    \hline
    Basis element & $A$ & $M$  \\
    \hline
    \hline
    $\widetilde{\eta}_0$ & $r_0=g$ & $2g$  \\ \hline
    $\widetilde{\omega}_i, i \in \{1, \ldots, k\} \, \mathrm{odd}$ & $r_i$ & $2g-1 - 2\sum\limits_{j=1}^{\frac{i-1}{2}} \ell_{2j-1}$  \\ \hline
    $\widetilde{\omega}_i, i \in \{1, \ldots, k\} \, \mathrm{even}$ & $r_i$ & $2g - 2\sum\limits_{j=1}^{\frac{i}{2}} \ell_{2j-1} $ \\ \hline
    $\widetilde{\omega}_{k+1}$ & $0$ & $2g - 2(\sum\limits_{j=1}^{\frac{k-1}{2}} \ell_{2j-1} + r_k) $ \\ \hline
    $\widetilde{\theta}_i, i \in \{1, \ldots, k\} \, \mathrm{odd}$ & $-r_i$ & $1 + 2\sum\limits_{j=1}^{\frac{i-1}{2}} \ell_{2j}$ \\ \hline
    $\widetilde{\theta}_i, i \in \{1, \ldots, k\} \, \mathrm{even}$ & $-r_i$ &  $2+ 2\sum\limits_{j=1}^{\frac{i-2}{2}} \ell_{2j} $  \\ \hline
    $\widetilde{\xi}_0$ & $-r_0=-g$ & $0$  \\ \hline
     \end{tabular}
\end{center}
\caption {The Alexander and Maslov gradings of the basis elements $\widetilde{\xi}_0$, $\widetilde{\omega}_i$, $\widetilde{\theta}_i$, and $\widetilde{\eta}_0$ of $\cfkm(K)$, when $K$ admits a negative L-space surgery and $k$ is odd.}
\label{tab:amlspacenegsurgery} 
\end{table}

\begin{table}[h]
\begin{center}
  \begin{tabular}{ | c | c | c | }
    \hline
    Basis element & $A$ & $M$  \\
    \hline
    \hline
    $\widetilde{\eta}_0$ & $r_0=g$ & $2g$  \\ \hline
    $\widetilde{\omega}_i, i \in \{1, \ldots, k\} \, \mathrm{odd}$ & $r_i$ & $2g-1 - 2\sum\limits_{j=1}^{\frac{i-1}{2}} \ell_{2j-1}$  \\ \hline
    $\widetilde{\omega}_i, i \in \{1, \ldots, k\} \, \mathrm{even}$ & $r_i$ & $2g - 2\sum\limits_{j=1}^{\frac{i}{2}} \ell_{2j-1} $ \\ \hline
    $\widetilde{\omega}_{k+1}$ & $0$ & $2g -1 - 2\sum\limits_{j=1}^{\frac{k}{2}} \ell_{2j-1} $ \\ \hline
    $\widetilde{\theta}_i, i \in \{1, \ldots, k\} \, \mathrm{odd}$ & $-r_i$ & $1 + 2\sum\limits_{j=1}^{\frac{i-1}{2}} \ell_{2j}$ \\ \hline
    $\widetilde{\theta}_i, i \in \{1, \ldots, k\} \, \mathrm{even}$ & $-r_i$ &  $2+ 2\sum\limits_{j=1}^{\frac{i-2}{2}} \ell_{2j} $  \\ \hline
    $\widetilde{\xi}_0$ & $-r_0=-g$ & $0$  \\ \hline
     \end{tabular}
\end{center}
\caption {The Alexander and Maslov gradings of the basis elements $\widetilde{\xi}_0$, $\widetilde{\omega}_i$, $\widetilde{\theta}_i$, and $\widetilde{\eta}_0$ of $\cfkm(K)$, when $K$ admits a negative L-space surgery and $k$ is even.}
\label{tab:amlspacekevennegsurgery}
\end{table}

\begin{figure}[h]
    \centering
    {\scalefont{0.8}
    \begin{tikzpicture}[xscale=0.62,yscale=0.62]
      \node at (0,-4) (x1) {$\eta_{0}$};
      \node at (0,-6) (x2) {$\kappa_{1}^{\eta_0}$};
      \node at (0,-8) (x3) {$\omega_{1}$};
      \node at (2,-8) (x4) {$\lambda_{\ell_1}^{\omega_2}$};
      \node at (4,-8) (x7) {$\lambda_{1}^{\omega_2}$};
      \node at (6,-8) (x8) {$\omega_{2}$};
      \node at (8,-10) (x9) {$\theta_{2}$};
      \node at (8,-12) (x10) {$\kappa_{1}^{\theta_2}$};
      \node at (8,-14) (x13) {$\kappa_{\ell_1}^{\theta_2}$};
      \node at (8,-16) (x14) {$\theta_{1}$};
      \node at (10,-16) (x15) {$\lambda_{1}^{\xi_0}$};
      \node at (12,-16) (x16) {$\xi_{0}$};
      \draw[algarrow] (x1) to node[left] {$\rho_1$} (x2);
      \draw[algarrow] (x3) to node[left] {$\rho_{123}$} (x2);
      \draw[algarrow] (x4) to node[below] {$\rho_2$} (x3);
      \draw[algarrow] (x8) to node[below] {$\rho_{3}$} (x7);
      \draw[algarrow] (x9) to node[left] {$\rho_{1}$} (x10);
      \draw[algarrow] (x14) to node[left] {$\rho_{123}$} (x13);
      \draw[algarrow] (x15) to node[below] {$\rho_{2}$} (x14);
      \draw[algarrow] (x16) to node[below] {$\rho_{3}$} (x15);
      \draw[loosely dotted, thick] (x7) to node[left] {} (x4);
      \draw[loosely dotted, thick] (x8) to node[left] {} (x9);
      \draw[loosely dotted, thick] (x13) to node[right] {} (x10);
      \draw[unst1] (x16) to node[below] {} (x1);
    \end{tikzpicture}	
    }	
\caption{$\cfdhat(X_{K,n})$ for L-space knots $K$ that admit negative L-space surgeries. The dotted arrow represents the unstable chain.}
\label{fig:cfdlspacenegsurgery}
  \end{figure}
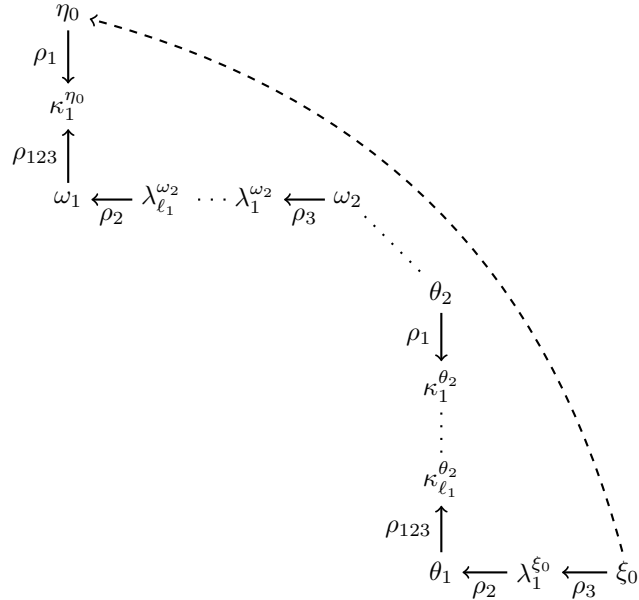

The type $D$ structure $\cfdhat(X_{K,n})$ is given by Figure \ref{fig:cfdlspacenegsurgery} (compare with Figure \ref{fig:cfdlspace} from the positive L-space surgery case). Using methods from the positive L-space surgery case, one can calculate the relative $\delta$-gradings of the generators of $H_{\ast}(\cfahat(\mathcal H)\boxtimes \cfdhat(X_{K, n})) \cong \hfkhat(Q_n(K))$, and verify that the thickness values of $Q_n(K)$ agree with the thickness values of $Q_n(K)$ from the positive L-space surgery case. This concludes the proof of Proposition \ref{prop:lspacethickness}.
\end{proof}

\subsubsection{Cosmetic Surgery Conjecture for L-space companions}

In this subsection, we prove Theorem~\ref{thm:csc} for L-space companions. Our main technical tool will be Inequality \ref{eqn:csc2}. We use the same notation as in Section \ref{sec:lspacethickness}.

\begin{proof}[Proof of Theorem~\ref{thm:csc} for L-space companions]
First suppose $K$ is neither the unknot nor a trefoil. Then $g = g(K) \geq 2$. By Theorem \ref{thm:g3}, $g(Q_n(K)) \geq 3$ for every $n$. This means that we can use Inequality~\ref{eqn:csc2} to test whether $Q_n(K)$ satisfies the cosmetic surgery conjecture. We consider several cases, depending on our values for $\mathrm{th}(Q_n(K))$ from Proposition \ref{prop:lspacethickness}.

We begin with the case $n \leq -2g$. Then $\mathrm{th}(Q_n(K)) = 2g - n -1$ and $g(Q_n(K)) = g -n$. By the argument in Case~\ref{case:12} of Section~\ref{sssec:alt-csc}, the satellites $Q_n(K)$ satisfy the cosmetic surgery conjecture.

Now suppose $n \in [-2g+1, g+r_2]$. Then $\mathrm{th}(Q_n(K)) = 4g-2$. We consider two subcases:
\begin{itemize}
\item[-] Suppose $n \in [-2g+1,-1]$. Then $g(Q_n(K)) = g-n$. As seen in Section \ref{sec:lspacethickness}, for every $n \leq -1$ and $g \geq 2$, except for $n=-1$ and $g=2$,  $f(Q_n(K)) >0$. Hence for every $n \leq -1$ and $g \geq 2$, except for $n=-1$ and $g=2$, the satellites $Q_n(K)$ satisfy the cosmetic surgery conjecture. Now we resolve the remaining case where $n=-1$ and $g=2$, using the Boyer-Lines obstruction in \cite[Proposition 5.1]{boyerlines}. First note that $\Delta_K(t) = t^{-2} - t^{-1} + 1 - t + t^2$ and $\Delta_{Q_{-1}(U)}(t) = \Delta_{5_2}(t)  = 2t^{-1} -3 + 2t$. Then $\Delta_{Q_{-1}(K)}''(1) = \Delta_{Q_{-1}(U)}''(1)  + \Delta_{K}''(1) = 4 + 6 =10$. By \cite[Proposition 5.1]{boyerlines}, $Q_{-1}(K)$ satisfies the cosmetic surgery conjecture.

\item[-] Suppose $n \in [0, g+r_2]$. Then $g(Q_n(K)) = n +g+1$. As seen in Section \ref{sec:lspacethickness}, for every $n \geq 0$ and $g \geq 2$, except for $n=0$ and $g=2$,  $f(Q_n(K)) >0$. Hence for every $n \geq 0$ and $g \geq 2$, except for $n=0$ and $g=2$, the satellites $Q_n(K)$ satisfy the cosmetic surgery conjecture. Now suppose $n=0$ and $g=2$. Then $\Delta_{Q_{0}(K)}(t) = \Delta_{K}(t) = t^{-2} - t^{-1} + 1 - t + t^2$, which implies that $\Delta_{Q_{0}(K)}''(1) =  \Delta_{K}''(1) = 6$. By \cite[Proposition 5.1]{boyerlines}, $Q_{0}(K)$ satisfies the cosmetic surgery conjecture.
\end{itemize}

Finally, we consider the case where $n \geq g+r_2+1$. Then $\mathrm{th}(Q_n(K)) = 3g-r_2+n-2$ and $g(Q_n(K)) = g +n +1$. For every $n \geq 3$ and $g \geq 2$,  $f(Q_n(K)) = 2g^2 -3g + 2n^2 -n +4gn+r_2 >0$. Thus, when $n \geq g+r_2+1$, the satellites $Q_n(K)$ also satisfy the cosmetic surgery conjecture. 

If $K$ is the unknot or a trefoil, then by Section \ref{sec:lspacethickness}, all nontrivial $Q_n(K)$ satisfy the cosmetic surgery conjecture. This concludes the proof of Theorem~\ref{thm:csc} for L-space companions.
\end{proof}


\bibliographystyle{alpha}

\bibliography{master}

\end{document}